\UseRawInputEncoding
\documentclass[11pt,reqno]{amsart}
\usepackage{amscd}
\usepackage{amsmath}
\usepackage{amsmath,amsthm,amssymb,amsfonts}
\usepackage{mathrsfs}
\usepackage{graphicx}
\usepackage{hyperref}
\usepackage{fancyhdr}
\usepackage{stmaryrd}
\usepackage[colorinlistoftodos,prependcaption,textsize=tiny]{todonotes}
\usepackage{color}
\usepackage{amsmath}
\usepackage{amsfonts}
\usepackage{amssymb}
\usepackage{geometry}
\usepackage{mathtools}
\usepackage{chngcntr}
\usepackage{amssymb}
\usepackage{pdfsync}
\usepackage{epstopdf}
\usepackage{setspace}

\newtheorem{thm}{Theorem}[section]

\newtheorem{lem}[thm]{Lemma}
\newtheorem{prop}[thm]{Proposition}
\theoremstyle{definition}

\newtheorem{rem}[thm]{Remark}
\numberwithin{equation}{section}

\newcommand{\R}{\mathbb{R}}

\newcommand{\abs}[1]{\left\vert#1\right\vert}

\allowdisplaybreaks
\begin{document}
\title[Multiple interfaces and 1D Cahn-Hilliard equation]{Multiple-interface solutions of one dimensional generalized parabolic Cahn-Hilliard equation}

\author{Chao Liu}
\address{School of Mathematics and Information Science Guangzhou University,
Guangzhou 510006, People's Republic of China}
\email{chaoliuwh@whu.edu.cn}

\author{Jun Yang}
\address{School of Mathematics and Information Science, Guangzhou University,
Guangzhou 510006, People's Republic of China}
\email{jyang2019@gzhu.edu.cn}

\email{}
\thanks{The work is supported by NSFC (No. 12171109),  Guangdong Basic and Applied Basic Research Foundation (No. 2022A1515010220)  and Basic Research Project of Guangzhou (No. 202201020094).}

\thanks{Corresponding author: Jun Yang, jyang2019@gzhu.edu.cn.}




\dedicatory{}

\commby{}


\begin{abstract}
We consider one dimensional generalized parabolic Cahn-Hilliard equation
$$
u_t=-\partial_{xx}\big[\partial_{xx}u-W'(u)\big]+W''(u)\big[\partial_{xx} u -W'(u)\big], \qquad
\forall\, (t,x)\in [0,+\infty)\times {\mathbb R},
$$
where the function $W(\cdot)$ is the standard double-well potential.
For any given positive integer $k\geq2$, we construct a solution $u(t,x)$
with $k$ interfaces, which has the form
$$
u(t,x)\approx\sum_{j=1}^k(-1)^{j+1}\omega\big(x-\gamma_j(t)\big)-\frac{1+(-1)^k}{2}\qquad \text{as}\ t\rightarrow +\infty,
$$
where $\omega$ is the solution to the Allen-Cahn equation
$$
\omega''-W'(\omega)=0,\quad\omega'>0\quad\mbox{in }{\mathbb R},
\quad \omega(0)=0, \quad \omega(\pm\infty)=\pm 1.
$$
The interfaces are described by the functions $\gamma_j(t)$ with $j=1,\cdots,k$,
which are determined by a Toda system and have the forms
$$
\gamma_j(t)=\frac{1}{2\sqrt{2}}\left(j-\frac{k+1}{2}\right)\ln t  +O(1).
$$
The Toda system is different from the one that determine the dynamics of the multiple interfaces of solutions to one dimensional parabolic Allen-Cahn equation established by M. del Pino and K. Gkikas in {\em Proc. R. Soc. Edinb. Sect. A}, 148 (2018), 6: 1165-1199.
\end{abstract}

\maketitle

\tableofcontents

\section{Introduction}
The aim of this paper is to construct a family of multiple-interface solutions to the following one dimensional generalized  parabolic Cahn-Hilliard equation
\begin{equation}\label{eq1.1}
  u_t=-\partial_{xx}\big[\partial_{xx}u -W'(u)\big]+W''(u)\big[\partial_{xx} u -W'(u)\big] \qquad  \text{for}\ (t,x)\in [0,+\infty)\times {\mathbb R},
\end{equation}
where the functions $W', W''$ denote the derivatives of first and second orders of $W$ respectively.
The function $W(s)$ is the  double-well potential defined by
\begin{equation}\label{eq1.4}
  W(s)=\frac{1}{4}\left(1-s^2\right)^2.
\end{equation}
It has two non-degenerate local minimum points $u=+1$ and $u=-1$, which are stable equilibria of equation \eqref{eq1.1}.

This is a companion paper of \cite{ly}. The readers can refer to \cite{ly} and the references therein for the background.
The main purpose of the present paper is to concern \eqref{eq1.1} and construct radial solutions with multiple interfaces, which are driven by the Willmore flow.
The Willmore flow equation is given by
 \begin{equation}\label{willmore}
   V(t)=\Delta_{\partial S}H-\frac{1}{2}H^3+H\|A\|^2,
 \end{equation}
with $V(t)$ denotes the outer normal velocity at $x \in \partial S(t)$, $\Delta_{\partial S}$ is the Laplace-Beltrami operator on surface $\partial S(t)$, $H$ and $A$ are the mean curvature and the second fundamental form of $\partial S(t)$ respectively, and $\|A\|^2$ is the sum of squared coefficients of $A$.

For the radial case of \eqref{willmore}, the radius $\gamma_n(t)$ of the sphere $\Gamma_n(t)$ of $n$-dimension in $\R^{n+1}$ satisfies the equation
\begin{equation}\label{willmoreflowradial}
  \gamma_n'(t)=-\frac{1}{2}\Big(\frac{n-1}{\gamma_n(t)}\Big)^3+\frac{\big(n-1\big)^2}{\big(\gamma_n(t)\big)^3},
\end{equation}
which has a solution
\begin{equation}\label{sphere eq}
\gamma_n(t):=\sqrt[4]{-2(n-3)(n-1)^2t} ,
\end{equation}
where $t\leq0$ when $n\geq 3$ and $t\geq0$ when $n=2$.
Inspired by this fact, the authors studied the parabolic Cahn-Hilliard equation \cite{ly}
\begin{align}\label{eqn}
u_t=-\Delta\left[\Delta u -W'(u)\right]+W''(u)\left[\Delta u -W'(u)\right],
\qquad \forall\,   (t,  x)\in \widetilde{{\mathbb R}}\times{\mathbb R}^n,
\end{align}
where $n=2$ or $n\geq 4$,  $W(\cdot)$ is the typical double-well potential function and $\widetilde{\mathbb R}$ is given by
$$
\widetilde{\mathbb R}=\left\{
                    \begin{array}{rl}
                      (0,  \infty),  &\quad \mbox{if } n=2,
\\
                      (-\infty,  0),  & \quad\mbox{if } n\geq 4.
                    \end{array}
                  \right.
$$
They constructed a radial solution $u(t, |x|)$ with an interface (Shrinking Sphere in \eqref{sphere eq}), which was driven by the Willmore flow.
On the other hand, if $n=1$ or $n=3$, the trivial solution of \eqref{sphere eq} will lead to a trivial solution with an interface to \eqref{eqn}, see Remark 1.3 in \cite{ly}.

A natural question aries: {\em can interaction between multiple interfaces bring new solutions to \eqref{eqn} of one dimension (i.e. problem \eqref{eq1.1}) or three dimension?}
There is an affirmative answer to this question for the case of one dimensional parabolic Allen-Cahn equation (second order) by M. del Pino and K. Gkikas in \cite{dG1}.
It is apparent that new difficulties will appear in this nonlinear parabolic equation of fourth order.
For simplicity, we here will only concern the existence of solutions with multiple interfaces to problem \eqref{eq1.1}.
The readers can refer to Remark \ref{remark1.2}.

To achieve our goal, we need to introduce a function with an interface (one zero point), which will be our block solution to construct others with multiple interfaces.
The elliptic Allen-Cahn equation
  \begin{equation}\label{eqq}
  \left\{
\begin{array}{l}
\omega''+\omega-\omega^3=0,\ \ \ \omega'>0\qquad \text{for}\ x\in{\mathbb R},\\[2mm]
 \omega(0)=0\
 \ \text{and} \ \lim\limits_{x\rightarrow\pm\infty}\omega(x)=\pm 1,
  \end{array}
\right.
  \end{equation}
has a unique solution given by
\begin{equation}\label{domega}
 \omega(x)=\tanh\left(\frac{x}{\sqrt{2}}\right).
\end{equation}

For any positive integer $k\geq2$, we shall find a solution of equation \eqref{eq1.1} with the form
\begin{equation}\label{eq1.12}
u(t,x)= \sum_{j=1}^k(-1)^{j+1}\omega\big(x-\gamma_j(t)\big)-\frac{1}{2}\left[1+(-1)^k\right]+\phi(t,x),
\end{equation}
where $\phi(t,x)$ is a small perturbation when $t$ goes to $+\infty$.
 The functions
 $\gamma_1(t),\cdots,\gamma_k(t)$ satisfy
\begin{equation}\label{eq1.13}
  \gamma_1(t)<\gamma_2(t)<\cdots<\gamma_k(t)\quad \text{and}\quad \gamma_j(t)=-\gamma_{k+1-j}(t).
\end{equation}
We will establish the existence of solution satisfying this characteristic. In fact, the interface dynamic is determined at main order by the following Toda system
\begin{align}\label{toda11}
\nonumber&\frac{2\sqrt{2}}{3}\gamma'_j(t)
\,-\,
384\bigg\{
\,-\,
e^{-\sqrt{2}\big(\gamma_{j}(t)-\gamma_{j-2}(t)\big)}
\,+\,
2e^{-2\sqrt{2}\big(\gamma_j(t)-\gamma_{j-1}(t)\big)}
\\
&\qquad\qquad\qquad\qquad
\,-\,
2e^{-2\sqrt{2}\big(\gamma_{j+1}(t)-\gamma_j(t)\big)}
\,+\,
e^{-\sqrt{2}\big(\gamma_{j+2}(t)-\gamma_j(t)\big)}
\bigg\}=0,
  \end{align}
with $j=1,\cdots,k$,  where we have used the conventions that
$$
\gamma_{-1}=\gamma_{0}=-\infty
\quad\mbox{and}\quad
\gamma_{k+1}=\gamma_{k+2}=+\infty.
$$
The readers can refer to \eqref{system1}.
The solvability of the Toda system \eqref{toda11} is obtained in Lemma \ref{lem11}.
We will see that a solution of the above system is given by
\begin{equation}\label{eq0.1}
  \gamma^0_j(t)=\frac{1}{2\sqrt{2}}\left(j-\frac{k+1}{2}\right)\ln t \,+\, a_{jk},\qquad j=1,\cdots,k,
\end{equation}
for some explicit  constants $a_{jk}$ depending on $j$ and $k$ only.

 Our main result can be stated as follows.
\begin{thm}\label{thm1}
Let $k\geq 2$ be a positive integer and $\gamma^0_j(t)$ be the solution \eqref{eq0.1} of the Toda system \eqref{toda11}.
Then there exist a large number $T>0$ and a solution $u(t,x)$ of $(\ref{eq1.1})$ defined on $[T,+\infty)\times{\mathbb R}$ of the form

\begin{equation*}
u(t,x)= \sum_{j=1}^k(-1)^{j+1}\omega\big(x-\gamma_j(t)\big)-\frac{1}{2}\left[1+(-1)^k\right]+\phi(t,x),
\end{equation*}

\noindent with
\begin{equation*}
\gamma_j(t)=\gamma^0_j(t)+h_j(t),\qquad j=1,\cdots,k,
\end{equation*}
where
\begin{equation*}
  \gamma^0_j(t)=\frac{1}{2\sqrt{2}}\left(j-\frac{k+1}{2}\right)\ln t \,+\, a_{jk},\qquad j=1,\cdots,k,
\end{equation*}
for some explicit  constants $a_{jk}$ depending on $j$ and $k$ only,
the small perturbations $h_j(t)$ and $\phi(t,x)$  satisfy that
$$
\lim\limits_{t\rightarrow +\infty}h_j(t)=0\quad\text{and}\quad\lim\limits_{t\rightarrow +\infty}\phi(t,x)=0 \quad \text{uniformly in}\ x\in {\mathbb R}.
$$
\qed
\end{thm}

More words will be provided in the following to give some explanations.

\begin{rem}\label{remark1.2}
{\it
Notice that the above theorem shows that the existence of multiple-interface solutions to equation \eqref{eq1.1} via a translation with respect to the time variable $t$.

There are many papers using the Toda system to describe the interactions between neighbouring interfaces for (elliptic or parabolic) Alle-Cahn equation.
The reader can refer to  \cite{dG1, dG2, delPinoKowWei1, delPinoKowWei2, delPinoKowWeiYang}
and the references therein.

As far as we know that the above dynamic law of of multiple interfaces given by \eqref{toda11} is firstly found.
This constitutes the new ingredient of the present paper.
The energy
\begin{equation*}
  \mathcal{W}(u)=\frac{1}{2}\int_\R\left[u_{xx}-W'(u)\right]^2\mathrm{d}x,
\end{equation*}
is a Lyapunov functional for \eqref{eq1.1}. If $u(t,x)$ is a solution of \eqref{eq1.1}, we have
\begin{equation}\label{eq1.2}
  \frac{\mathrm{d}}{\mathrm{d}t}\mathcal{W}\big(u(\cdot,t)\big)=-\int_\R\abs{u_t}^2\mathrm{d}x,
\end{equation}
hence $\mathcal{W}(\cdot)$ is non-increasing along any solution of \eqref{eq1.1}.
It is easy to check that Toda system \eqref{toda11} comes from \eqref{eq1.2} as $u(t,x)$ takes the main order term of \eqref{eq1.12}.
Moreover, \eqref{eq1.1} corresponds to the $\mathrm{L}^2$-gradient flow
for $\mathcal{W}(\cdot)$ in $\mathrm{H}^2(\R)$.

It shows a different phenomenon from that of radial multiple interfaces to
parabolic Allen-Cahn equation
\begin{align}\label{al}
u_t=\Delta u-W'(u)=\Delta u+u-u^3\qquad \text{in}\ (-\infty,0]\times{\mathbb R}^n,
\end{align}
which is determined by the system, for $j=1,\cdots,k$
\begin{equation*}
  \rho'_j(t)+\frac{n-1}{\rho_j(t)}-12\sqrt{2}\left\{e^{-\sqrt{2}\big(\rho_{j+1}(t)-\rho_j(t)\big)}-e^{-\sqrt{2}\big(\rho_{j}(t)-\rho_{j-1}(t)\big)}\right\}=0,
\end{equation*}
with $\rho_0=-\infty$ and $\rho_{k+1}=+\infty$,
which has been studied by X. Chen, J.-S. Guo and H. Ninomiya in \cite{cgh} with $n=1$ and $k=2$,
M. del Pino and K. Gkikas in \cite{dG1} with $n=1$ and $k\geq 2$, and \cite{dG2} with $n\geq2$.

Let us explain the  difference of interface dynamics in \eqref{eq1.1} and \eqref{al} with $n=1$.
Consider the energy
\begin{equation*}
 \mathbf{ W}(u):=\int_\R\left[\frac{1}{2}\abs{u_x}^2+W(u)\right]\mathrm{d}x,
\end{equation*}
which is a Lyapunov functional for \eqref{al} with $n=1$.
Similarly, there also holds that for any solution $u(t,x)$ of \eqref{al} with $n=1$,
\begin{equation}
  \frac{\mathrm{d}}{\mathrm{d}t}\mathbf{W}\big(u(\cdot,t)\big)=-\int_\R\abs{u_t}^2\mathrm{d}x.
\label{gradientflowAllenCahn}
\end{equation}
Choosing
$$
u(t,x)= \sum_{j=1}^k(-1)^{j+1}\omega\big(x-\rho_j(t)\big)-\frac{1}{2}\left[1+(-1)^k\right]
$$
in \eqref{gradientflowAllenCahn},
we find that the terms in the left hand side formally take the forms $e^{-\sqrt{2}\{\rho_{j}-\rho_{j-1}\}}$.
However, for \eqref{eq1.2} with a profile in \eqref{eq1.12} the main terms are $e^{-2\sqrt{2}\{\gamma_j-\gamma_{j-1}\}}$ and $e^{-\sqrt{2}\{\gamma_j-\gamma_{j-2}\}}$
due to the vanishing of the coefficients of the terms $e^{-\sqrt{2}\{\gamma_j-\gamma_{j-1}\}}$ of first order in formal expansions.
Roughly speaking, the distance between neighbouring interfaces of the  equation \eqref{eq1.1} is almost a half for equation \eqref{al} with $n=1$.

The interaction of neighbouring two interfaces is of importance for Allen-Cahn equation \eqref{al},
while an interface of Cahn-Hilliard equation \eqref{eq1.1} can be affected by four neighbouring ones in the cluster.
Hence, more work will be needed here to handle this case and prove Theorem \ref{thm1}.
Moreover, this is the reason that we will encounter delicate situations in the construction of radial
multiple interfaces to Cahn-Hilliard equation of higher dimensions (the geometric quantities such as curvatures of interfaces will also play a crucial role),
which will be left to the forthcoming paper.
\qed
}
\end{rem}

\begin{rem}\label{remark1.3}
{\it
We prove our main result in Theorem \ref{thm1} by exploiting Lyapunov-Schmidt reduction method. This method has been extensively
applied to solve the existence problem  of solutions to various kinds of equations \cite{cdm,CDD,dd1,dd2,dG2}. Comparing to the previous works, we are facing three difficulties during the process of dealing with the problem \eqref{eq1.1} as the following:

\noindent{\textbf{(1).}}
Firstly, we need to construct a suitable approximate
solution of equation \eqref{eq1.1} to ensure that the error term is small enough. To get it, some correcting functions are introduced in Section \ref{sec:a}.
The process is divided into two steps: Step one, we give a correcting function around each mid point of adjacent transition points.
Step two, another correcting function is provided around each transition point.

\noindent{\textbf{(2).}}
Secondly, equation \eqref{eq1.1} is a nonlinear parabolic equation of fourth order,
which does not satisfy the Maximum Principle since the kernel of the biharmonic parabolic operator is sign-changing in \cite{FG}.
To overcome it, we employ the blow-up technique and representation of parabolic heat kernel in Section \ref{sec:lp}.
The same tools also are used by us to study the existence of solutions with single radial interface of high-dimensional equation \eqref{eqn} in \cite{ly}.

\noindent{\textbf{(3).}}
Finally, the related heat kernel for the linearization operator of equation \eqref{eq1.1} is more intricate, which leads to hard task to get a suitable priori estimate of
linearized problem of equation \eqref{eq1.1}.
To solve it, we will modify the decay of the error term $E(t,x)$ in \eqref{eq7.5} with polynomial decay in Section \ref{sec:a} and perform more complicated calculations in Section \ref{sec:lp}.
\qed
}
\end{rem}

The paper is organized as follows.

\noindent$\clubsuit$
In Section \ref{sec:a}, we construct an approximate solution of problem \eqref{eq2.1.1} and establish an equivalent problem \eqref{eq7.4} in terms of a perturtation term $\phi$.
To deal with non-invertibility of problem \eqref{eq7.4} caused by the existence of kernel of the corresponding linear operator, a projected version of the nonlinear problem will be given in (\ref{eq7.7})-(\ref{eq7.8}).

\noindent$\clubsuit$
In Section \ref{sec:lp}, we collect some results of linear parabolic equations with  biharmonic operator and obtain the solvability of linear projection problem (\ref{eq3.1}).
The main results of this section will be collected in Proposition \ref{prop2}.

\noindent$\clubsuit$
Section \ref{sec:nl} is devoted to solving the nonlinear problem (\ref{eq7.7})-(\ref{eq7.8}) by the arguments of  the fixed-point theorem.

\noindent$\clubsuit$
The last step (given in Section \ref{sec:tc}) is to choose a suitable parameter $\vec{h}=\{h_1(t),\cdots,h_k(t)\}$ with a form of \eqref{decgama} such that all $c_i(t)=0$ with $i=1,\cdots,k$ in equation (\ref{eq7.7}), and finish the proof of Theorem \ref{thm1}.

\noindent$\clubsuit$
Besides, some tedious but straightforward analysis such as the proofs of Lemmas \ref{esxi}-\ref{lem33}
will be collected as Appendices in Section \ref{appendix}.

\section{The ansatz} \label{sec:a}

\subsection{Preliminaries}
Before proving the main theorem, we need to derive that the decay estimates of the basic layer $\omega$ in \eqref{domega} and the kernel of a  linear operator of fourth order.

Let $\omega$ be the solution of problem $(\ref{eqq})$. It is easy to see that
\begin{align}
\omega(x)-1=-2e^{-\sqrt{2}x}+2e^{-2\sqrt{2}x}+O\left(e^{-3\sqrt{2}x}\right)\qquad \text{for}\ x>1,
\nonumber\\[3mm]
 \omega(x)+1=2e^{\sqrt{2}x}-2e^{2\sqrt{2}x}+O\left(e^{3\sqrt{2}x}\right)\qquad \hspace{1.0cm} \text{for}\ x<-1,
\label{asymptotic}\\[3mm]
 \omega'(x)=2\sqrt{2}e^{-\sqrt{2}\abs{x}}+O\left(e^{-2\sqrt{2}x}\right)\hspace{2.76cm} \text{for}\ \abs{x}>1.
\nonumber
\end{align}
These are well-known results.

Next we consider the kernel of a  linear operator of fourth order. Main result is stated by the following Lemma. For its proof,
see Lemma $3.6$ in \cite{ly} in details.
\begin{lem}
Let $\omega(x)$ be the solution of \eqref{eqq}.
Then any solution of the linear homogeneous problem
\begin{equation*}
 \big[-\partial_{xx}+W'(\omega(x))\big]^2\varphi(x)=0 \qquad\text{for}\ x\in{\mathbb R},
\end{equation*}
with $\varphi, \partial_{xx}\varphi\in L^\infty({\mathbb R})$, has the form
\begin{equation*}
  \varphi(x)=\lambda \omega'(x),
\end{equation*}
with some constant $\lambda\in{\mathbb R}$, where the potential $W(s)$ is defined by \eqref{eq1.4}.
\qed
\end{lem}

\subsection{The setting-up to prove Theorem \ref{thm1}}\label{subsec}
We will prove the validity of Theorem \ref{thm1} for any integer $k\geq2$.
Let $\check{u}(t,x)$ be
 a solution of \eqref{eq1.1}, by a translation $u(t,x)=\check{u}(t+T,x)$,  then we have that $u$ satisfies  the following problem
\begin{equation}\label{eq2.1.1}
  u_t=-\partial_{xx}\big[\partial_{xx}u-W'(u)\big]+W''(u)\big[\partial_{xx}u-W'(u)\big]\qquad \text{in}\ [T,+\infty)\times{\mathbb R},
\end{equation}
where $W$ is defined in (\ref{eq1.4}) and $T$ is a large positive number which will be adjusted at each step.
For the convenience of notation, we set
\begin{equation}\label{deF}
  F(u):=-\partial_{xx}\big[\partial_{xx}u-W'(u)\big]+W''(u)\big[\partial_{xx}u-W'(u)\big],
\end{equation}
and then get
\begin{equation}\label{deF'}
  F'(u)[v]:=-\partial_{xx}\big[\partial_{xx}v-W''(u)v\big]+W''(u)\big[\partial_{xx}v-W''(u)v\big]+W'''(u)\big[\partial_{xx}u-W'(u)\big]v.
\end{equation}
The strategy to solve \eqref{eq2.1.1} will be sketched in Sections \ref{section2.2.1}-\ref{section2.2.2}.

\subsubsection{The approximate solution}\label{section2.2.1}
Let $k\geq2$ be an integer. We define the first approximate solution of \eqref{eq2.1.1} as the following

\begin{equation}\label{eq2.1.2}
  z_1(t,x):=\sum_{j=1}^k(-1)^{j+1}\omega\big(x-\gamma_j(t)\big)-\frac{1}{2}\left[1+(-1)^k\right],
\end{equation}
where $\omega$ is the solution of (\ref{eqq}), the functions $\gamma_j(t)$ are ordered and symmetric, that is
\begin{equation}\label{eq2.2}
  \gamma_1(t)<\gamma_2(t)<\cdots<\gamma_{k-1}(t)<\gamma_k(t)\qquad \text{and}\qquad\gamma_j(t)=-\gamma_{k+1-j}(t).
\end{equation}
We set $\vec{\gamma}(t):=\big[\gamma_1(t),\cdots,\gamma_k(t)\big]^\top$
 and decompose $\vec{\gamma}(t)$ into two parts,
\begin{align}\label{decgama}
\vec{\gamma}(t):=\vec{\gamma}^0(t)+\vec{h}(t),
\end{align}
where $\vec{\gamma}^0(t):=\big[\gamma_1^0(t),\cdots,\gamma_k^0(t)\big]^\top$  with
\begin{align}\label{degama0}
\gamma^0_j(t):=\frac{1}{2\sqrt{2}}\left(j-\frac{k+1}{2}\right)\ln\big[1152t\big]+a_j,
\qquad j=1,\cdots, k,
\end{align}
and $a_j$ are some constants to be determined in Section \ref{sec:tc}.
In the above, we also have set
$$
\vec{h}(t):=\big[h_1(t),\cdots,h_k(t)\big]^\top,
$$
where the functions $h_j(t)$ satisfy the following assumptions
\begin{equation}\label{eq2.4}
  \|h_j\|_{L^\infty[1,+\infty)}+\|\abs{t} h'_j\|_{L^\infty[1,+\infty)}\leq 1\quad \text{and}\quad \lim_{t\rightarrow+\infty}\big[\abs{h_j(t)}+\abs{h'_j(t)}\big]=0.
\end{equation}

In the next, we will introduce some correcting functions to improve the error term
\begin{equation*}
  \partial_tz_1(t,x)-F(z_1(t,x)).
\end{equation*}
This process is divided into the following two steps in details.

\noindent \textbf{Step 1:}  {\em The first correction}

We here focus on the neighbourhood of  $\frac{\gamma_i(t)+\gamma_{i+1}(t)}{2}$ with  $1\leq i\leq k-1$.
Substituting $z_1(t,x)$ into equation \eqref{eq2.1.1},
 by \eqref{asymptotic} and  equation \eqref{eqq},  we get that
\begin{align*}
&\partial_tz_1(t,x)-F(z_1(t,x))
\\[2mm]
&=\partial_tz_1(t,x)
\ +\ (-1)^{i+1}\Big[\partial_{xx}-W''\big(z_1(t,x)\big)\Big]
\Bigg\{W'\big(f^-_i(t,x)+1\big)
-W'\big(f^+_{i+1}(t,x)-1\big)
\\[2mm]&\quad
-W'\big(1+f^-_i(t,x)-f^+_{i+1}(t,x)+f(t,x)\big)
+\sum_{\substack{j=1,\\ j\neq i,i+1}}^{k}(-1)^{j+i}\omega''\big(x-\gamma_j(t)\big)\Bigg\}
\\[2mm]
&=(-1)^{i+1}\Big[\partial_{xx}-W''\big(z_1(t,x)\big)\Big]\Bigg\{\underbrace{6f^-_i(t,x)f^+_{i+1}(t,x)}_{\textrm{I}_i(t,x)}
\ +\ O\big(f(t,x)\big)
\\[2mm]
&\qquad\qquad+3f^-_i(t,x)f^+_{i+1}(t,x)\bigg[f^-_i(t,x)-f^+_{i+1}(t,x)\bigg]
+\sum_{\substack{j=1,\\ j\neq i,i+1}}^{k}(-1)^{j+i}\omega''\big(x-\gamma_j(t)\big)\Bigg\}
\\[2mm]
&\quad
+\sum_{j=1}^k(-1)^{j}\gamma'_j(t)\omega'\big(x-\gamma_j(t)\big),
\end{align*}
where have defined the functions
$$
f^{\pm}_{i}(t,x):=\omega\big(x-\gamma_i(t)\big)\pm1
\quad\mbox{ with } i=1,\cdots,k,
$$
and
$$f(t,x):=\sum_{j=1}^{i-1}(-1)^{j+i}\bigg[\omega\big(x-\gamma_j(t)\big)-1\bigg]
+\sum_{j=i+2}^{k}(-1)^{j+i}\bigg[\omega\big(x-\gamma_j(t)\big)+1\bigg].$$

To get a better approximation, we need to cancel the terms $\textrm{I}_i(t,x)$ in the above estimate. In order to achieve it, we define  correcting functions $\xi_i(t,x)$ as the following
\begin{align}\label{dxi}
 \nonumber \xi_i(t,x):&=6e^{\sqrt{2}x}\int_{\frac{\gamma_i(t)+\gamma_{i+1}(t)}{2}}^xe^{-2\sqrt{2}s}\int_{-\infty}^se^{\sqrt{2}\tau}f^-_i(t,\tau)f^+_{i+1}(t,\tau)
 \,{\mathrm d}\tau \,{\mathrm d}s
 \\[2mm]
 &+12e^{\sqrt{2}\big\{x-\gamma_{i+1}(t)\big\}}\int_{0}^{e^{\frac{1}{\sqrt{2}}\big\{\gamma_i(t)+\gamma_{i+1}(t)\big\}}}
 \frac{y}{\big(y+e^{\sqrt{2}\gamma_i(t)}\big)
  \big(y+e^{\sqrt{2}\gamma_{i+1}(t)}\big)}\,{\mathrm d}y,
\end{align}
with $i=1,\cdots,k-1$, which solve
\begin{equation}\label{xi'seq}
  \big[\partial_{xx}-W''(1)\big]\xi_i(t,x)=6f^-_i(t,x)f^+_{i+1}(t,x),
\end{equation}
where we notice that $W''(1)=2$.

The estimates of the functions $\xi_i(t,x)$ are given by the following lemma, whose proof will be postponed to Appendix \ref{section6.1}.

\begin{lem}\label{esxi}
Let $\sigma\in(0,\sqrt{2})$ and $1\leq i\leq k-1$. There hold that
\begin{align*}
  &\abs{\xi_i(t,x)}+\abs{\partial_x\xi_i(t,x)}+\abs{t}\abs{\partial_t\xi_i(t,x)}\leq Ct^{-\frac{1}{2}}e^{-\sigma\abs{x-\gamma_i(t)}}\qquad\quad \text{if}\  x\leq\gamma_i(t),
 \\[3mm]
 &\abs{\xi_i(t,x)}+\abs{\partial_x\xi_i(t,x)}+\abs{t}\abs{\partial_t\xi_i(t,x)}\leq Ct^{-\frac{1}{2}}e^{-\sigma\abs{x-\gamma_{i+1}(t)}}\qquad\ \text{if}\ x\geq\gamma_{i+1}(t),
\end{align*}
and
\begin{equation*}
   \abs{\xi_i(t,x)}+\abs{\partial_x\xi_i(t,x)}+\abs{t}\abs{\partial_t\xi_i(t,x)}\leq Ct^{-\frac{1}{2}}\qquad \text{if}\  \gamma_i(t)\leq x\leq\gamma_{i+1}(t),
\end{equation*}
where $C>0$ is a constant  which does not depend on $t$ and $x$. Moreover, the function $\xi_i(t,x)$ is symmetric about the point $x=\frac{\gamma_i(t)+\gamma_{i+1}(t)}{2}$ with respect to the variable $x$.
\qed
\end{lem}

\noindent \textbf{Step 2:} {\em The second correction}

Let us define
$$
z_2(t,x):=\sum_{j=1}^k(-1)^{j+1}\omega\big(x-\gamma_j(t)\big)
-\sum_{j=1}^{k-1}(-1)^{j+1}\xi_j(t,x)-1.
$$
When $x$ close to the point $x=\gamma_i(t)$ with $1\leq i\leq k$, using \eqref{asymptotic}, equations \eqref{eqq} and \eqref{xi'seq} together with Taylor expansion, we have that
\begin{align*}
&(-1)^{i+1}\big[\partial_tz_2(t,x)-F(z_2(t,x))\big]
\\[2mm]
&=(-1)^{i+1}\Big\{\partial_tz_2(t,x)
\ +\ \big[\partial_{xx}-W''(z_2(t,x))\big]\big[\partial_{xx}z_2(t,x)-W'(z_2(t,x))\big]\Big\}
\\[2mm]
&=\Big[\partial_{xx}-W''\big(z_2(t,x)\big)\Big]
\Bigg\{-W'\big(f^+_{i+1}(t,x)-1\big)+W'\big(\omega(x-\gamma_i)\big)-W'\big(f^-_{i-1}(t,x)+1\big)
\\[2mm]
&\qquad-W'\Big(\omega\big(x-\gamma_i(t)\big)-f^+_{i+1}(t,x)-f^-_{i-1}(t,x)+\xi_{i-1}(t,x)-\xi_i(t,x)+\widetilde{f}(t,x)\Big)
\\[2mm]
&\qquad+\partial_{xx}\xi_{i-1}(t,x)-\partial_{xx}\xi_{i}(t,x)
-\sum_{\substack{j=1,\\ j\neq i-1,i}}^{k-1}(-1)^{j+i}\partial_{xx}\xi_j(t,x)
\\[2mm]
&\qquad+\sum_{\substack{j=1,\\ j\neq i-1,i,i+1}}^k(-1)^{j+i}\omega''\big(x-\gamma_j(t)\big)
\Bigg\}
\\[2mm]
&\quad+\sum_{j=1}^k(-1)^{j+i}\gamma'_j(t)\omega'\big(x-\gamma_j(t)\big)+\sum_{j=1}^{k-1}(-1)^{j+i}\partial_t\xi_j(t,x)
\\[2mm]
&=\Big[\partial_{xx}-W''\big(z_2(t,x)\big)\Big]
\Bigg\{\underbrace{3\Big[1-\omega\big(x-\gamma_i(t)\big)\Big]^2f^+_{i+1}(t,x)
\ +\ 3\Big[\omega^2\big(x-\gamma_i(t)\big)-1\Big]\xi_i(t,x)}_{\widehat{\textrm{I}}_{1,i}(t,x)}
\\[2mm]
&\qquad+\underbrace{3\Big[1+\omega\big(x-\gamma_i(t)\big)\Big]^2f^-_{i-1}(t,x)
\ +\
3\Big[1-\omega^2\big(x-\gamma_i(t)\big)\Big]\xi_{i-1}(t)
}_{\widehat{\textrm{I}}_{2,i}(t,x)}
\ +\ O\Big(f^-_{i-1}f^+_{i+1}\Big)
\\[2mm]
&\qquad+3\Big[1-\omega\big(x-\gamma_i(t)\big)\Big]\left[f^+_{i+1}(t,x)\right]^2
\ -\
3\Big[1+\omega\big(x-\gamma_i(t)\big)\Big]\left[f^-_{i-1}(t,x)\right]^2
\\[2mm]
&\qquad+O\Big(\xi^2_i(t,x)+\xi^2_{i-1}(t,x)\Big)
+O\Bigg(\sum_{l=i-1,i}\big[f_{i-1}^-(t,x)+f_{i+1}^+(t,x)\big]\xi_l(t,x)\Bigg)
\\[2mm]
&\qquad+O\big(\widetilde{f}(t,x)\big)+\sum_{\substack{j=1,\\ j\neq i-1,i,i+1}}^k(-1)^{j+i}W'\big(\omega\big(x-\gamma_j(t)\big)\big)-\sum_{\substack{j=1,\\ j\neq i-1,i}}^{k-1}(-1)^{j+i}\partial_{xx}\xi_j(t,x)\Bigg\}
\\[2mm]
&\quad+\sum_{j=1}^{k-1}(-1)^{j+i}\partial_t\xi_j(t,x)+\sum_{j=1}^k(-1)^{j+i}\gamma'_j(t)\omega'\big(x-\gamma_j(t)\big),
\end{align*}
where the function $\widetilde{f}$ is given by
\begin{align*}
\widetilde{f}(t,x):=&\sum_{j=1}^{i-2}(-1)^{j+i}\bigg[\omega\big(x-\gamma_j(t)\big)-1\bigg]
\ +\
\sum_{j=i+2}^{k}(-1)^{j+i}\bigg[\omega\big(x-\gamma_j(t)\big)+1\bigg]
\\&
-\sum_{\substack{j=1,\\ j\neq i-1,i}}^{k-1}(-1)^{j+i}\xi_j(t,x).
\end{align*}
Notice that $\widehat{\textrm{I}}_{i,1}(t,x)=O\big(t^{-\frac{1}{2}}\big)$ and  $\widehat{\textrm{I}}_{i,2}(t,x)=O\big(t^{-\frac{1}{2}}\big)$  when $x$ near to $\gamma_i(t)$.

We will get rid of the terms $(-1)^{i+1}\big[\,\widehat{\textrm{I}}_{i,1}(t,x)+\widehat{\textrm{I}}_{i,2}(t,x)\big]$
with $i=1,\cdots,k$ in the above equation to improve the approximation.
Inspired by \cite{mm,r1}, we define the modifying functions as the following
\begin{equation}\label{dxi1}
  \widetilde{\xi}_i(t,x):=\omega'(x)\int_0^x\big(\omega'(s)\big)^{-2}\,{\mathrm d}s\int_{-\infty}^s\widetilde{\textrm{I}}_{i}(t,\tau)\omega'(\tau)\mathrm{d}\tau,
\end{equation}
for all $i=1,\cdots,k$, which satisfy
\begin{equation}\label{xi1eq}
  \big[\partial_{xx}-W''(\omega(x))\big] \widetilde{\xi}_i(t,x)=\widetilde{\textrm{I}}_i(t,x),
\end{equation}
where the terms $\widetilde{\textrm{I}}_i(t,x)$ with $ i=1,\cdots,k$, are the projection of $\widehat{\textrm{I}}_{i,1}\big(t,x+\gamma_i(t)\big)+\widehat{\textrm{I}}_{i,2}\big(t,x+\gamma_i(t)\big)$ on the space
\begin{equation}\label{dpots}
\textrm{K}:=\left\{\mathfrak{f}\in L^2({\mathbb R}):\int_{\mathbb R}\omega'(x)\mathfrak{f}(x)\mathrm{d}x=0\right\},
\end{equation}
More precisely, the functions $\widetilde{\textrm{I}}_i$ are defined as
\begin{equation}\label{dIi}
  \widetilde{\textrm{I}}_i(t,x):=\widehat{\textrm{I}}_{1,i}\big(t,x+\gamma_i(t)\big)+\widehat{\textrm{I}}_{2,i}\big(t,x+\gamma_i(t)\big)-d_i(t)\omega'(x),
\end{equation}
where the functions $d_i(t)$ are given by
\begin{equation}\label{ddt}
\ d_i(t):=\frac{\displaystyle\int_{\mathbb R} \Big[\widehat{\textrm{I}}_{1,i}\big(t,x+\gamma_i(t)\big)+\widehat{\textrm{I}}_{2,i}\big(t,x+\gamma_i(t)\big)\Big]\omega'(x)\mathrm{d}x}{\int_{\mathbb R}\big(\omega'(x)\big)^2\mathrm{d}x}.
\end{equation}
The estimates of the functions $\widetilde{\xi}_j(t,x)$ and $d_j(t)$ are given by the following lemma. For its proof, see Appendix \ref{section6.2}.
\begin{lem}\label{lem33}
Let $1\leq i\leq k$. There hold that
\begin{align*}
\left\{
\begin{array}{l}
  \abs{\widetilde{\xi}_i(t,x)}
  \,+\,\abs{\partial_x\widetilde{\xi}_i(t,x)}
  \,+\,\abs{t}\abs{\partial_t\widetilde{\xi}_i(t,x)}
  \,\leq\, Ct^{-\frac{1}{2}}\abs{x}e^{-\sqrt{2}\abs{x}},
  \\[4mm]
   d_i(t)=12\sqrt{2}\left\{e^{-\sqrt{2}\eta_i(t)}-e^{-\sqrt{2}\eta_{i+1}(t)}\right\} \,+\, O\left(t^{-1}\ln t\right),
  \end{array}
\right.
\end{align*}
where the functions $\eta_j$ are defined by
\begin{align}
\begin{aligned}
\eta_j(t):=\gamma_j(t)-\gamma_{j-1}(t),\quad \forall\, j=2,\cdots,k,
\\
\eta_1(t)=\eta_{k+1}(t):=+\infty.
\end{aligned}
\label{eta}
\end{align}
Moreover, the constant $C>0$ does not depend on $t$ and $x$.
\qed
\end{lem}

 According to the above arguments,  we now define an approximate solution of problem \eqref{eq2.1.1} as the following
\begin{align}
  z(t,x):=&\sum_{j=1}^k(-1)^{j+1}\omega\big(x-\gamma_j(t)\big)-\sum_{j=1}^{k-1}(-1)^{j+1}\xi_j(t,x)
\nonumber\\[2mm]
&-\sum_{j=1}^k(-1)^{j+1}\widehat{\xi}_j(t,x)-\frac{1}{2}\left[1+(-1)^k\right],
\label{dzz}
\end{align}
where the functions  $\widehat{\xi}_j(t,x)$ with $j=1,\cdots,k$,  are defined as
\begin{equation}\label{dhatxi}
  \widehat{\xi}_j(t,x):=\widetilde{\xi}_j\big(t,x-\gamma_j(t)\big).
\end{equation}
Note that the functions $\xi_i(t,x)$ and $\widetilde{\xi}_j(t,x)$ are given by \eqref{dxi} and \eqref{dxi1} respectively..

\subsubsection{The nonlinear projected problem}\label{section2.2.2}
We want to find a solution to equation (\ref{eq2.1.1}) of the form
\begin{equation*}
  u(t,x)=z(t,x)+\phi(t,x).
\end{equation*}
Thus the function $\phi(t,x)$ satisfies the following equation:
\begin{equation}\label{eq7.4}
\begin{aligned}
\phi_t=F'(z(t,x))[\phi]+E(t,x)+N(\phi)\qquad \text{in}\
[T,+\infty)\times{\mathbb R},
\end{aligned}
\end{equation}
where error term $E(t,x)$ is given by
\begin{equation}\label{eq7.5}
\begin{aligned}
  E(t,x):=-\partial_tz(t,x)+F(z(t,x)).
\end{aligned}
\end{equation}
The nonlinear term $N(\phi)$ is defined by
\begin{equation}\label{eq7.6}
N(\phi):=F\big(z(t,x)+\phi(t,x)\big)-F(z(t,x))-F'(z(t,x))[\phi(t,x)],
\end{equation}
where the operator $F(u)$ is given by \eqref{deF}.

To solve equation (\ref{eq7.4}), we first consider the projection problem:
\begin{equation}\label{eq7.7}
\begin{aligned}
\phi_t=F'(z(t,x))[\phi]+E(t,x)+N(\phi)-\sum\limits_{j=1}^kc_j(t)\omega'\big(x-\gamma_j(t)\big) \qquad   \text{in}\
[T,+\infty)\times{\mathbb R},
 \end{aligned}
\end{equation}
and
\begin{equation}\label{eq7.8}
  \int_{{\mathbb R}} \phi(t,x)\omega'\big(x-\gamma_j(t)\big)\mathrm{d}x=0 \qquad  \text{for all}\ j=1,\cdots,k\ \text{and}\ t> T.
\end{equation}
Here the functions $c_j(t)$ are chosen so that $\phi(t,x)$ satisfies the orthogonality condition (\ref{eq7.8}), namely the following (nearly
diagonal) system holds:
\begin{equation}\label{eq7.9}
  \begin{aligned}
&\sum_{j=1}^kc_j(t)\int_{{\mathbb R}}\omega'\big(x-\gamma_j(t)\big)\omega'\big(x-\gamma_i(t)\big)\mathrm{d}x
  \\
  &=\int_{{\mathbb R}} F'(z(t,x))[\phi]\omega'\big(x-\gamma_i(t)\big)\mathrm{d}x
  -\gamma'_i(t)\int_{0}^\infty \phi(t,r)\omega''\big(x-\gamma_i(t)\big)\mathrm{d}x
 \\
 &  \ \ \ +\int_{{\mathbb R}}\big[E(t,x)+N(\phi)\big]\omega'\big(x-\gamma_i(t)\big)\mathrm{d}x,
  \ \ \ \ \forall\, i=1,\cdots,k\ \text{and}\ t>T.
  \end{aligned}
\end{equation}
The existence of solutions to \eqref{eq7.7}-\eqref{eq7.8} will be proved in Sections \ref{sec:lp} and \ref{sec:nl}.
We shall choose suitable $h_j(t)$ with $j=1,\cdots,k$ such that all $c_j(t)$ equal to zero in Section \ref{sec:tc},
which means that $z(t,x)+\phi(t,x)$ will exactly satisfy problem \eqref{eq2.1.1}.

\subsubsection{The estimates of the error terms}
At the end of this section, we  establish some estimates of the error term $E(t,x)$ in (\ref{eq7.5}) which are stated by the following lemma.
\begin{lem}\label{lem10}
Let $\sigma\in\big(0,\sqrt{2}\big)$, $\alpha>1$ and $T>1$. We set
\begin{equation}
\begin{aligned}\label{dphi}
\Phi(t,x):=\left\{
\begin{array}{l}
 \frac{1}{t^{\frac{3}{4}+\frac{\sigma}{8\sqrt{2}}}}\left[\frac{1}{\big(\abs{x-\gamma_0^0}+1\big)^\alpha}+\frac{1}{\big(\abs{x-\gamma_2^0}+1\big)^\alpha}\right]\ \ \ \ \qquad \text{if}\ x\leq \frac{\gamma^0_1(t)+\gamma^0_2(t)}{2},\\[6mm]
 \frac{1}{t^{\frac{3}{4}+\frac{\sigma}{8\sqrt{2}}}}\left[\frac{1}{\big(\abs{x-\gamma_{j-1}^0}+1\big)^\alpha}+\frac{1}{\big(\abs{x-\gamma_{j+1}^0}+1\big)^\alpha}\right]\quad \quad\text{if}\ \frac{\gamma^0_{j-1}(t)+\gamma^0_{j}(t)}{2}\leq x\leq \frac{\gamma^0_j(t)+\gamma^0_{j+1}(t)}{2},\\[2mm] \hspace{7.8cm} \text{with }\ j=2,\cdots,k-1,
 \\[4mm]
 \frac{1}{t^{\frac{3}{4}+\frac{\sigma}{8\sqrt{2}}}}\left[\frac{1}{\big(\abs{x-\gamma_{k-1}^0}+1\big)^\alpha}+\frac{1}{\big(\abs{x-\gamma_{k+1}^0}+1\big)^\alpha}\right] \qquad \text{if}\ x\geq\frac{\gamma^0_{k-1}(t)+\gamma^0_k(t)}{2},
  \end{array}
\right.
\end{aligned}\end{equation}
with the functions $\gamma^0_j(t)$ for $1\leq j\leq k$ in \eqref{degama0}, and the conventions $\gamma_0^0=-\infty$, $\gamma_{k+1}^0=+\infty$.
Then there exists a positive constant $C$ independent of $t$ and $x$ such that
\begin{equation*}
  \abs{E(t,x)}\leq Ct^{-\frac{\sigma}{16\sqrt{2}}}\Phi(t,x)\qquad \text{for all} \ x\in{\mathbb R}\ \text{and}\ t>T,
\end{equation*}
where the error term $E(t,x)$ is given by $(\ref{eq7.5})$.
\end{lem}

\begin{proof}
We first notice that
$${\mathbb R}=\bigcup_{j=1}^k\left[\frac{\gamma_j^0(t)+\gamma_{j-1}^0(t)}{2},\frac{\gamma^0_j(t)+\gamma^0_{j+1}(t)}{2}\right],$$
the functions $\gamma^0_i(t)$ with $1\leq i\leq k$ are defined in \eqref{degama0}, $\gamma^0_0=-\infty$ and $\gamma_{k+1}^0=+\infty$.
Thus there exist the following two cases.

\textbf{Case one:} {\em when $x\in\left[\frac{\gamma^0_{i-1}(t)+\gamma^0_i(t)}{2},\frac{\gamma^0_i(t)+\gamma^0_{i+1}(t)}{2}\right]$ with $i=2,\cdots,k-1$}

By the definitions in \eqref{deF}, \eqref{eq7.5} and \eqref{dzz}, we have that
  \begin{align}\label{eq3}
  \nonumber&E(t,x)=-\partial_tz(t,x)+F(z(t,x))
  \\[2mm]&\nonumber=\sum_{j=1}^{k}(-1)^{j+1}\gamma'_j(t)\omega'\big(x-\gamma_j(t)\big)
      +\sum_{j=1}^{k-1}(-1)^{j+1}\partial_t\xi_j(t,x)+\sum_{j=1}^k(-1)^{j+1}\partial_t\widehat{\xi}_j(t,x)
\\[2mm]&\nonumber\quad -\Big[\partial_{xx}-W''(z(t,x))\Big]\Big[\partial_{xx}z(t,x)-W'(z(t,x))\Big]
\\[2mm]&=-\Big[\partial_{xx}-W''(z(t,x))\Big]\Big[\partial_{xx}z(t,x)-W'(z(t,x))\Big]+O\left(t^{-1}\right),
  \end{align}
  where we have used  \eqref{degama0}, \eqref{eq2.4} and Lemmas \ref{esxi}-\ref{lem33} in the last equality.

Thus we need to study the first term in the above. Firstly,  we define the function
\begin{align*}
  \widehat{f}(t,x):=&\sum_{j=1}^{i-3}(-1)^{j+i}\bigg[\omega\big(x-\gamma_j(t)\big)-1\bigg]+\sum_{j=i+3}^{k}(-1)^{j+i}\bigg[\omega\big(x-\gamma_j(t)\big)+1\bigg]
  \\&-\sum_{j=1}^{i-3}(-1)^{i+j}\xi_j(t,x)-\sum_{j=i+2}^{k-1}(-1)^{i+j}\xi_j(t,x)
     -\sum_{\substack{j=1,\\ j\neq i-1,i,i+1}}(-1)^{i+j}\widehat{\xi}(t,x).
  \end{align*}
By the Taylor expansion and the equations \eqref{eqq},  \eqref{xi'seq} and \eqref{xi1eq}, we have that
\begin{align}
\nonumber&(-1)^{i+1}\Big[\partial_{xx}z(t,x)-W'\big(z(t,x)\big)\Big]
\\[2mm]
&\nonumber=\omega''\big(x-\gamma_i(t)\big)+\sum_{j=i-2}^{i-1}(-1)^{j+i}\partial_{xx}f^-_{j}(t,x)+\sum_{j=i+1}^{i+2}(-1)^{j+i}\partial_{xx}f^+_{j}(t,x)+\partial_{xx}\widehat{f}(t,x)
      \\[2mm]
  &\nonumber\quad-\sum_{j=i-2}^{i+1}(-1)^{j+i}\partial_{xx}\xi_j(t,x)-\sum_{j=i-1}^{i+1}(-1)^{j+i}\partial_{xx}\widehat{\xi}_j(t,x)-W'\left((-1)^{i+1}z(t,x)\right)
  \\[2mm]
&\nonumber=W'\big(f^-_{i-2}(t,x)+1\big)-W'\big(f^-_{i-1}(t,x)+1\big)-W'\big(f^+_{i+1}(t,x)-1\big)+W'\big(f^+_{i+2}(t,x)-1\big)
      \\[2mm]
      &\nonumber\quad-\sum_{j=i-2}^{i+1}(-1)^{j+i}\left[2\xi_j(t,x)+6f_j^-(t,x)f_{j+1}^+(t,x)\right]
      \\[2mm]
      &\nonumber\quad-\sum_{j=i-1}^{i+1}(-1)^{j+i}\partial_{xx}\widehat{\xi}_j(t,x)+\partial_{xx}\widehat{f}(t,x)
      \\[2mm]
      &\nonumber\quad-W''\big(\omega\big(x-\gamma_i(t)\big)\big)\Bigg\{\sum_{j=i-2}^{i-1}(-1)^{j+i}f^-_{j}(t,x)+\sum_{j=i+1}^{i+2}(-1)^{j+i}f^+_{j}(t,x)+\widehat{f}(t,x)
       \\[2mm]
       &\nonumber\qquad\qquad-\sum_{j=i-2}^{i+1}(-1)^{j+i}\xi_j(t,x)
       -\sum_{j=i-1}^{i+1}(-1)^{j+i}\widehat{\xi}_j(t,x)\Bigg\}
       \\[2mm]
       &\nonumber\quad-3\omega\big(x-\gamma_i(t)\big)\Bigg\{\sum_{j=i-2}^{i-1}(-1)^{j+i}f^-_{j}(t,x)
       +\sum_{j=i+1}^{i+2}(-1)^{j+i}f^+_{j}(t,x) \ +\ \widehat{f}(t,x)
       \\[2mm]
       &\nonumber\qquad\qquad-\sum_{j=i-2}^{i+1}(-1)^{j+i}\xi_j(t,x)
       -\sum_{j=i-1}^{i+1}(-1)^{j+i}\widehat{\xi}_j(t,x)\Bigg\}^2
       \\[2mm]
       &\nonumber\quad-\Bigg\{\sum_{j=i-2}^{i-1}(-1)^{j+i}f^-_{j}(t,x)+\sum_{j=i+1}^{i+2}(-1)^{j+i}f^+_{j}(t,x)+\widehat{f}(t,x)
       \\[2mm]
       &\nonumber\qquad\qquad-\sum_{j=i-2}^{i+1}(-1)^{j+i}\xi_j(t,x)
       -\sum_{j=i-1}^{i+1}(-1)^{j+i}\widehat{\xi}_j(t,x)\Bigg\}^3.
\end{align}
By the notation
\begin{align}
d_i(t)\omega'\big(x-\gamma_i(t)\big)
=&-\partial_{xx}\widehat{\xi}_i(t,x)+W''\big(\omega\big(x-\gamma_i(t)\big)\big)\widehat{\xi}_i(t,x)
\nonumber\\[2mm]
&+3\left[1-\big[\omega\big(x-\gamma_i(t)\big)\big]^2\right]\xi_{i-1}(t,x)
+3\big[f_i^+\big]^2f_{i-1}^-
\nonumber\\[2mm]
&-3\left[1-\big[\omega\big(x-\gamma_i(t)\big)\big]^2\right]\xi_{i}(t,x)
+3\big[f_i^-\big]^2f_{i+1}^+,
\end{align}
and
\begin{align}
\hbar(t, x)=&\sum_{j=i-2}^{i-1}(-1)^{j+i}f^-_{j}(t,x)+\sum_{j=i+1}^{i+2}(-1)^{j+i}f^+_{j}(t,x)+\widehat{f}(t,x)
       \nonumber\\[2mm]
       &-\sum_{j=i-2}^{i+1}(-1)^{j+i}\xi_j(t,x)-\sum_{j=i-1}^{i+1}(-1)^{j+i}\widehat{\xi}_j(t,x),
\end{align}
we further obtain that
\begin{align}
\nonumber&(-1)^{i+1}\Big[\partial_{xx}z(t,x)-W'\big(z(t,x)\big)\Big]
\\[2mm]
&\nonumber= d_i(t)\omega'\big(x-\gamma_i(t)\big) +\partial_{xx}\widehat{\xi}_{i-1}(t,x)-W''\big(\omega\big(x-\gamma_i(t)\big)\big)\widehat{\xi}_{i-1}(t,x)
     \\[2mm]
       &\nonumber\quad
       +\partial_{xx}\widehat{\xi}_{i+1}(t,x)-W''\big(\omega\big(x-\gamma_i(t)\big)\big)\widehat{\xi}_{i+1}(t,x)
       \\[2mm]
       &\nonumber\quad-3\left[1-\big[\omega\big(x-\gamma_i(t)\big)\big]^2\right]\xi_{i-2}(t,x)+3\left[1-\big[\omega\big(x-\gamma_i(t)\big)\big]^2\right]\xi_{i+1}(t,x)
        \\[2mm]&\nonumber\quad -3f_i^+(t,x)f_i^-(t,x)\Big[f^-_{i-2}(t,x)+f^+_{i+2}(t,x)\Big]-6f_{i-2}^-(t,x)f_{i-1}^+(t,x)+6f_{i+1}^-(t,x)f_{i+2}^+(t,x)
       \\[2mm]
       &\nonumber\quad+3\sum_{j=i-2}^{i-1}(-1)^{j+i}\Big[f^-_{j}(t,x)\Big]^2-3\sum_{j=i+1}^{i+2}(-1)^{j+i}\Big[f^+_{j}(t,x)\Big]^2
       \\[2mm]
       &\nonumber\quad-3\omega\big(x-\gamma_i(t)\big)\hbar^2(t, x)-\hbar^3(t, x)
       \\[2mm]
       &\nonumber\quad+\sum_{j=i-2}^{i-1}(-1)^{j+i}\Big[f^-_{j}(t,x)\Big]^3+\sum_{j=i+1}^{i+2}(-1)^{j+i}\Big[f^+_{j}(t,x)\Big]^3
       \\[2mm]
       &\nonumber\quad +\partial_{xx}\widehat{f}(t,x)-W''\big(\omega\big(x-\gamma_i(t)\big)\big)\widehat{f}(t,x)
\\[2mm]&\nonumber=d_i(t)\omega'\big(x-\gamma_i(t)\big)
       +\Big[W''\big(\omega\big(x-\gamma_{i-1}(t)\big)\big)-W''\big(\omega\big(x-\gamma_i(t)\big)\big)\Big]\widehat{\xi}_{i-1}(t,x)
      \\[2mm]
      &\nonumber\quad+3\left\{\big[\omega\big(x-\gamma_i(t)\big)\big]^2-\big[\omega\big(x-\gamma_{i-1}(t)\big)\big]^2\right\}\xi_{i-2}(t,x)
       \\[2mm]
      &\nonumber\quad -3\left\{1-\big[\omega\big(x-\gamma_{i-1}(t)\big)\big]^2\right\}\xi_{i-1}(t,x)
       \\[2mm]
       &\nonumber\quad+3\big[f_{i-1}^+\big]^2f_{i-2}^-+3\big[f_{i-1}^-\big]^2f_{i}^+-d_{i-1}(t)\omega'\big(x-\gamma_{i-1}(t)\big)
       \\[2mm]
       &\nonumber\quad+\Big\{W''\big(\omega\big(x-\gamma_{i+1}(t)\big)\big)-W''\big(\omega\big(x-\gamma_i(t)\big)\big)\Big\}\widehat{\xi}_{i+1}(t,x)
       \\[2mm]
      &\nonumber\quad+3\left\{1-\big[\omega\big(x-\gamma_{i+1}(t)\big)\big]^2\right\}\xi_{i}(t,x)
       \\[2mm]
       &\nonumber\quad
       +3\Big\{W''\big(\omega\big(x-\gamma_{i+1}(t)\big)\big)-W''\big(\omega\big(x-\gamma_i(t)\big)\big)\Big\}\xi_{i+1}(t,x)
       \\[2mm]
       &\nonumber\quad+3\big[f_{i+1}^+\big]^2f_{i}^- +3\big[f_{i+1}^-\big]^2f_{i+2}^+
       -d_{i+1}(t)\omega'\big(x-\gamma_{i+1}(t)\big)
        \\[2mm]&\nonumber\quad -3f_i^+(t,x)f_i^-(t,x)\Big[f^-_{i-2}(t,x)+f^+_{i+2}(t,x)\Big]-6f_{i-2}^-(t,x)f_{i-1}^+(t,x)+6f_{i+1}^-(t,x)f_{i+2}^+(t,x)
       \\[2mm]
       &\nonumber\quad+3\sum_{j=i-2}^{i-1}(-1)^{j+i}\Big[f^-_{j}(t,x)\Big]^2-3\sum_{j=i+1}^{i+2}(-1)^{j+i}\Big[f^+_{j}(t,x)\Big]^2
       \\[2mm]
       &\nonumber\quad-3\omega\big(x-\gamma_i(t)\big)\hbar(t, x)^2-\hbar(t, x)^3
       \\[2mm]
       &\nonumber\quad+\sum_{j=i-2}^{i-1}(-1)^{j+i}\Big[f^-_{j}(t,x)\Big]^3+\sum_{j=i+1}^{i+2}(-1)^{j+i}\Big[f^+_{j}(t,x)\Big]^3
       \\[2mm]
       &\nonumber\quad +\partial_{xx}\widehat{f}(t,x)-W''\big(\omega\big(x-\gamma_i(t)\big)\big)\widehat{f}(t,x)
  \\[2mm]&\nonumber
  =d_i(t)\omega'\big(x-\gamma_i(t)\big)-d_{i-1}(t)\omega'\big(x-\gamma_{i-1}(t)\big)-d_{i+1}(t)\omega'\big(x-\gamma_{i+1}(t)\big)
       \\[2mm]
       &\nonumber\quad+3\big[f_{i-1}^-\big]^2f_{i}^+ +3\big[f_{i+1}^+\big]^2f_{i}^--3\left\{1-\big[\omega\big(x-\gamma_{i-1}(t)\big)\big]^2\right\}\xi_{i-1}(t,x)
       \\[2mm]
       &\nonumber\quad+3\left\{1-\big[\omega\big(x-\gamma_{i+1}(t)\big)\big]^2\right\}\xi_{i}(t,x)
       \\[2mm]
       &\nonumber\quad-3\left[\big[f^-_{i-1}\big]^2f^+_i+\big[f^+_{i+1}\big]^2f^-_i\right]-6\omega\big(x-\gamma_i(t)\big)\Big[f^+_{i+1}\xi_i-f^-_{i-1}\xi_{i-1}\Big]
       \\[2mm]
       &\nonumber\quad+3\left[\omega^2\big(x-\gamma_{i-1}(t)\big)-1\right]\left[\widehat{\xi}_{i-1}(t,x)-\xi_{i-2}(t,x)\right]+3\left[\omega^2\big(x-\gamma_i(t)\big)-1\right]\xi_{i-2}(t,x)
       \\[2mm]
       &\nonumber\quad+3\left[\omega^2\big(x-\gamma_{i+1}(t)\big)-1\right]\left[\widehat{\xi}_{i+1}(t,x)-\xi_{i+1}(t,x)\right]+3\left[\omega^2\big(x-\gamma_{i+1}(t)\big)-1\right]\xi_{i+1}(t,x)
       \\[2mm]
       &\nonumber\quad+3\Big\{1-\omega^2\big(x-\gamma_i(t)\big)\Big\}\big[\widehat{\xi}_{i-1}(t,x)+\widehat{\xi}_{i+1}(t,x)\big]+3f_{i-1}^+f_{i-1}^-f_{i-2}^-+3f_{i+1}^-f_{i+1}^+f_{i+2}^+
       \\[2mm]
       &\nonumber\quad-3f_i^+(t,x)f_i^-(t,x)\Big[f^-_{i-2}(t,x)+f^+_{i+2}(t,x)\Big]
        \\[2mm]
       &\nonumber\quad+3\sum_{j=i-2}^{i-1}(-1)^{j+i}\Big[f^-_{j}(t,x)\Big]^2-3\sum_{j=i+1}^{i+2}(-1)^{j+i}\Big[f^+_{j}(t,x)\Big]^2
       \\[2mm]
       &\nonumber\quad+3\left[\big[f^-_{i-1}\big]^2f^+_i+\big[f^+_{i+1}\big]^2f^-_i\right]+6\omega\big(x-\gamma_i(t)\big)\Big[f^+_{i+1}\xi_i-f^-_{i-1}\xi_{i-1}\Big]
       \\[2mm]
       &\nonumber\quad-3\omega\big(x-\gamma_i(t)\big)\hbar(t, x)^2-\hbar(t, x)^3
        \\[2mm]
       &\nonumber\quad+\sum_{j=i-2}^{i-1}(-1)^{j+i}\Big[f^-_{j}(t,x)\Big]^3+\sum_{j=i+1}^{i+2}(-1)^{j+i}\Big[f^+_{j}(t,x)\Big]^3
       \\[2mm]
       &\nonumber\quad +\partial_{xx}\widehat{f}(t,x)-W''\big(\omega\big(x-\gamma_i(t)\big)\big)\widehat{f}(t,x)
  \\[2mm]
  &:=d_i(t)\omega'\big(x-\gamma_i(t)\big)-d_{i-1}(t)\omega'\big(x-\gamma_{i-1}(t)\big)-d_{i+1}(t)\omega'\big(x-\gamma_{i+1}(t)\big)+\mathcal{E}_i(z(t,x)).
\label{esee}
\end{align}
  According to \eqref{asymptotic} and Lemmas \ref{esxi}-\ref{lem33}, we get that the function  $\mathcal{E}_i(z(t,x))$ satisfies that
  \begin{align}\label{estmE}
  \mathcal{E}_i(z(t,x))=O\left(t^{-\frac{3}{4}-\frac{\sigma}{4\sqrt{2}}}\right),
  \end{align}
  for all $x\in\left[\frac{\gamma^0_{i-1}(t)+\gamma^0_i(t)}{2},\frac{\gamma^0_i(t)+\gamma^0_{i+1}(t)}{2}\right]$ and $t>T$, and $\sigma\in\left(0,\sqrt{2}\right)$.

  Thus, by the Taylor expansion and Lemmas \ref{esxi}-\ref{lem33} again, we have that
  \begin{align}\label{eq2}
  \nonumber&(-1)^{i+1}\Big[\partial_{xx}-W''(z(t,x))\Big]\Big[\partial_{xx}z(t,x)-W'(z(t,x))\Big]
  \\[2mm]
  &\nonumber=\Big[\partial_{xx}-W''(\omega\big(x-\gamma_i(t)\big)+\check{f}(t,x))\Big]\Big[\partial_{xx}z(t,x)-W'(z(t,x))\Big]
  \\[2mm]
  &\nonumber=\Big[\partial_{xx}-W''\big(\omega\big(x-\gamma_i(t)\big)\big)\Big]\Big[\partial_{xx}z(t,x)-W'(z(t,x))\Big]
  \\[2mm]
  &\nonumber\quad  +O\left(\check{f}(t,x)\right)\Big[\partial_{xx}z(t,x)-W'(z(t,x))\Big]
  \\[2mm]
  &\nonumber=\Big[\partial_{xx}-W''\big(\omega\big(x-\gamma_i(t)\big)\big)\Big]\Big[\partial_{xx}z(t,x)-W'(z(t,x))\Big]+O\left(t^{-\frac{3}{4}-\frac{\sigma}{4\sqrt{2}}}\right)
  \\[2mm]
  &\nonumber=-d_{i-1}(t)\Big[W''\big(\omega\big(x-\gamma_{i-1}(t)\big)\big)-W''\big(\omega\big(x-\gamma_i(t)\big)\big)\Big]\omega'\big(x-\gamma_{i-1}(t)\big)
  \\[2mm]
  &\nonumber\quad-d_{i+1}(t)\Big[W''\big(\omega\big(x-\gamma_{i+1}(t)\big)\big)-W''\big(\omega\big(x-\gamma_i(t)\big)\big)\Big]\omega'\big(x-\gamma_{i+1}(t)\big)+O\left(t^{-\frac{3}{4}-\frac{\sigma}{4\sqrt{2}}}\right)
  \\[2mm]
  &=O\left(t^{-\frac{3}{4}-\frac{\sigma}{4\sqrt{2}}}\right),
  \end{align}
  where $\sigma\in(0,\sqrt{2})$, and the function $\check{f}(t,x)$ is defined as
  \begin{align*}
    \check{f}(t,x):=&\sum_{j=1}^{i-1}(-1)^{j+1}\bigg[\omega\big(x-\gamma_j(t)\big)-1\bigg]+\sum_{j=i+1}^{k}(-1)^{j+i}\bigg[\omega\big(x-\gamma_j(t)\big)+1\bigg]
  \\[2mm]&-\sum_{j=1}^{k-1}(-1)^{j+1}\xi_j(t,x)-\sum_{j=1}^k(-1)^{j+1}\widehat{\xi}(t,x).
  \end{align*}

Hence according to \eqref{eq3} and \eqref{eq2}, we have that
\begin{align*}
\abs{E(t,x)}\leq Ct^{-\frac{3}{4}-\frac{\sigma}{4\sqrt{2}}}\qquad \text{for all}\ (t,x)\in[T,+\infty)\times\left[\frac{\gamma^0_1(t)+\gamma^2_2(t)}{2},\frac{\gamma^0_{k-1}(t)+\gamma^0_{k}(t)}{2}\right].
\end{align*}
Using the above estimate and the following fact that
\begin{align*}
\abs{\Phi(t,x)}\geq t^{-\frac{3}{4}-\frac{\sigma}{8\sqrt{2}}}\frac{1}{[\eta(t)]^\alpha}= t^{-\frac{3}{4}-\frac{\sigma}{8\sqrt{2}}}\left[\frac{\ln t}{2\sqrt{2}}\right]^{-\alpha}\quad \text{for}\ x\in\left[\frac{\gamma^0_1(t)+\gamma^2_2(t)}{2},\frac{\gamma^0_{k-1}(t)+\gamma^0_{k}(t)}{2}\right],
\end{align*}
we obtain that
\begin{align*}
\abs{E(t,x)}\leq C t^{-\frac{\sigma}{16\sqrt{2}}}\Phi(t,x),
\end{align*}
in this case, where $C>0$ does not depend on $t$ and $x$.

\smallskip
\textbf{Case two:} {\em when $x\leq\frac{\gamma^0_1(t)+\gamma^2_2(t)}{2}$ or $x\geq\frac{\gamma^0_{k-1}(t)+\gamma^0_{k}(t)}{2}$}

Similarly, we can obtain that
  $$\abs{E(t,x)}=O\left(t^{-\frac{3}{4}-\frac{\sigma}{4\sqrt{2}}}\right).$$
  Moreover, according \eqref{asymptotic} and Lemmas \ref{esxi}-\ref{lem33}, we have that $E(t,x)$ has an exponential decay i.e.
  \begin{equation*}
    \abs{E(t,x)}=O\left(t^{-\frac{3}{4}-\frac{\sigma}{4\sqrt{2}}}e^{-\sigma\abs{x-\gamma^0_1(t)}}\right)\qquad \text{for all}\ x\leq\gamma^0_1(t),
  \end{equation*}
  and
 \begin{equation*}
    \abs{E(t,x)}=O\left(t^{-\frac{3}{4}-\frac{\sigma}{4\sqrt{2}}}e^{-\sigma\abs{x-\gamma^0_k(t)}}\right)\qquad \text{for all}\ x\geq\gamma^0_k(t).
  \end{equation*}

  Combining all above and the following facts that
  \begin{align*}
  t^{-\frac{3}{4}-\frac{\sigma}{4\sqrt{2}}}e^{-\sigma\abs{x-\gamma^0_1(t)}}&\leq C(\alpha,\sigma) t^{-\frac{3}{4}-\frac{\sigma}{4\sqrt{2}}}\left[1+\abs{x-\gamma^0_1(t)}\right]^{-\alpha}
  \\[2mm]&\leq C(\alpha,\sigma) t^{-\frac{3}{4}-\frac{3\sigma}{16\sqrt{2}}}\left[1+\abs{x-\gamma^0_2(t)}\right]^{-\alpha}\qquad \text{for }\ x\leq\gamma^0_1(t),
  \end{align*}
  and
  \begin{align*}
  t^{-\frac{3}{4}-\frac{\sigma}{4\sqrt{2}}}e^{-\sigma\abs{x-\gamma^0_k(t)}}&\leq C(\alpha,\sigma) t^{-\frac{3}{4}-\frac{\sigma}{4\sqrt{2}}}\left[1+\abs{x-\gamma^0_k(t)}\right]^{-\alpha}
  \\[2mm]&\leq C(\alpha,\sigma) t^{-\frac{3}{4}-\frac{3\sigma}{16\sqrt{2}}}\left[1+\abs{x-\gamma^0_{k-1}(t)}\right]^{-\alpha}\qquad \text{for }\ x\geq\gamma^0_k(t),
  \end{align*}
   we can obtain the desired result in the lemma.
\end{proof}

\begin{rem}
{\em
According the proof of the above lemma, the error term $E(t,x)$
  satisfies that
   $$E(t,x)\sim t^{-\frac{3}{4}}e^{-\sigma\abs{x}},$$
  with $0<\sigma<\sqrt{2}$ when $x$ goes to infinity.
The linearization  operator of problem \eqref{eq7.7} at $z(t,x)$ is defined as
\begin{equation*}\begin{aligned}
\partial_t-F'(z(t,x))=&\partial_t+\partial_{xxxx}-2W''(z(t,x))\partial_{xx}-\partial_{xx}\Big{[}W''(z(t,x))+(W''(z(t,x))\Big{]}^2\\ &-2W'''(z(t,x))\partial_xz(t,x)\partial_x-W'''(z(t,x))\Big{[}\partial_{xx}z(t,x)-W'(z(t,x))\Big{]}.
\end{aligned}\end{equation*}
However, the kernel of this linear fourth order parabolic operator does not provide enough decay, that is  solutions $\phi(t,x)$ of linear parabolic equation
$$\partial_t\phi-F'(z(t,x))\phi=f\ \  \ \ \ \ \text{with}\ f \ \text{satisfying}\ \sup\limits_{(t,x)\in (T,+\infty)\times{\mathbb R}}t^{\frac{3}{4}}e^{\sigma\abs{x}}\abs{f}<+\infty,$$
 may does not have exponential decay $e^{-\sigma\abs{x}}$.  As far as we known,  the solution $\phi(t,x)$ satisfies polynomial decay.
 }
 \qed
 \end{rem}

\section{The linear problems}\label{sec:lp}

In this section, we will prove that the linear projected problem \eqref{eq3.1} is solvable by using Proposition \ref{prop8} and a priori estimate in Lemma \ref{lem5}. The main result is given by Proposition \ref{prop2}.

\subsection{Some preliminary facts}
Let $u_0\in C^1({\mathbb R}^n)\cap L^\infty({\mathbb R}^n)$, and we consider the initial value problem
\begin{equation}\label{heatnonlineareq}
\left\{
\begin{array}{l}
u_t+(-\Delta)^2u=G[\Delta u,\nabla u,u,t,x]\qquad \text{in}\ (t_0,t_1)\times{\mathbb R},
\\[2mm] u(t_0,x)=u_0(x)\qquad \text{in}\ {\mathbb R},
 \end{array}
\right.
\end{equation}
where $G[p,\vec{q},s,t,x]:{\mathbb R}\times{\mathbb R}\times{\mathbb R}\times[t_0,+\infty)\times{\mathbb R}\rightarrow {\mathbb R}$  is a measurable function satisfying
\begin{enumerate}
\item For every $M>0$ and $T>t_0$, there exists $C_{T,M}>0$ such that
$$
\Big{|}G[p,\vec{q},s,t,x]\Big{|}\leq C_{T,M}
$$
for all $x\in{\mathbb R}^n$, $t\in[t_0,T]$ and $p,\abs{\vec{q}},s \in[-M,M]$.
\item There is a constant $L>0$ satisfies
\begin{equation}\label{delpf}
  \Big{|}G[p_1,\vec{q}_1,s_1,t,x]-G[p_2,\vec{q}_2,s_2,t,x]\Big{|}\leq L\Big{(}\abs{p_1-p_2}+\abs{\vec{q}_1-\vec{q}_2}+\abs{s_1-s_2}\Big{)},
\end{equation}
for all $ x\in{\mathbb R}^n,t\geq t_0$, $p_1,p_2 \in{\mathbb R}$, $\vec{q}_1,\vec{q}_2\in{\mathbb R}^n$ and $s_1,s_2\in{\mathbb R}$.
\end{enumerate}
 In particular, we can take
\begin{equation}\label{defG}
G[\Delta u,\nabla u,u, t, x]=a(t,x)\Delta u+\sum_{i=1}^nb^i(t,x)\nabla_{x_i} u+c(t,x)u+g(t,x),
\end{equation}
where those functions $a(t,x)$, $b^1(t,x),\cdots, b^n(t, x)$, $c(t,x)\in L^\infty ([t_0,+\infty),C^1({\mathbb R}))$ and $g\in L^\infty([t_0,+\infty)\times{\mathbb R})$.

The solvability of problem \eqref{heatnonlineareq} is given by the following proposition which comes from Proposition $3.4$ in \cite{ly}.
\begin{prop}\label{prop8}
For any given $u_0\in C^1({\mathbb R})\cap L^\infty({\mathbb R})$, there exists a unique solution $u$ to problem \eqref{heatnonlineareq} with $t_1=+\infty$. Moreover, if $u_0\in C^2({\mathbb R})$, then the solution $u$ has the property: for any $s>t_0$, there exists $C$ independent of $s$ such that
\begin{equation*}
\sup_{t\in(t_0,s)}\|u(t,\cdot)\|_{C^{2}({\mathbb R})}
 \leq C\Big{(}\Big[(s-t_0)+(s-t_0)^{1/2}\Big]\|G(0,0,0,t,x)\|_{L^{\infty}((t_0,s)\times{\mathbb R})}+\|u_0\|_{C^2({\mathbb R})}\Big{)}.
\end{equation*}
\qed
\end{prop}

\subsection{The solvability of linear problem}
In what follows our main purpose is to solve the problem (\ref{eq7.7})-\eqref{eq7.8}.
To achieve it, we first find a solution of the following linear parabolic problem:
\begin{equation}\label{eq3.1}
\left\{
\begin{array}{l}
\psi_t=F'(z(t,x))[\psi]+g(t,x)-\sum\limits_{j=1}^kc_j(t)\omega'\big(x-\gamma_j(t)\big)
\qquad \text{in}\ \ [T,+\infty)\times{\mathbb R},\\[6mm]
\int_{{\mathbb R}}\psi(t,x)\omega'\big(x-\gamma_j(t)\big)\,{\mathrm d}x=0\qquad \text{for all}\  t\in [T,+\infty) \ \text{and} \ j=1,\cdots,k,
 \end{array}
\right.
\end{equation}
for a bounded function $g$, and $T$ fixed sufficiently large, where the linear operator $F'(z(t,x))[\psi]$ is defined by \eqref{deF'}.
The numbers $c_j(t)$ are exactly those that make the relations above consistent, namely, by definition for each $t>T$
they solve the linear system of the equations
\begin{equation}\label{eq3.2}
  \begin{aligned}
  &\sum_{j=1}^kc_j(t)\int_{{\mathbb R}}\omega'\big(x-\gamma_j(t)\big)\omega'\big(x-\gamma_i(t)\big)\,{\mathrm d}x\\[2mm]&=
  \int_{{\mathbb R}} F'(z(t,x))[\psi]\omega'\big(x-\gamma_i(t)\big)\,{\mathrm d}x-\gamma_i'(t)\int_{{\mathbb R}} \psi(t,x)\omega''\big(x-\gamma_i(t)\big)\,{\mathrm d}x\\[2mm]&
  \ \ \ +\int_{{\mathbb R}}g(t,x)\omega'(x-\gamma_i(t)\,{\mathrm d}x,\quad \forall\ i=1,\cdots,k\ \text{and}\ t>T.
  \end{aligned}
\end{equation}
This system can indeed be solved uniquely since if $T$ is taken sufficiently large, the matrix with coefficients $\int_{{\mathbb R}}\omega'\big(x-\gamma_j(t)\big)\omega'\big(x-\gamma_i(t)\big)\,{\mathrm d}x$ are nearly diagonal.

Next we mainly build a linear operator $\psi=\mathcal{A}(g)$ that defines a solution of (\ref{eq3.1}) which is bounded for a norm suitably adapted to our setting.
Defining a function space
\begin{equation}\label{eq3.3}
  C_{\Phi}((s,t)\times{\mathbb R}):=\left\{ u : \|u\|_{C_{\Phi}((s,t)\times{\mathbb R})}:=\left\| \frac{u(t,x)}{\Phi(t,x)}\right\|
_{L^\infty((s,t)\times{\mathbb R})}<+\infty\right\},
\end{equation}
where $s>t\geq T$, the function $\Phi(t,x)$ is given by (\ref{dphi}).

The main result in this subsection is as follows:
\begin{prop}\label{prop2}
There exist positive numbers $T_0$ and  $C$ such that for each $g\in C_{\Phi}((0,+\infty)\times{\mathbb R}) $,
there exists a solution $\psi=\mathcal{A}(g)$ of the problem \eqref{eq3.1} with $T=T_0$,  which defines a linear operator of $g$ and
satisfies the estimate
\begin{align}
& \|\psi\|_{C_{\Phi}((t,+\infty)\times{\mathbb R})} +\|\psi_x\|_{C_{\Phi}((t,+\infty)\times{\mathbb R})}+ \|\psi_{xx}\|_{C_{\Phi}((t,+\infty)\times{\mathbb R})}
\nonumber\\[2mm]
&\leq C\|g\|_{C_{\Phi}((t,+\infty)\times{\mathbb R})}, \quad \forall\ t\geq T_0.
\label{eq3.4}
\end{align}
\end{prop}
The proof of Proposition \ref{prop2} will be provided in Section \ref{section3.2.4}.
It bases on several lemmas in Sections \ref{section3.2.1}-\ref{section3.2.3}.

\subsubsection{}\label{section3.2.1}
Let us first consider the following Cauchy problem:
\begin{equation}\label{eq3.5}
\left\{
\begin{array}{l}
\psi_t=F'(z(t,x))[\psi]+g(t,x)\qquad  \text{for}\ (t,x)\in [T,s)\times{\mathbb R},\\[3mm]
\psi(T,x)=0\qquad  \text{for}\ x\in {\mathbb R},
 \end{array}
\right.
\end{equation}
where $g(t,x)\in C_{\Phi}((T,+\infty)\times{\mathbb R})$, $T>2$ and $s>T+1$. Notice that the above problem \eqref{eq3.5} has a unique solution $\psi^s(t,r)$.
Indeed, recall the definition of  $F'(z(t,x))[\psi]$ in \eqref{deF'},
\begin{equation*}\begin{aligned}
 F'(z(t,x))[\psi]=&-\psi_{xxxx}+2W''\big(z(t,x)\big)\psi_{xx}+2W'''\big(z(t,x)\big)z'(t,x)\psi_x-\big[W''(z(t,x))\big]^2\psi
 \\[2mm]&
 +\big[W''(z(t,x))\big]''\psi+W'''\big(z(t,x)\big)\Big{(}z''(t,x)-W'(z(t,x))\Big{)}\psi,
\end{aligned}\end{equation*}
 where the approximate solution $z(t,x)$ is defined by \eqref{dzz}. In above equation, we find that
 those coefficients in the front of the terms $\partial_{xx}\psi$, $\partial_x\psi$ and $\psi$, are all smooth and bounded for $x\in{\mathbb R}$ and $t>2$. Thus, according to
 Proposition \ref{prop8}, we know that the problem \eqref{eq3.5} is uniquely solvable.

We will establish a priori estimate for the solution $\psi^s$ of (\ref{eq3.5}) which is independent of $s$.

\begin{lem}\label{lem5}
Let $g\in C_{\Phi}((T,s)\times{\mathbb R})$ and $\psi^s$ be a solution of the problem
$(\ref{eq3.5})$ which satisfies the orthogonality condition
\begin{equation}\label{eq3.6}
  \int_{\mathbb R} \psi^s(t,x)\omega'\big(x-\gamma_j(t)\big)\,{\mathrm d}x=0 \qquad \text{for all}\ j=1,\cdots,k\ \text{and}\ t\in\left(T,s\right).
\end{equation}
Then there exists a large constant $T_0>2$ such that for any $T>T_0$, the following estimates is valid
\begin{equation}\label{eq3.7}
  \left\|\psi^s\right\|_{C_{\Phi}((T,s)\times{\mathbb R})}+\|\psi^s_x\|_{C_{\Phi}((T,s)\times{\mathbb R})}+ \|\psi^s_{xx}\|_{C_{\Phi}((T,s)\times{\mathbb R})}\leq C\|g\|_{C_{\Phi}((T,s)\times{\mathbb R})},
\end{equation}
where $C>0$ is a uniform constant depending on $T_0$.
\end{lem}
\begin{proof}

We prove the above lemma by arguments of contradiction. Assume that there exist two sequences
$\{s_m\}$ and $\{t_m\}$ such that $0<s_m+1<t_m$ and $s_m\rightarrow +\infty$,  $t_m\rightarrow +\infty$
when the sub-index $m$ goes to infinity. We assume that there exists $g_{m}\in C_{\Phi}((s_m,t_m)\times{\mathbb R})$ such that
\begin{equation}\label{eq3.8}
 \|\psi^m\|_{C_{\Phi}((s_m,t_m)\times{\mathbb R})}+\|\psi^m_x\|_{C_{\Phi}((s_m,t_m)\times{\mathbb R})}+ \|\psi^m_{xx}\|_{C_{\Phi}((s_m,t_m)\times{\mathbb R})}=1,
\end{equation}
and
\begin{equation}\label{eq8}
  \|g_{m}\|_{C_{\Phi}((s_m,t_m)\times{\mathbb R})}\rightarrow 0
\qquad \text{as}\ m\rightarrow+\infty,
\end{equation}
where $\psi^m$ is the solution of (\ref{eq3.5}) and (\ref{eq3.6}) with $g=g_m$, $T=s_m$ and $s=t_m$.
\\

\textbf{Assertion}: \emph{For any number $R>0$, we have that
\begin{equation}\label{eq3.9}
\lim_{m\rightarrow\infty}\left\{\left\|\frac{\psi^m}{\Phi}\right\|_{L^\infty\left(A^{(s_m,t_m)}_{j,R}\right)}+\left\|\frac{\psi^m_x}{\Phi}\right\|_{L^\infty\left(A^{(s_m,t_m)}_{j,R}\right)}
+\left\|\frac{\psi^m_{xx}}{\Phi}\right\|_{L^\infty\left(A^{(s_m,t_m)}_{j,R}\right)}\right\}=0,
\end{equation}
for any $j=1,\cdots,k$, where $\Phi(t,x)$ is given by $(\ref{dphi})$ and $A^{(s_m,t_m)}_{j,R}$ is defined as
$$A^{(s_m,t_m)}_{j,R}:=\Big\{(\tau,x)\in (s_m,t_m)\times {\mathbb R}: \ \abs{x-\gamma^0_j(\tau)}<R+1\Big\},$$
and these functions $\gamma^0_j$ with $j=1,\cdots,k$ are defined by \eqref{degama0}.}
\qed
\\

The proof of (\ref{eq3.9}) will be postponed to Section \ref{section3.2.2}.
We first accept the validity of (\ref{eq3.9}).
By \eqref{eq3.8}, we have that the following three situations at least one holds,
 $$\|\psi^m\|_{C_{\Phi}((s_m,t_m)\times{\mathbb R})}\geq\frac{1}{3}\quad \text{or}\quad \|\psi^m_x\|_{C_{\Phi}((s_m,t_m)\times{\mathbb R})}\geq\frac{1}{3}\quad \text{or}\quad
\|\psi^m_{xx}\|_{C_{\Phi}((s_m,t_m)\times{\mathbb R})}\geq\frac{1}{3}.$$
Without loss of generality, we assume that $\|\psi^m\|_{C_{\Phi}((s_m,t_m)\times{\mathbb R})}\geq\frac{1}{3}$. For the other two situations, the following
arguments can be adapted in a simple way.

According to the definition of $C_{\Phi}((s_m,t_m)\times{\mathbb R})$ in \eqref{eq3.3}, we derive that
\begin{equation*}
  \|\psi^m\|_{C_{\Phi}((s_m,t_m)\times{\mathbb R})}=\sup_{(t,x)\in(s_m,t_m)\times{\mathbb R}}\frac{\abs{\psi^i(t,x)}}{\Phi(t,x)}\geq\frac{1}{3}.
\end{equation*}
Thus there exists $(\bar{t}_m,\bar{x}_m)\in (s_m,t_m)\times{\mathbb R}$ such that
\begin{equation*}
 \frac{\abs{\psi^m(\bar{t}_m,\bar{x}_m)}}{\Phi(\bar{t}_m,\bar{x}_m)}\geq\frac{1}{4}.
\end{equation*}
Furthermore, by (\ref{eq3.9}), we have
\begin{equation}\label{j1}
  \lim_{m\rightarrow +\infty}\abs{\gamma^0_{j}(\bar{t}_m)-\bar{x}_m}=+\infty\qquad \text{for any}\ j=1,\cdots,k.
\end{equation}

Recall that
$$
{\mathbb R}=\cup_{j=1}^{k}\Big{(}\frac{\gamma^0_{j-1}+\gamma^0_{j}}{2},\ \frac{\gamma^0_{j+1}+\gamma^0_{j}}{2}\Big{]}
$$
with the conventions $\gamma^0_0(t)=-\infty$ and $\gamma^0_{k+1}(t)=+\infty$.
Hence for each $m$, there exists $j_m\in \{1,\cdots,k\}$ such that
$$
\bar{x}_m\in\left(\frac{\gamma^0_{j_m-1}(\bar{t}_m)+\gamma^0_{j_m}(\bar{t}_i)}{2},\ \frac{\gamma^0_{j_m}(\bar{t}_i)+\gamma^0_{j_m+1}(\bar{t}_m)}{2}\right].
$$
Let us define
\begin{equation}\label{dp}
  \phi^m(t,y):=\frac{\psi^m\big(t+\bar{t}_m,\ y+\bar{y}_m+\gamma^0_{j_m}(t+\bar{t}_m)\big)}{\Phi\big(\bar{t}_m,\ \bar{y}_m+\gamma^0_{j_m}(\bar{t}_m)\big)},
\end{equation}
where $\bar{y}_m:=\bar{x}_m-\gamma^0_{j_m}(\bar{t}_m)$.
Thus $\phi^m(t,y)$ satisfies the following problem
\begin{equation}\label{eqj.2}
\left\{
\begin{array}{l}
\phi^m_t=-\phi^m_{yyyy} \,+\,2W''\big(\bar{z}(t,y)\big)\phi^m_{yy}
\,+\, 2W'''\big(\bar{z}(t,y)\big)\bar{z}'(t,y)\phi^m_y \,-\,\big[W''(\bar{z}(t,y))\big]^2\phi^m
\\[3mm]
\hspace{1cm}\,+\,\big[W''(\bar{z}(t,y))\big]''\phi^m
\,+\,W'''\big(\bar{z}(t,y)\big)\Big(\bar{z}''(t,y) \,-\,W'(\bar{z}(t,x))\Big)\phi^m  \,+\,g^m(t,y)
\\[2mm]\hspace{1.5cm}
\text{in}\ (s_m-\bar{t}_m,t_m-\bar{t}_m]\times{\mathbb R},
\\[3mm]
\abs{\phi^m(0,0)}\geq\frac{1}{4},
\\[3mm]
\phi^m(s_m-\bar{t}_m,y)=0\qquad  \text{in}\ {\mathbb R},
 \end{array}
\right.
\end{equation}
where the functions are given by
$$
\bar{z}(t,y):=z\big(t+\bar{t}_m,y+\bar{y}_m+\gamma^0_{j_m}(t+\bar{t}_m)\big),
$$
and
$$
g^m(t,y):=\frac{g\big{(}t+\bar{t}_m,\ y+\bar{y}_m+\gamma^0_{j_m}(t+\bar{t}_m)\big{)} \,+\,\psi^m_y\partial_t\gamma^0_{j_m}(t+\bar{t}_m)}
{\Phi\big(\bar{t}_m,\ \bar{y}_m+\gamma^0_{j_m}(\bar{t}_m)\big)}.
$$
Notice that there exist two cases
\\
\noindent\textbf{Case $1$}: $\lim\limits_{m\rightarrow\infty}s_m-\bar{t}_m=-\infty$.
\\
\noindent\textbf{Case $2$}: $\lim\limits_{m\rightarrow\infty}s_m-\bar{t}_m=-\ell$ for a constant $\ell>0$.
\\
In the sequel, to obtain a contradiction and finish the proof of Lemma \ref{lem5},
the detailed process is divided into two steps, where the contradiction will be given in {\textbf{Step two}}.

\textbf{Step one.} We have that: $\phi^m\rightarrow\phi$ locally uniformly with $\abs{\phi(0,0)}>\frac{1}{4}$,  and $\phi$ satisfies
\begin{equation}\label{limeq}
\phi_{t}=-\phi_{yyyy}+2W''(1)\phi_{yy}-[W''(1)]^2\phi\qquad \text{in}\ (\bar{\ell}, 0]\times{\mathbb R},
\end{equation}
where $\bar{\ell}=-\ell$ or $-\infty$.

By the definition of $\phi^m(t,y)$ in \eqref{dp} and the assumption in \eqref{eq3.8}, we have that
\begin{align}
 \abs{\phi^m(t,y)}=&\,\abs{\frac{\psi^m(t+\bar{t}_m,\ y+\bar{y}_m+\gamma^0_{j_m}(t+\bar{t}_m))}{\Phi(\bar{t}_m,\ \bar{y}_m+\gamma^0_{j_m}(\bar{t}_m))}}
\nonumber \\[2mm]
\leq&\,\abs{\frac{\Phi(t+\bar{t}_m,\ y+\bar{y}_m+\gamma^0_{j_m}(t+\bar{t}_m))}{\Phi(\bar{t}_m,\ \bar{y}_m+\gamma^0_{j_m}(\bar{t}_m))}},
\label{esp0}
 \end{align}
for all $(t,y)\in A_l^m $ and $l=1,\cdots,k$, where the sets $ A_l^m$ are defined as
\begin{align}\label{dA}
  A_l^m:=\Bigg\{
  &(t,y):\ t\in(s_m-\bar{t}_m,\ 0]\quad \text{and}
   \\
   &y+\bar{y}_m+\gamma^0_{j_m}(t+\bar{t}_m)\in\left(\frac{\gamma^0_{l-1}(t+\bar{t}_m)+\gamma^0_{l}(t+\bar{t}_m)}{2},\ \frac{ \gamma^0_{l}(t+\bar{t}_m)+\gamma^0_{l+1}(t+\bar{t}_m)}{2}\right]\Bigg\},
   \nonumber
\end{align}
with $\gamma^0_0=-\infty$ and $\gamma_{k+1}^0=+\infty$.
Next we estimate the term on the right hand side of the above inequality \eqref{esp0}.
Recall that
$$
\bar{y}_m+\gamma^0_{j_m}(\bar{t}_m)\in\left(\frac{\gamma^0_{j_m-1}(\bar{t}_m)+\gamma^0_{j_m}(\bar{t}_m)}{2},\ \frac{\gamma^0_{j_m}(\bar{t}_i)+\gamma^0_{j_m+1}(\bar{t}_m)}{2}\right].
$$
There exist the following three situations:

\noindent$\spadesuit$
We first consider $j_m=2,\cdots,k-1$.
By the definition of $\Phi(t,x)$ in \eqref{dphi}, we have that
 \begin{align*}
 &\abs{\frac{\Phi(t+\bar{t}_m,y+\bar{y}_m+\gamma^0_{j_m}(t+\bar{t}_m))}{\Phi(\bar{t}_m,\bar{y}_m+\gamma^0_{j_m}(\bar{t}_m))}}
\\[3mm]
&=\frac{\abs{\bar{t}_m}^{\frac{3}{4}+\frac{\sigma}{8\sqrt{2}}}}
{\abs{\bar{t}_m+t}^{\frac{3}{4}+\frac{\sigma}{8\sqrt{2}}}\ }
\frac{\ \frac{1}{\big(\abs{y+\bar{y}_m+\gamma^0_{j_m}(t+\bar{t}_m)-\gamma_{l-1}^0(t+\bar{t}_m)}+1\big)^\alpha}
\ +\ \frac{1}{\big(\abs{y+\bar{y}_m+\gamma^0_{j_m}(t+\bar{t}_m)-\gamma_{l+1}^0(t+\bar{t}_m)}+1\big)^\alpha}\ }
{\frac{1}{\big(\abs{\bar{y}_m+\gamma^0_{j_m}(\bar{t}_m)-\gamma_{j_m-1}^0(\bar{t}_m)}+1\big)^\alpha}
\ +\
\frac{1}{\big(\abs{\bar{y}_m+\gamma^0_{j_m}(\bar{t}_m)-\gamma_{j_m+1}^0(\bar{t}_m)}+1\big)^\alpha}}
\\[3mm]
&\leq\frac{\abs{\bar{t}_m}^{\frac{3}{4}+\frac{\sigma}{8\sqrt{2}}}}{\abs{\bar{t}_m+t}^{\frac{3}{4}+\frac{\sigma}{8\sqrt{2}}}}
\Bigg[\frac{\left(\abs{\bar{y}_m+\gamma^0_{j_m}(\bar{t}_m)-\gamma_{j_m-1}^0(\bar{t}_m)}+1\right)^\alpha}{\left(\abs{y+\bar{y}_m+\gamma^0_{j_m}(t+\bar{t}_m)-\gamma_{l-1}^0(t+\bar{t}_m)}+1\right)^\alpha}
\\[3mm]
&\qquad\qquad+\frac{\left(\abs{\bar{y}_m+\gamma^0_{j_m}(\bar{t}_m)-\gamma_{j_m+1}^0(\bar{t}_m)}+1\right)^\alpha}{\left(\abs{y+\bar{y}_m+\gamma^0_{j_m}(t+\bar{t}_m)-\gamma_{l+1}^0(t+\bar{t}_m)}+1\right)^\alpha}
\Bigg]
\\[3mm]
&\leq C\frac{\abs{\bar{t}_m}^{\frac{3}{4}+\frac{\sigma}{8\sqrt{2}}}}{\abs{\bar{t}_m+t}^{\frac{3}{4}+\frac{\sigma}{8\sqrt{2}}}}
\frac{\big[\eta(t_m)\big]^\alpha}{\big[\eta(t+t_m)\big]^\alpha}
= C\frac{\abs{\bar{t}_m}^{\frac{3}{4}+\frac{\sigma}{8\sqrt{2}}}}{\abs{\bar{t}_m+t}^{\frac{3}{4}+\frac{\sigma}{8\sqrt{2}}}}
\frac{\big[\ln(t_m)\big]^\alpha}{\big[\ln(t+t_m)\big]^\alpha}
\\[3mm]
&\leq C\frac{\left[\abs{\bar{t}_m+t}^{\frac{3}{4}+\frac{\sigma}{8\sqrt{2}}}+\abs{t}^{\frac{3}{4}+\frac{\sigma}{8\sqrt{2}}}\right]}{\abs{\bar{t}_m+t}^{\frac{3}{4}+\frac{\sigma}{8\sqrt{2}}}}
\frac{\big[\ln(t+t_m)\big]^\alpha+\big[\ln\big(1-\frac{t}{t+t_m}\big)\big]^\alpha}{\big[\ln(t+t_m)\big]^\alpha}
\\[3mm]
&\leq C\left[1+\frac{\abs{t}^{1+\alpha}}{\abs{t+\bar{t}_m}^{1+\alpha}}\right]\qquad \text{for all}\ (t,y)\in A^m_l\ \text{with}\ l=1,\cdots,k,
 \end{align*}
 where $\alpha>0$ and $C>0$ does only depend on $\alpha$, $k$ and $\sup\limits_{1 \leq j\leq k}\abs{a_l}$.

\noindent$\spadesuit$
Here is the case of $j_m=k$.
By the definition of $\Phi(t,x)$ in \eqref{dphi} again,  we have that
\begin{align*}
 &\abs{\frac{\Phi(t+\bar{t}_m,y+\bar{y}_m+\gamma^0_{j_m}(t+\bar{t}_m))}{\Phi(\bar{t}_m,\bar{y}_m+\gamma^0_{k}(\bar{t}_m))}}
 \\[3mm]&=\frac{\abs{\bar{t}_m}^{\frac{3}{4}+\frac{\sigma}{8\sqrt{2}}}}{\abs{\bar{t}_m+t}^{\frac{3}{4}+\frac{\sigma}{8\sqrt{2}}}}
\Bigg[\frac{\big(\abs{\bar{y}_m+\gamma^0_{k}(\bar{t}_m)-\gamma_{k-1}^0(\bar{t}_m)}+1\big)^\alpha}{\left(\abs{y+\bar{y}_m+\gamma^0_{k}(t+\bar{t}_m)-\gamma_{l-1}^0(t+\bar{t}_m)}+1\right)^\alpha}
\\[3mm]&\qquad+\frac{\left(\abs{\bar{y}_m+\gamma^0_{k}(\bar{t}_m)-\gamma_{k-1}^0(\bar{t}_m)}+1\right)^\alpha}{\left(\abs{y+\bar{y}_m+\gamma^0_{k}(t+\bar{t}_m)-\gamma_{l+1}^0(t+\bar{t}_m)}+1\right)^\alpha}
\Bigg]
\\[3mm] &\left(\text{using}\ \bar{y}_m+\gamma_k^0(\bar{t}_m)-\gamma^0_{k-1}(\bar{t}_m)\geq 0\ \text{and}\ \gamma_{l-1}^0\leq\gamma_{k-1}^0\right)\\[3mm]
&\leq \frac{\abs{\bar{t}_m}^{\frac{3}{4}+\frac{\sigma}{8\sqrt{2}}}}{\abs{\bar{t}_m+t}^{\frac{3}{4}+\frac{\sigma}{8\sqrt{2}}}}
\Bigg[\frac{\big(\abs{\bar{y}_m+\gamma^0_{k}(\bar{t}_m)-\gamma_{l-1}^0(\bar{t}_m)}+1\big)^\alpha}{\left(\abs{y+\bar{y}_m+\gamma^0_{k}(t+\bar{t}_m)-\gamma_{l-1}^0(t+\bar{t}_m)}+1\right)^\alpha}
\\[3mm]&\qquad+\frac{\left(\abs{\bar{y}_m+\gamma^0_{k}(\bar{t}_m)-\gamma_{k-1}^0(\bar{t}_m)}+1\right)^\alpha}{\left(\abs{y+\bar{y}_m+\gamma^0_{k}(t+\bar{t}_m)-\gamma_{k-1}^0(t+\bar{t}_m)}+1\right)^\alpha}
\Bigg]
\\[3mm]
&\leq C\frac{\abs{\bar{t}_m}^{\frac{3}{4}+\frac{\sigma}{8\sqrt{2}}}}{\abs{\bar{t}_m+t}^{\frac{3}{4}+\frac{\sigma}{8\sqrt{2}}}}\left(\abs{y}^\alpha+1+\ln\left[\frac{\bar{t}_m}{t+\bar{t}_m}\right]\right)
\\[3mm]&\leq C\frac{\abs{t}^2}{\abs{\bar{t}_m+t}^2}\left(\abs{y}^\alpha+1\right)\qquad \text{for all}\ (t,y)\in A^m_l\ \text{with}\ l=1,\cdots,k,
\end{align*}
 where $C>0$ does only depend on $\alpha$, $k$ and $\sup\limits_{1 \leq j\leq k}\abs{a_l}$.

\noindent$\spadesuit$
The last one is $j_m=1$.
By a similar argument in the case for $j_m=k$, we have that
\begin{align*}
 &\abs{\frac{\Phi(t+\bar{t}_m,y+\bar{y}_i+\gamma^0_{j_m}(t+\bar{t}_m))}{\Phi(\bar{t}_m,\bar{y}_i+\gamma^0_{j_m}(\bar{t}_m))}}
 \\[2mm]&=\frac{\abs{\bar{t}_m}^{\frac{3}{4}+\frac{\sigma}{8\sqrt{2}}}}{\abs{\bar{t}_m+t}^{\frac{3}{4}+\frac{\sigma}{8\sqrt{2}}}}
\Bigg[\frac{\big(\abs{\bar{y}_m+\gamma^0_{1}(\bar{t}_m)-\gamma_{2}^0(\bar{t}_m)}+1\big)^\alpha}{\left(\abs{y+\bar{y}_m+\gamma^0_{1}(t+\bar{t}_m)-\gamma_{l-1}^0(t+\bar{t}_m)}+1\right)^\alpha}
\\[2mm]&\qquad+\frac{\left(\abs{\bar{y}_m+\gamma^0_{1}(\bar{t}_m)-\gamma_{2}^0(\bar{t}_m)}+1\right)^\alpha}{\left(\abs{y+\bar{y}_m+\gamma^0_{1}(t+\bar{t}_m)-\gamma_{l+1}^0(t+\bar{t}_m)}+1\right)^\alpha}
\Bigg]
\\[2mm]&\leq C\frac{\abs{t}^2}{\abs{\bar{t}_m+t}^2}\left(\abs{y}^\alpha+1\right)\qquad \text{for all}\ (t,y)\in A^m_l\ \text{with}\ l=1,\cdots,k,
\end{align*}
 for some $C>0$ does only depend on $\alpha$, $k$ and $\sup\limits_{1 \leq j\leq k}\abs{a_l}$.

Combining the above arguments, we derive that
\begin{equation}\label{esp}
  \abs{\phi^m(t,y)}\leq C\left[\frac{\abs{t}^{1+\alpha}(\abs{y}^\alpha+1)}{\abs{\bar{t}_m+t}^{1+\alpha}}+1\right],
\end{equation}
for all $(t,y)\in A_{l}^m$ given in \eqref{dA}, where $C$ is a positive constant only depending on $\alpha$ and $k$.
Similarly, we have that
\begin{equation}\label{esp1}
   \abs{\phi^m_y(t,y)}\leq C\left[\frac{\abs{t}^{1+\alpha}(\abs{y}^\alpha+1)}{\abs{\bar{t}_m+t}^{1+\alpha}}+1\right]\ \ \text{and}\ \ \abs{\phi^m_{yy}(t,y)}\leq C\left[\frac{\abs{t}^{1+\alpha}(\abs{y}^\alpha+1)}{\abs{\bar{t}_m+t}^{1+\alpha}}+1\right],
\end{equation}
for all $(t,y)\in A_{l}^m$, where $C$ is a positive constant which only depends on $\alpha$, $k$ and $\sup\limits_{1 \leq j\leq k}\abs{a_l}$.

We notice that $$\cup_{l=1}^kA_{l}^m=(s_m-\bar{t}_m,0]\times{\mathbb R}.$$
Owing to \eqref{esp} and \eqref{esp1}, we have that $\phi^m\rightarrow \phi $ in $L_{\text{loc}}^2[(-\bar{\ell},0); H^2_{\text{loc}}({\mathbb R})]$ up to a subsequence, where $\bar{\ell}=\ell$ or $\infty$.

For \textbf{Case $2$}:$\lim\limits_{m\rightarrow\infty}s_m-\bar{t}_m=-\ell$ with some constant $\ell>0$,
we have $\phi^m\rightarrow\phi$ locally uniformly, and $\phi$ satisfies
\begin{equation*}
\left\{
\begin{array}{l}
\phi_{t}=-\phi_{yyyy} \,+\, 2W''(1)\phi_{yy} \,-\, [W''(1)]^2\phi\qquad \text{in}\ (-\ell,0]\times{\mathbb R},\\[3mm]
\phi(-\ell,y)=0\qquad \text{for all}\ y\in{\mathbb R}.
\end{array}
\right.
\end{equation*}
By Proposition \ref{prop8}, the above equation has a unique solution $\phi\equiv0$.
However, according to \eqref{eqj.2}, we have that $$\abs{\phi(0,0)}=\lim\limits_{i\rightarrow\infty}\abs{\phi^m(0,0)}\geq\frac{1}{4},$$
which leads to a contradiction. Hence, we have that $\bar{\ell}=\infty$, which is \textbf{Case $1$}.

\textbf{Step two.}
{\em We claim that
\begin{align}\label{des}
\phi(t,y)\equiv0\qquad \text{for all} \ (t,y)\in(-\infty,0)\times{\mathbb R}.
\end{align}
}
 This result contradicts with $\abs{\phi(0,0)}>\frac{1}{4}$, thus we can derive that Lemma \ref{lem5} holds.

We now verify the validity of the above conclusion in \eqref{des}.
Recall that $\sqrt{W''(1)}=\sqrt{2}$, we consider the following parabolic equation
\begin{equation*}
  u_t=-u_{yyyy}+4 u_{yy}\qquad \text{for}\ (t,y)\in (0,+\infty)\times{\mathbb R}.
\end{equation*}
Using the Fourier transformation for $y$, the above equation has a formula solution
\begin{equation}\label{11}
Q(t,y)=\frac{1}{2\pi}\int_{\mathbb R} \exp\left\{-t(\abs{\xi}^4+4\abs{\xi}^2)\right\}e^{i\xi y}{\mathrm d}\xi
\end{equation}
Hence for any $\widetilde{T}>0$ and ${\tilde g}\in L^\infty\big((-\widetilde{T},0)\times{\mathbb R}\big)$, the following initial value problem
\begin{equation*}
  u_t=-u_{yyyy}+4 u_{yy}-4u+{\tilde g}(t,y)\quad \forall\,(t,y)\in (-\widetilde{T},0)\times{\mathbb R}, \quad  u(-\widetilde{T},y)=0, \quad \forall\, y\in{\mathbb R},
\end{equation*}
has a solution of the form
\begin{equation}\label{f1}
u(t,y)=\frac{1}{2\pi}\int_{0}^{t-\widetilde{T}}\int_{{\mathbb R}}e^{-4\tau}Q(\tau,x){\tilde g}(t-\tau,y-x){\mathrm d}x{\mathrm d}\tau, \qquad \forall\, (t,y)\in(-\widetilde{T},0)\times{\mathbb R}.
\end{equation}

Thus according to formula \eqref{f1} and equation \eqref{eqj.2}, we have that
\begin{equation}\label{dfs}
\phi^m(t,y)=\int_{0}^{t-s_m+\bar{t}_m}\int_{{\mathbb R}}e^{-\alpha^4\tau}Q(\tau,x){\tilde g}_m(\phi^m(t-\tau,y-x),t-\tau,y-x)
{\mathrm d}x{\mathrm d}\tau,
\end{equation}
for all $(t,y)\in(s_m-\bar{t}_m,0)\times{\mathbb R}$, where the functions ${\tilde g}_m(\phi^m,t,y)$ are defined as follows
\begin{equation*}\begin{aligned}
{\tilde g}_m(\phi^m,t,y)
:=&2\Big[W''(\bar{z}(t,y))-W''(1)\Big]\phi^m_{yy}
\,+\, 2W'''(\bar{z}(t,y))\bar{z}'(t,y)\phi^m_y
\\[2mm]
&\,-\,\Big([W''(\bar{z}(t,y))]^2-[W''(1)]^2\Big)\phi^m
\,+\, \Big[W''(\bar{z}(t,y))\Big]''\phi^m
\\[2mm]
&+W'''(\bar{z}(t,y))\Big(\bar{z}''(t,y)-W'(\bar{z}(t,x))\Big)\phi^m \,+\,g_m(t,y).
\end{aligned}\end{equation*}
Using \eqref{eq3.8}, \eqref{esp}, \eqref{esp1} and $\bar{z}(t,y)=z(t+\bar{t}_m,y+\bar{y}_i+\gamma^0_{j_m}(t+\bar{t}_m))$, as the same as the proof of Lemma \ref{lem10}, we have that

\begin{align*}
 & \abs{{\tilde g}_m(\phi^m(t,y),t,y)}
 \\
 &\leq C\left[\frac{\abs{t}^{1+\alpha}(\abs{y}^\alpha+1)}{\abs{\bar{t}_m+t}^{1+\alpha}}+1\right]
  \Big{[}\exp\big\{-\sqrt{2}/2\abs{y+\bar{y}_m}\big\}
  +\|g_{m}\|_{C_{\Phi}((s_m,t_m)\times{\mathbb R})}\Big{]}
\end{align*}
for all $(t,y)\in (s_m-\bar{t}_m,0]\times{\mathbb R}$. Hence we have that
\begin{equation}\label{12}
 \abs{{\tilde g}_m(\phi^m(t,y),t,y)}\leq
  C\left[\frac{\abs{t}^{1+\alpha}(\abs{y}^\alpha+1)}{\abs{s_m}^{1+\alpha}}+1\right]\Big{[}\frac{1}{1+\abs{y+\bar{y}_m}^{2\alpha}}
  +\|g_{m}\|_{C_{\Phi}((s_m,t_m)\times{\mathbb R})}\Big{]}
\end{equation}
  for all $(t,y)\in(s_m-\bar{t}_m,0)\times{\mathbb R}$, where $C$ only depends on $\alpha$
  and $k$.

 Moreover a straightforward shift of contour gives that there exists a positive constant $C$ such that
$$
\abs{Q(\tau,y)}\leq C \tau^{-\frac{1}{4}}\exp\big\{-\tau^{-1/4}\abs{y}\big\},
$$
or see the inequality $(3.2)$ in \cite{BKT}, where $Q(\tau,y)$ is defined in \eqref{11}.
Thus using \eqref{dfs}, \eqref{12} and the above inequality, we have that
  \begin{align}\label{13}
&\abs{\phi^m(t,y)}\leq C\int_0^{t-s_m+\bar{t}_m}\int_{{\mathbb R}}\abs{Q(\tau,x)}\abs{{\tilde g}_m(\phi^m(t-\tau,y-x),t-\tau,y-x)}\,{\mathrm d}x {\mathrm d}\tau
\nonumber\\[2mm]
&\leq C\int_0^{+\infty}\int_{{\mathbb R}}\tau^{-\frac{1}{4}}\exp\big\{-\tau^{-1/4}\abs{x}-\alpha^4\tau\big\}
\left[\frac{\abs{t-\tau}^{1+\alpha}(\abs{y-x}^\alpha+1)}{\abs{s_m}^{1+\alpha}}+1\right]
\nonumber\\[2mm]
&\qquad\qquad \qquad\times\Bigg{[}\frac{1}{1+\abs{y-x+\bar{y}_m}^{2\alpha}}
  +\|g_{m}\|_{C_{\Phi}}\Bigg{]}\,{\mathrm d}x {\mathrm d}\tau
  \nonumber\\[2mm]
  &
  \leq C\int_0^\infty\int_{\mathbb R} \exp\big\{-\abs{x}-\alpha^4\tau\big\}\left[\frac{\abs{t-\tau}^{1+\alpha}(\abs{y-\tau^{1/4}x}^\alpha+1)}{\abs{s_m}^{1+\alpha}}+1\right]
\nonumber\\[2mm]
&\qquad\qquad \qquad\times\Bigg{[}\frac{1}{1+\abs{y-\tau^{1/4}x+\bar{y}_m}^{2\alpha}}
  +\|g_{m}\|_{C_{\Phi}}\Bigg{]}\,{\mathrm d}x {\mathrm d}\tau.
\end{align}
Thus by   Lebesgue's dominated convergence theorem,  $s_m\rightarrow+\infty$, $\abs{\bar{y}_m} \rightarrow +\infty$ and
$$\|g_{m}\|_{C_{\Phi}((s_i,t_i)\times{\mathbb R})}\rightarrow 0$$
when $m$ goes to infinity in \eqref{eq8}, we can obtain \eqref{des}.

\end{proof}

\subsubsection{The proof of \textbf{Assertion}}\label{section3.2.2}
We will prove that the inequality (\ref{eq3.9}) holds by contradiction in the following four steps.

We assume that (\ref{eq3.9}) is not valid. Then there exist $j_n\in\{1,\cdots,k\}$ , a sequence $\{m_n\in \mathbb{N}\}$ with $\lim\limits_{n\rightarrow \infty} m_n=\infty$, $R>0$ and $\delta>0$ such that the following three situations at least one holds,
\begin{equation*}
  \left\|\frac{\psi^{m_n}_{yy}}{\Phi}\right\|_{L^\infty\left(A^{(s_{m_n},t_{m_n})}_{j_n,R}\right)}>2\delta,
  \quad \left\|\frac{\psi^{m_n}_y}{\Phi}\right\|_{L^\infty\left(A^{(s_{m_n},t_{m_n})}_{j_n,R}\right)}>2\delta
  \ \ \text{and}\ \
  \left\|\frac{\psi^{m_n}}{\Phi}\right\|_{L^\infty\left(A^{(s_{m_n},t_{m_n})}_{j_n,R}\right)}>2\delta.
\end{equation*}
Without loss of the generality, we assume that the first situation happens, that is   $$\left\|\frac{\psi^{m_n}_{yy}}{\Phi}\right\|_{L^\infty\left(A^{(s_{m_n},t_{m_n})}_{j_n,R}\right)}>2\delta.$$ For the other situations,
the following arguments are similar.
Thus there exists

$$
(\hat{t}_{m_n}, x_{m_n})\in A^{(s_{m_n},t_{m_n})}_{j_n,R}=\Big\{(\tau,x)\in (s_{m_n},t_{m_n})\times {\mathbb R}: \abs{x-\gamma^0_{j_n}(\tau)}<R+1\Big\}
$$
such that
\begin{equation}\label{eq3.12}
\abs{\frac{\psi^{m_n}(\hat{t}_{m_n},x_{m_n})}{\Phi(\hat{t}_{m_n},x_{m_n})}}>\delta>0.
\end{equation}

Let us introduce the following variable transformation
\begin{equation*}
  y=x-\gamma_{m_n}(t+\hat{t}_{m_n})
  \qquad \text{and}\qquad
  y_{m_n}=x_{m_n}-\gamma^0_{j_n}(\hat{t}_{m_n}),
\end{equation*}
and set
\begin{equation}\label{eq3.14}
\widehat{\psi}^{m_n}(t,y):=\frac{\psi^{m_n}\big(t+\hat{t}_{m_n},\, y+y_{m_n}+\gamma^0_{j_n}(t+\hat{t}_{m_n})\big)}
  {\Phi\big(\hat{t}_{m_n},\, y_{m_n}+\gamma^0_{j_n}(\hat{t}_{m_n})\big)}.
\end{equation}
Hence, there exists $y_0$ satisfies $\lim\limits_{n\rightarrow\infty}y_{m_n}=y_0$ and $\abs{y_0}<R+2$.
According to (\ref{eq3.5}), we have that $\widehat{\psi}^{m_n}$ satisfies the following problem
\begin{align}\label{eq3.10}
&\widehat{\psi}^{m_n}_t=-\widehat{\psi}^{m_n}_{yyyy}+2W''(\hat{z}(t,y))\widehat{\psi}^{m_n}_{yy}+2W'''(\hat{z}(t,y))\partial_y\hat{z}(t,y)\widehat{\psi}^{m_n}
_y-[W''(\hat{z}(t,y))]^2\widehat{\psi}^{m_n}
\nonumber\\[3mm] &\hspace{1.5cm}+\Big{\{}\partial_{yy}[W''(\hat{z}(t,y))]+W'''(\hat{z}(t,y))\big[\partial_{yy}\hat{z}(t,y)-W'(\hat{z}(t,x))\big]\Big{\}}\widehat{\psi}^{m_n}+\hat{g}^{m_n}(t,y) \nonumber\\[3mm]
&\hspace{1.8cm}\text{in}\ \ \Gamma^{s_{m_n},t_{m_n}},
\\[3mm]
&\abs{\widehat{\psi}^{m_n}(0,0)}\geq\delta,\qquad
\widehat{\psi}^{m_n}(s_{m_n}-\hat{t}_{m_n},y)=0\qquad  \text{in}\ {\mathbb R},
\end{align}
where the functions $\hat{z}(t,x):=z(t+\hat{t}_{m_n},y+y_{m_n}+\gamma^0_{j}(t+\hat{t}_{m_n}))$, $$\hat{g}^{m_n}(t,y):=\frac{g(t+\hat{t}_{m_n},y+y_{m_n}+\gamma^0_{j}(t+\hat{t}_{m_n}))+\partial_t\gamma^0_{j_n}(t+\hat{t}_{m_n})}{\Phi(\bar{t}_{m_n},y_{m_n}+\gamma^0_{j_n}(\hat{t}_{m_n}))}$$ and the set
$\Gamma^{s_{m_n},t_{m_n}}:=\left\{(t,y)\in (s_{m_n}-\hat{t}_{m_n},t_{m_n}-\hat{t}_{m_n})\times (-\infty,+\infty)\right\}$.
As the same as previous part, $\lim\limits_{n\rightarrow\infty}(s_{m_n}-\hat{t}_{m_n})=\hat{\ell}$, where $\hat{\ell}$ equals to $-\infty$ or a negative number.

The rest of the proof will be divided into four steps.
We will show $\hat{\ell}=-\infty$ in \textbf{Step} $\mathbf{1}$.
Then the going further analysis will lead to a contradiction in \textbf{Step} $\mathbf{4}.$

\textbf{Step} $\mathbf{1}.$
We have that $\widehat{\psi}^{m_n}\rightarrow\widehat{\psi}$ locally uniformly, which $\widehat{\psi}$ satisfies that
$\abs{\widehat{\psi}(0,0)}>\delta>0$ and the following equation
\begin{align}\label{eq3.13}
\nonumber\widehat{\psi}_t=&-\widehat{\psi}_{yyyy}+2W''(\omega(y+y_0))\widehat{\psi}_{yy}+2W'''(\omega(y+y_0))\omega'(y+y_0)\widehat{\psi}
_y-[W''(\omega(y+y_0)]^2\widehat{\psi}
\\[2mm] \nonumber &+\Big{[}W''''(\omega(y+y_0))(\omega'(y+y_0))^2+W'''(\omega(y+y_0))\omega''(y+y_0)\Big{]}\widehat{\psi},
\\[2mm]
&\ \ \  \text{in} \ (-\infty,0]\times(-\infty,+\infty).
\end{align}
In fact, according to (\ref{eq3.8}) and (\ref{eq3.14}) and the definition of $\gamma^0_j$ in \eqref{degama0}, similar as \eqref{esp}, we have that
\begin{align}
  \abs{\widehat{\psi}^{m_n}(t,y)}
  =&\abs{\frac{\psi^{m_n}(t+\hat{t}_{m_n},y+y_{m_n}+\gamma^0_{j_n}(t+\hat{t}_{m_n}))}
{\Phi(\hat{t}_{m_n},y_{m_n}+\gamma^0_{j_n}(\hat{t}_{m_n}))}}
\nonumber\\[2mm]
\leq&\left\|\psi^{m_n}\right\|_{C_{\Phi}}\abs{\frac{\Phi(t+\hat{t}_{m_n},y+y_{m_n}+\gamma^0_{j_n}(t+\hat{t}_{m_n}))}
{\Phi(\hat{t}_{m_n},y_{m_n}+\gamma^0_{j_n}(\hat{t}_{m_n}))}}
\label{eq3.15}
\\[2mm]
\leq &C\big(\alpha,k, R,\sup\limits_{1 \leq j\leq k}a_l\big) \left[\frac{\abs{t}^{1+\alpha}(\abs{y}^\alpha+1)}{\abs{\hat{t}_{m_n}+t}^{1+\alpha}}+1\right],\qquad
\forall\, (t,y)\in B^n_{l}\ \text{and}\ l=1,\cdots,k,
\nonumber
\end{align}
where the set
\begin{equation*}\begin{aligned}
B_{l}^n£º=\Big{\{}(t,y)&\in (s_{m_n}-\hat{t}_{m_n},0]\times{\mathbb R}:\ \ \gamma^0_{l-1}(t+\hat{t}_{m_n}) \leq y+y_{m_n}+\gamma^0_{j_n}(t+\hat{t}_{m_n})\leq \gamma^0_{l}(t+\hat{t}_{m_n})\Big{\}}.
\end{aligned}\end{equation*}
According to the definition of $\gamma^0_j$ in (\ref{degama0}), there holds that
\begin{equation}\label{eq3.17}
\cup_{l=1}^kB_{l}^n=(s_{m_n}-\hat{t}_{m_n},0]\times{\mathbb R}.
\end{equation}
Similarly, hence we have that
\begin{equation}\label{f4}
  \abs{\widehat{\psi}^{m_n}(t,y)}+\abs{\widehat{\psi}^{m_n}_y(t,y)}+\abs{\widehat{\psi}^{m_n}_{yy}(t,y)}\leq C\left[\frac{\abs{t}^{1+\alpha}(\abs{y}^\alpha+1)}{\abs{\hat{t}_{m_n}+t}^{1+\alpha}}+1\right] ,
\end{equation}
 for all $(t,y)\in (s_{m_n}-\hat{t}_{m_n},0]\times{\mathbb R}$, where $C$ depends on $k, \alpha,R$ and $\sup\limits_{1 \leq j\leq k}\abs{a_l}$.
Since $\hat{t}_{m_n}\rightarrow -\infty$ and $\gamma_{l}(t+\hat{t}_{m_n})\rightarrow+\infty$ as $n$ goes to infinity for any $l>\frac{k+1}{2}$, and
$\gamma_{l}(t+\hat{t}_{m_n})\rightarrow-\infty$ with $l<\frac{k+1}{2}$. Thus  by (\ref{eq3.8}), (\ref{eq3.14}),
(\ref{eq3.15}) and $y_{m_n}\rightarrow y_0$, we can get the limiting equation \eqref{eq3.13}.

If $\hat{\ell}\in (-\infty, 0)$, we have $\hat{\psi}(\hat{\ell},y)=0$. According to Proposition \ref{prop8}, we have that equation \eqref{eq3.13} has unique solution $\hat{\psi}(t,y)\equiv0$, which contradicts with $\abs{\hat{\psi}(0,0)}>0$. Hence, it only happens the case that $\hat{\ell}=-\infty$.

\textbf{Step} $\mathbf{2}.$
We will prove that the following orthogonality condition for $\widehat{\psi}$ holds.
\begin{equation}\label{eq3.16}
  \int_{{\mathbb R}}\widehat{\psi}(t,y)\omega'(y+y_0)\,{\mathrm d}y=0\qquad \text{for all} \ t\in(-\infty,0].
\end{equation}
Indeed, according to (\ref{eq3.6}) and (\ref{eq3.14}), we have that
\begin{equation*}
  0= \frac{1}{\Phi(\hat{t}_{m_n},y_{m_n}+\gamma_{j_n}(\hat{t}_{m_n}))}\int_{\mathbb R} \psi^{m_n}(t,x)\omega'(x-\gamma_{j_n}(t))\,{\mathrm d}x=\int_{{\mathbb R}}\widehat{\psi}^{m_n}(t,y)\omega'(y+y_{m_n})\,{\mathrm d}y.
\end{equation*}
According to (\ref{eq3.15}), (\ref{eq3.17}) and \eqref{asymptotic}, we can get that
\begin{equation*}\begin{aligned}
&\abs{\widehat{\psi}^{m_n}(t,y)\omega'(y+y_{m_n})}\leq C\big(k,\alpha,R, \sup_{1 \leq j\leq k}\abs{a_l}\big)\left[\frac{\abs{t}^{1+\alpha}(\abs{y}^\alpha+1)}{\abs{\hat{t}_{m_n}+t}^{1+\alpha}}+1\right]
\exp\{-\sqrt{2}\abs{y+y_{m_n}}\}.
\end{aligned}\end{equation*}
Since $t+\hat{t}_{m_n}\geq s_{m_n}$, $s_{m_n}\rightarrow+\infty$ and $y_{m_n}\rightarrow y_0$ as $n$ goes to infinity,
using Lebesgue's dominated convergence theorem, we can  obtain \eqref{eq3.16}.

\textbf{Step} $\mathbf{3}.$
In this step we will prove the following  decay of $\widehat{\psi}(t,y)$:

{\em
There exists $C=C\big(\alpha,k, R, \sup\limits_{1 \leq j\leq k}\abs{a_l}\big)>0$ such that
\begin{equation}\label{eq3.18}
  \abs{\widehat{\psi}(t,y)}\leq Ce^{-\sqrt{2}/2\abs{y}}, \qquad \forall\, (t,y)\in (-\infty,0]\times{\mathbb R}.
\end{equation}
}
In fact, for any $(t,y)\in B_{l}^n$, by the definition of $\gamma^0_j$ in \eqref{degama0},  in view of the  proof of (\ref{eq3.15}), we have
\begin{align}\label{f3}
  \nonumber\abs{\hat{g}^{m_n}(t,y)}&=\abs{\frac{g_{m_n}(t+\hat{t}_{m_n},y+y_{m_n}+\gamma_{j_n}(t+\hat{t}_{m_n}))}
{\Phi(\hat{t}_{m_n},y_{m_n}+\gamma_{j}(\hat{t}_{m_n}))}}
\\&\leq C \|g_{m_n}\|_{C_{\Phi}((s_{m_n},\bar{t}_{m_n})\times{\mathbb R})}\left[\frac{\abs{t}^{1+\alpha}(\abs{y}^\alpha+1)}{\abs{\hat{t}_{m_n}+t}^{1+\alpha}}+1\right],
\end{align}
 for all $(t,y)\in (s_{m_n}-\hat{t}_{m_n}, 0]\times{\mathbb R}$, where $C$ depends on $k,\alpha, R$ and $\sup\limits_{1 \leq j\leq k}\abs{a_l}$.

By the formula \eqref{f1}, the solution of equation \eqref{eq3.10} has the form
\begin{equation}\label{f2}
  \widehat{ \psi}^{m_n}(t,y)=\int_{0}^{t-s_{m_n}+\hat{t}_{m_n}}\left\{\int_{{\mathbb R}}e^{-4\tau}Q(\tau,x)\hat{f}^{m_n}\big{(}\hat{\psi}^{m_n}(t-\tau,y-x),t-\tau,y-x\big{)}\,{\mathrm d}x\right\}{\mathrm d}\tau,
\end{equation}
for any $(t,y)\in(s_{m_n}-\hat{t}_{m_n},0)\times{\mathbb R}$, where the function $\hat{f}^{m_n}$ is given by
\begin{equation*}\begin{aligned}
\hat{f}^{m_n}=&2\big[W''(\hat{z}(t,y))-W''(1)\big]\widehat{\psi}^{m_n}_{yy}
+2W'''(\hat{z}(t,y))\partial_y\hat{z}(t,y)\widehat{\psi}^{m_n}_y
\\[2mm]
&+\Big{\{}\partial_{yy}\big[W''(\hat{z}(t,y))\big] +W'''(\hat{z}(t,y))\big[\partial_{yy}\hat{z}(t,y)-W'(\hat{z}(t,x))\big]\Big{\}}\widehat{\psi}^{m_n}
\\[2mm]
&-\left(\big[W''(\hat{z}(t,y))\big]^2-\big[W''(1)\big]^2\right)\widehat{\psi}^{m_n}+\hat{g}^{m_n}(t,y),
\end{aligned}\end{equation*}
and $\hat{z}(t,x)=z\big(t+\hat{t}_{m_n}, y+y_{m_n}+\gamma^0_{j_n}(t+\hat{t}_{m_n})\big)$.

According to \eqref{asymptotic} and Lemmas \ref{esxi}-\ref{lem33}, \eqref{f3}, \eqref{eq3.15} and \eqref{f4}, we have that
\begin{equation*}
  \abs{\hat{f}^{m_n}}\leq C\Big{(}\|g_{m_n}\|_{C_{\Phi}((s_{m_n},\bar{t}_{m_n})\times{\mathbb R})}+\exp\big\{-2\sqrt{2}/3\abs{y+y_{m_n}}\big\}\Big{)}
  \left[\frac{\abs{t}^{1+\alpha}(\abs{y}^\alpha+1)}{\abs{\hat{t}_{m_n}+t}^{1+\alpha}}+1\right],
\end{equation*}
for all $(t,y)\in (s_{m_n}-\hat{t}_{m_n}, 0)\times{\mathbb R}$, where $C$ depends on $k,\alpha,R$ and $\sup\limits_{1 \leq j\leq k}a_l$.
Thus by \eqref{f2}, similar arguments as \eqref{13} and $\abs{y_{m_n}}\leq R+2$, we have
\begin{equation*}
  \abs{\hat{\psi}^{m_n}(t,y)}\leq C\Big{(}
\exp\big\{-(\sqrt{2}/2)\abs{y}\big\}
+\|g_{m_n}\|_{C_{\Phi}((s_{i_m},\bar{t}_{i_m})\times{\mathbb R})}\Big{)}\left[\frac{\abs{t}^{1+\alpha}(1+\abs{y}^\alpha)}{s_{m_n}}+1\right].
\end{equation*}
where $C$ depend on $\alpha,k,R$ and $\sup_{1 \leq j\leq k}a_l$. Then let $n$ goes to infinity, by \eqref{eq3.8}, we can get the desire result since $\|g_{m_n}\|_{C_{\Phi}((s_{m_n},\bar{t}_{m_n})\times{\mathbb R})}\rightarrow0$ and $s_{m_n}\rightarrow+\infty$.

\textbf{Step} $\mathbf{4}.$
To proceed further, we need  the following result, which is Lemma $3.6$ in \cite{ly}.
\begin{lem}\label{lem2}
Considering the Hilbert space
$$H=\Big{\{}\phi(y)\in H^2({\mathbb R}):\int_{{\mathbb R}}\phi(y)\omega'(y)\,{\mathrm d}y=0\Big{\}},$$
then the following inequality is valid
\begin{equation}\label{eqi}
\int_{\mathbb R}\abs{\phi''(y)-W''(\omega)\phi(y)}^2\,{\mathrm d}y\geq c\int_{\mathbb R}\abs{\phi(y)}^2\,{\mathrm d}y , \qquad \forall\, \phi\in H,
\end{equation}
where $c>0$ is an uniform constant.
\qed
\end{lem}

 Multiplying (\ref{eq3.13}) by $\hat{\psi}(y)$ and integrate with respect to $y$, and using \eqref{eqi}, we have
\begin{align*}
0\,=\,&\frac{1}{2}\int_{\mathbb R} \big(\widehat{\psi}\,\big)_t^2\,{\mathrm d}y
\,+\,
\int_{\mathbb R}\abs{\widehat{\psi}_{yy}-W''\big(\omega(y+y_0)\big)\widehat{\psi}(t,y)}^2\,{\mathrm d}y
\\[2mm]
\,\geq\,
&\frac{1}{2}\frac{{\mathrm d}}{{\mathrm d}t}\int_{\mathbb R}\Big(\widehat{\psi}\Big)^2\,{\mathrm d}y
\,+\,
c\int_{\mathbb R}\abs{\widehat{\psi}(t,y)}^2\,{\mathrm d}y
\end{align*}
According to the Gronwall's inequality, we get that
\begin{equation*}
  \widetilde{a}(t)\geq \widetilde{a}(0)e^{-2ct}, \qquad \forall\, t<0,
\end{equation*}
where $c>0$ is given by \eqref{eqi} and the function $$\widetilde{a}(t):=\int_{\mathbb R}\abs{\widehat{\psi}(t,y)}^2\,{\mathrm d}y,$$ which is a contradiction since (\ref{eq3.18}). The proof is completed.\qed
\bigskip

\subsubsection{}\label{section3.2.3}
We consider the following projection problem:
\begin{equation}\label{eq3.19}
  \left\{
\begin{array}{l}
\psi_t=F'(z(t,x))[\psi]+f(t,x)-\sum\limits_{j=1}^kc_j(t)\omega'\big(x-\gamma_j(t)\big)\qquad    \text{for}\ (t,x)\in [T,s)\times{\mathbb R},\\
\psi(T,x)=0\qquad   \text{for}\ x\in{\mathbb R},
 \end{array}
\right.
\end{equation}
where $f\in C_{\Phi}((T, s)\times {\mathbb R})$ and $c_j(t)$ satisfies the following system
\begin{equation}\label{eq3.20}
  \begin{aligned}
  &\sum_{j=1}^kc_j(t)\int_{{\mathbb R}}\omega'\big(x-\gamma_j(t)\big)\omega'\big(x-\gamma_i(t)\big)\,{\mathrm d}x\\[2mm]&=
  \int_{{\mathbb R}} F'(z(t,x))[\psi]\omega'\big(x-\gamma_i(t)\big)\,{\mathrm d}x
  -\gamma_i'(t)\int_{{\mathbb R}} \psi(t,x)\omega''\big(x-\gamma_i(t)\big)\,{\mathrm d}x\\[2mm]&
  \ \ \ +\int_{{\mathbb R}}f(t,x)\omega'(x-\gamma_i(t)\,{\mathrm d}x\qquad \text{for all}\
   i=1,\cdots,k\ \text{and}\ t>T.
  \end{aligned}
\end{equation}
We note that if $\psi$ is a solution of (\ref{eq3.19}) and these $c_j(t)$ satisfy (\ref{eq3.20}), by integration by parts, then
$\psi$ satisfies the following the orthogonality conditions
\begin{equation}\label{eq3.21}
    \int_{\mathbb R} \psi(t,x)\omega'\big(x-\gamma_i(t)\big)\,{\mathrm d}x=0, \qquad \forall\, i=1,\cdots,k\ \text{and}\ t\in(T,s).
\end{equation}
For system (\ref{eq3.20}), we have the following result:
\begin{lem}\label{lem6}
Let $T>1$ big enough,  $f\in C_{\Phi}((T,s)\times {\mathbb R})$
and $\psi\in C_{\Phi}((T, s)\times {\mathbb R})$, then there exist $c_j(t)$ with $j=1,\cdots,k$ such that system $(\ref{eq3.20})$ holds.
Furthermore the following estimates are valid,
\begin{equation*}\begin{aligned}
  \abs{c_j(t)}\leq C\Bigg[\Big{(}\frac{1}{\abs{t}}\Big{)}^{\frac{5}{4}+\frac{\sigma}{8\sqrt{2}}}
\|\psi\|_{C_{\Phi}((T,s)\times{\mathbb R})}
\,+\,
\Big{(}\frac{1}{\abs{t}}\Big{)}^{\frac{3}{4}+\frac{\sigma}{8\sqrt{2}}}\|f\|_{C_{\Phi}((T, s)\times{\mathbb R})}\Bigg],
\end{aligned}
\end{equation*}
and
\begin{equation*}\begin{aligned}
 \abs{\frac{c_j(t)\omega^{(l)}\big(x-\gamma_i(t)\big)}{\Phi(t,x)}}
 \,\leq\,
 &C\frac{1}{\big[\log\abs{t}\big]^\alpha}\Bigg[\frac{1}{\abs{t}^{\frac{1}{2}}}
\|\psi\|_{C_{\Phi}((T,s)\times{\mathbb R})}
\,+\,
\|f\|_{C_{\Phi}((T,s)\times{\mathbb R})}\Bigg],\
\end{aligned}
\end{equation*}
for any $ l=1,2,3$, $i=1,\cdots,k$ and  $t\in[T,s]$, where $\sigma\in(0,\sqrt{2})$ and $\alpha>1$, $C$ is a positive constant which does not depend on $s, t, T$, $\vec{h}$ and $\psi$.
\end{lem}
\begin{proof}  We first consider the left side of system $(\ref{eq3.20})$. Set
\begin{equation*}
  d_{ij}(t):=\int_{{\mathbb R}}\omega'\big(x-\gamma_j(t)\big)\omega'\big(x-\gamma_i(t)\big)\,{\mathrm d}x \qquad \text{with}\ i,j=1,\cdots,k.
\end{equation*}
Then by \eqref{asymptotic}, we have
\begin{equation*}
\begin{aligned}
d_{ij}(t)&=d_{ji}(t)=
\int_{{\mathbb R}}\omega'(x)\omega'\big(x+\gamma_j(t)-\gamma_i(t)\big)\,{\mathrm d}x=
\left\{
\begin{array}{l}
\int_{{\mathbb R}}[\omega'(x)]^2\,{\mathrm d}x\qquad   \text{if}\ i=j,\\[3mm]
O\Big{(}e^{-\sqrt{2}/2\abs{\gamma_j(t)-\gamma_i(t)}}\Big{)}\qquad  \text{if}\ i\neq j.
 \end{array}
\right.
\end{aligned}
\end{equation*}
Hence system $(\ref{eq3.20})$ is nearly diagonal if we choose $T$ large enough.

By the definition of $\Phi$ in (\ref{dphi}) and $\alpha>1$,  we have the following fact that
\begin{equation}\label{eq3.22}\begin{aligned}
  \int_{\mathbb R}\Phi(t,x)\,{\mathrm d}x
  \leq C\frac{\log\abs{t}}{t^{\frac{3}{4}+\frac{\sigma}{8\sqrt{2}}}}
  \qquad \text{and}\qquad
  \int_{\mathbb R}\Phi(t,x)\omega'\big(x-\gamma_j(t)\big)\,{\mathrm d}x
  \leq C\frac{1}{t^{\frac{3}{4}+\frac{\sigma}{8\sqrt{2}}}},
\end{aligned}\end{equation}
for all $j=1,2,\cdots,k$, where $\sigma\in(0,\sqrt{2})$.

Moreover, using integration by parts,  we have that
\begin{align*}
&\int_{{\mathbb R}} F'(z(t,x))[\psi]\omega'\big(x-\gamma_i(t)\big)\,{\mathrm d}x
\\[2mm]&=-\int_{{\mathbb R}}W''(z(t,x))\psi\Big{[}W''(z(t,x))-W''\big(\omega\big(x-\gamma_i(t)\big)\big)\Big{]}\omega'\big(x-\gamma_i(t)\big)\,{\mathrm d}x\\[2mm]&
\ \ \ +\int_{{\mathbb R}}\psi \partial_{xx}\Big{\{}W''(z(t,x))\Big{[}W''(z(t,x))-W''\big(\omega\big(x-\gamma_i(t)\big)\big)\Big{]}\omega'\big(x-\gamma_i(t)\big)\Big{\}}\,{\mathrm d}x
\\[2mm]&\ \ \ +\int_{{\mathbb R}}W'''(z(t,x))\Big{[}\partial_{xx}z(t,x)-W'(z(t,x))\Big{]}\psi\omega'\big(x-\gamma_i(t)\big)\,{\mathrm d}x.
\end{align*}
Thus according to Taylor expansion and \eqref{eq3.22},  Lemmas \ref{esxi}-\ref{lem33}, we can show that
\begin{align}\label{eq3.24}
  \nonumber&\abs{\int_{{\mathbb R}} F'(z(t,x))[\psi]\omega'\big(x-\gamma_i(t)\big)\,{\mathrm d}x}
  \\[2mm]
  &\nonumber
\leq C\|\psi\|_{C_\Phi((T,s)\times{\mathbb R})}
\Bigg{\{}\int_{\mathbb R}\Phi(t,y+\gamma_i)\omega'(x)\sum_{j\neq i}\Big{(}\abs{\omega''(y+\gamma_i-\gamma_j)}+\abs{\omega'''(x+\gamma_i-\gamma_j)}\Big{)}\,{\mathrm d}x
\\[2mm]
&\nonumber\qquad
+\int_{\mathbb R}\Phi(t,y+\gamma_i)\omega'(x)\Bigg[\sum_{j< i}\Big{|}\omega(x+\gamma_i-\gamma_j)-1\Big{|}
+\sum_{l=1}^{k-1}\sum_{m=0}^2\abs{\partial_x^m\xi_l(t,x+\gamma_i)}\Bigg]\,{\mathrm d}x
\\[2mm]
&\nonumber\qquad+\int_{\mathbb R}\Phi(t,y+\gamma_i)\omega'(x)\Bigg[\sum_{j>i}\Big{|}\omega(x+\gamma_i-\gamma_j)+1\Big{|}
+\sum_{l=1}^{k}\sum_{m=0}^2\abs{\partial_x^m\widehat{\xi}_l(t,x+\gamma_i)}\Bigg]\,{\mathrm d}x\Bigg{\}}
\\[2mm]
&\nonumber
\leq C\|\psi\|_{C_\Phi((T,s)\times{\mathbb R})}\,\frac{1}{\abs{t}^{\frac{1}{2}}}\,\int_{\mathbb R}\Phi(t,y+\gamma_i)\omega'(x)\,{\mathrm d}x
\\ & \leq C\|\psi\|_{C_\Phi((T,s)\times{\mathbb R})}\,\frac{1}{\abs{t}^{\frac{5}{4}+\frac{\sigma}{8\sqrt{2}}}},
\end{align}
where $C$ is a positive constant which does not depend on $s, t$ and $T$.

Similarly, we have that
\begin{equation*}
  \begin{aligned}
  \abs{\gamma_i'(t)\int_{{\mathbb R}} \psi(t,x)\omega''\big(x-\gamma_i(t)\big)\,{\mathrm d}x}\leq C\|\psi\|_{C_\Phi((T,s)\times{\mathbb R})}\,
\Big{(}\frac{1}{\abs{t}}\Big{)}^{{\frac{7}{4}+\frac{\sigma}{8\sqrt{2}}}}
  \end{aligned}
\end{equation*}
and
\begin{equation*}
  \abs{\int_{{\mathbb R}}f(t,x)\omega'\big(x-\gamma_i(t)\big)\,{\mathrm d}x}\leq C\|f\|_{C_\Phi((T,s)\times{\mathbb R})}\,
\Big{(}\frac{1}{\abs{t}}\Big{)}^{{\frac{3}{4}+\frac{\sigma}{8\sqrt{2}}}}.
\end{equation*}
According to (\ref{eq3.20}) and the above estimates, we can get the first estimate of the lemma.

Using the following facts
\begin{equation*}
 \abs{\partial_x^l\omega(x)}\leq C\abs{\omega'(x)}, \ \ \ l=1,2,3,
 \qquad \text{and}\qquad
 \abs{\frac{\omega'\big(x-\gamma_i(t)\big)}{\Phi(t,x)}}\leq C\frac{\abs{t}^{\frac{3}{4}+\frac{\sigma}{8\sqrt{2}}}}{\big[\log\abs{t}\big]^\alpha},
\end{equation*}
for all $(t,x)\in (T,s)\times {\mathbb R}$, where $C$ is a constant which does not depend on $t$, $s$ and $T$, where $\Phi$ is defined by \eqref{dphi}. Thus we can get the second estimate.
\end{proof}

According to the above Lemma \ref{lem6}, we have
\begin{lem}\label{lem7}
 There exists a uniform constant $T_0>0$ such that for any $T>T_0$, and $f\in C_{\Phi}((T,s)\times{\mathbb R})$,
 there exists a unique solution $\psi^s$ of the problem $(\ref{eq3.19})$. Moreover,  $\psi^s$ satisfies the orthogonality condition $(\ref{eq3.21})$
with $t\in(T,s)$, and
\begin{equation}\label{eq3.27}
  \|\psi^T_{xx}\|_{C_{\Phi}((T,t)\times{\mathbb R})}+ \|\psi^T_x\|_{C_{\Phi}((s,t)\times{\mathbb R})}+ \|\psi^T\|_{C_{\Phi}((T,t)\times{\mathbb R})}\leq C\|f\|_{C_{\Phi}((T,t)\times{\mathbb R})},
\end{equation}
where $C$ is a uniform constant which does not depend on $T, t$.
\end{lem}

\begin{proof}We will use a fixed point argument to prove this lemma.
Let
$$ X^T:=\Big\{\psi\,:\, \|\psi_{xx}\|_{C_{\Phi}\big((T,T+1)\times{\mathbb R}\big)}+\|\psi_x\|_{C_{\Phi}\big((T,T+1)\times{\mathbb R}\big)}+\|\psi\|_{C_{\Phi}\big((T,T+1)\times{\mathbb R}\big)}<+\infty\Big\}.
$$
and consider the operator $\mathcal{A}^T: X^T\rightarrow X^T$ defined by
$$\mathcal{A}^T(\psi):=\mathbf{T}^T(f-C(\psi)),$$
where $\mathbf{T}^T(g)$ denotes the solution to equation \eqref{eq3.5} with $s=T+1$ and
$$
C(\psi)=\sum\limits_{l=1}^kc_l(t)\omega'(x-\gamma_l(t)),
$$
these functions $c_l(t)$($1\leq l\leq k$) satisfy
system (\ref{eq3.20}). Hence $c_l(t)$ depends on $\psi$.

 By an estimate of the biharmonic heat kernel in Proposition $3.1$ of \cite{ly}, we have that
\begin{equation}\label{eq3.26}\begin{aligned}
 &\|\mathcal{A}^T(\psi)_{xx}\|_{C_{\Phi}\big((T,T+1)\times{\mathbb R}\big)}
 \,+\,
 \|\mathcal{A}^T(\psi)_x\|_{C_{\Phi}\big((T,T+1)\times{\mathbb R}\big)}
 \,+\,
 \|\mathcal{A}^T(\psi)\|_{C_{\Phi}\big((T,T+1)\times{\mathbb R}\big)}
 \\
 &\leq C_0\|f-C(\psi)\|_{C_{\Phi}((s,s+1)\times{\mathbb R})},
\end{aligned}\end{equation}
for some uniform constant $C_0>0$.
Set
$$
c:=C_0\|f\|_{C_{\Phi}\big((T,T+1)\times{\mathbb R}\big)}
$$
and
$$
X^T_c:=\big\{\psi\,:\, \|\psi_{xx}\|_{C_{\Phi}\big((T,T+1)\times{\mathbb R}\big)}
\,+\,
\|\psi_x\|_{C_{\Phi}\big((T,T+1)\times{\mathbb R}\big)}
\,+\,
\|\psi\|_{C_{\Phi}\big((T,T+1)\times{\mathbb R}\big)}<2c \big\},
$$
where $C_0$ is given by (\ref{eq3.26}).

By (\ref{eq3.26}) and Lemma \ref{lem6}, we have
\begin{align*}
&\|\mathcal{A}^T(\psi)_{xx}\|_{C_{\Phi}\big((T,T+1)\times{\mathbb R}\big)}
\,+\,
\|\mathcal{A}^T(\psi)_x\|_{C_{\Phi}\big((T,T+1)\times{\mathbb R}\big)}
\,+\,
\|\mathcal{A}^T(\psi)\|_{C_{\Phi}\big((T,T+1)\times{\mathbb R}\big)}
\\[2mm]
&\leq  C_0\Big{(}\|C(\psi)\|_{C_{\Phi}\big((T,T+1)\times{\mathbb R}\big))}
\,+\,
\|f\|_{C_{\Phi}\big((T,T+1)\times{\mathbb R}\big)}\Big{)}
\\[2mm]
&\leq CC_0\Big{(}\frac{1}{\sqrt{\abs{T+1}}}\|\psi\|_{C_{\Phi}\big((T,T+1)\times{\mathbb R}\big)}
\,+\,
\frac{1}{\log{\abs{T+1}}}\|f\|_{C_{\Phi}\big((T,T+1)\times{\mathbb R}\big)}\Big{)}+c,
\end{align*}
where $C$ is given by Lemma \ref{lem6}.
Hence, taking $T$ large enough, we obtain that
$$
\mathcal{A}^T(X^T_c)\subseteq X^T_c.
$$
For any $\psi_1, \psi_2\in X^T_c$, by (\ref{eq3.26}) and Lemma \ref{lem6} again, we have
\begin{align*}
\|\mathcal{A}^T(\psi_1)-\mathcal{A}^T(\psi_2)\|_{C_{\Phi}\big((T,T+1)\times{\mathbb R}\big)}
\leq &\, C_0\|C(\psi_1)-C(\psi_2)\|_{C_{\Phi}\big((T,T+1)\times{\mathbb R}\big)}
\\[2mm]
\leq &\, C_0\|C(\psi_1-\psi_2)\|_{C_{\Phi}\big((T,T+1)\times{\mathbb R}\big)}
\\[2mm]
\leq&\,\frac{CC_0}{\sqrt{\abs{T+1}}}\|\psi_1-\psi_2\|_{C_{\Phi}\big((T,T+1)\times{\mathbb R}\big)}.
\end{align*}
Similarly, we have
\begin{equation*}
  \|\mathcal{A}^T(\psi_1)_x-\mathcal{A}^T(\psi_2)_x\|_{C_{\Phi}\big((T,T+1)\times{\mathbb R}\big)}
\leq\frac{CC_0}{\sqrt{\abs{T+1}}}\|\psi_1-\psi_2\|_{C_{\Phi}\big((T,T+1)\times{\mathbb R}\big)}
\end{equation*}
and
\begin{equation*}
  \|\mathcal{A}^T(\psi_1)_{xx}-\mathcal{A}^T(\psi_2)_{xx}\|_{C_{\Phi}\big((T,T+1)\times{\mathbb R}\big)}
\leq\frac{CC_0}{\sqrt{\abs{T+1}}}\|\psi_1-\psi_2\|_{C_{\Phi}\big((T,T+1)\times{\mathbb R}\big)},
\end{equation*}
where $C$ is given by Lemma \ref{lem6}.
Hence, taking $T$ large enough, we have the operator $\mathcal{A}^T$ is a contraction map from $X^T_c$ to itself.

According to the Banach fixed point theorem, we know that there exists an unique $\psi^T\in X^T_c$ such that $\mathcal{A}^T(\psi^T)=\psi^T$,
that is a solution to equation (\ref{eq3.19}) with $s=T+1$.

We finally extend the solution $\psi^T(t,r)$ in $(T,T+1)\times{\mathbb R}$ to $(T, s)\times{\mathbb R}$, $T>T_0$,
which still satisfies the orthogonality condition (\ref{eq3.20}) and the priori estimate in Lemma \ref{lem5}. We choose $T_0$ large enough such  that $\frac{CC_0}{\sqrt{\abs{T_0}}}<1$, where $C$ is given by Lemma \ref{lem6} and $C_0$ is given by (\ref{eq3.26}).
Thus the above fixed-point argument can be repeated when $T+1\leq s$. Hence passing finite steps
of fixed-point arguments, the solution $\psi^s(t,x)$ can be extended to $[T,s)$.
Moreover the solution $\psi^T$ satisfies (\ref{eq3.27}) and the orthogonality condition.
\end{proof}

\subsubsection{\textbf{Proof  of Proposition  \ref{prop2}}}\label{section3.2.4}

We choose a sequence $T_j\rightarrow -\infty$. Let $\psi^{T_j}$ be the solution to
the problem (\ref{eq3.19}) with $T=T_j$, according to Lemma \ref{lem7}. By (\ref{eq3.27}), we can find the sequence $\{\psi^{T_j}\}$
convergence to $\psi$ (up to subsequence) locally uniformly in $(T,+\infty)\times{\mathbb R}$. Using (\ref{eq3.7}) and standard parabolic
theory we have that $\psi$ is a solution of (\ref{eq3.1}) and satisfies (\ref{eq3.4}). The proof is completed.
\qed

\section{Solving the nonlinear problem} \label{sec:nl}
In this section, we mainly solve the nonlinear problem (\ref{eq7.7})-(\ref{eq7.8}) by using a fixed-point argument.
%
According to Proposition \ref{prop2}, $\phi$ is a solution to (\ref{eq7.7})-(\ref{eq7.8}) if only if $\phi\in C_{\Phi}((T, +\infty)\times{\mathbb R}) $
 is a fixed point of the operator
\begin{equation}\label{eq4.2}
 \mathbf{ T}(\phi):=\mathcal{A}(E(t,x)+N(\phi)),
\end{equation}
where $T>0$ is large enough  and the map $\mathcal{A}$ is given by Proposition \ref{prop2}.

Let $T_1>1$. We define two spaces
\begin{equation}\label{ls}
  \Lambda_{T_1}:=\Big{\{} \vec{h} \,:\, \text{ each}\ h_i\in C^1[T_1,+\infty)\ \text{and}\ \sup_{t\geq T_1}|h_i(t)|+\sup_{t\geq T_1}\Big{(}|t||h_i'(t)|\Big{)}<1\Big{\}}
\end{equation}
with $\vec{h}=(h_1,\cdots,h_k)$ and the norm
\begin{equation}\label{lamuda}
  \|\vec{h}\|_{\Lambda_{T_1}}:=\sum_{i=1}^k\left\{\sup_{t\geq T_1}|h_i(t)|+\sup_{t\geq T_1}\Big{(}|t||h'_i(t)|\Big{)}\right\},
\end{equation}
and also
\begin{equation}\label{eq4.9}
X_{T_1}:=\left\{\phi: \|\phi\|_{X_{T_1}}\leq 2 \widehat{C}\left[T_1\right]^{-\frac{\sigma}{16\sqrt{2}}}\right\},
\end{equation}
with the norm
$$\|\phi\|_{X_{T_1}}:=\|\phi\|_{C_{\Phi}((T_1,+\infty)\times{\mathbb R})}
\,+\,
\|\phi_{x}\|_{C_{\Phi}((T_1,+\infty)\times{\mathbb R})}
\,+\,
\|\phi_{xx}\|_{C_{\Phi}((T_1,+\infty)\times{\mathbb R})},$$
where the norm $\|\cdot\|_{C_{\Phi}}$ is defined by \eqref{eq3.3} with $\Phi$ in \eqref{dphi}  and $\widehat{C}$ is a fixed positive constant.
The main result is given by the following proposition.

\begin{prop}\label{prop3}
Let $\sigma\in(0,\sqrt{2})$ and $\alpha>1$ in \eqref{dphi}.
There exists $T_1\geq 1$ depending only on $ k$, $\sigma$ and $\alpha$
such that for any given function $\vec{h}=(h_1(t),\cdots,h_k(t)) \in \Lambda_{T_1}$, there is a solution $\phi(t,x,\vec{h})$ to $\phi=\mathbf{ T}(\phi)$ with respect to $\vec{\gamma}(t)=\vec{\gamma}^0(t)+\vec{h}(t)$.
The solution $\phi(t,x,\vec{h})$ satisfies problem (\ref{eq7.7})-(\ref{eq7.8}).
Furthermore, the following estimate holds
\begin{equation}\label{eq4.3}
  \big\|\phi(t,x,\vec{h}^1)-\phi(t,x,\vec{h}^2)\big\|_{X_{T_1}}\leq C \left[T_1\right]^{-\frac{\sigma}{16\sqrt{2}}}\big\|\vec{h}^1-\vec{h}^2\big\|_{\Lambda_{T_1}},
\end{equation}
where $C$ is a positive constant which does not depend on $\vec{h}^1, \vec{h}^2$ and $T_1$.
\qed
\end{prop}

To prove the above proposition, we first consider the following lemmas. Notice that the error term $E(t,x)$ in (\ref{eq7.5}) and the nonlinear term $N(\phi)$  in (\ref{eq7.6}) are all depend on $\vec{h}$, due to $\vec{\gamma}=\vec{\gamma}^0+\vec{h}$. So we denote $E(t,x)$ and $N(\phi)$ by $E(t,x,\vec{h})$ and $N(\phi,\vec{h})$ respectively.
\begin{lem}\label{lem8}
Let $\vec{h}^1,\vec{h}^2\in\Lambda_{T_1}$ and $\phi_1,\phi_2\in X_{T_1}$, then there exists a constant $C$ depending on $\widehat{C}$ such that
\begin{equation}\label{eq4.4}
\begin{aligned}
 &\|N\big(\phi_1,\vec{h}^1\big)-N\big(\phi_2, \vec{h}^2\big)\|_{C_{\Phi}((T_1,+\infty)\times{\mathbb R})}\leq C\big[ T_1]^{-\frac{\sigma}{16\sqrt{2}}}\Big{\{}\|\vec{h}^1-\vec{h}^2\|_{\Lambda_{T_1}}+\|\phi_1-\phi_2\|_{X_{T_1}}\Big{\}}
\end{aligned}
\end{equation}
and
\begin{equation}\label{eq4.5}
\begin{aligned}
 \|E\big(t,x,\vec{h}^1\big)-E\big(t,x,\vec{h}^2\big)\|_{C_{\Phi}((T_1,+\infty)\times{\mathbb R})}\leq &C\big[T_1\big]^{-\frac{\sigma}{16\sqrt{2}}}\|\vec{h}^1-\vec{h}^2\|_{\Lambda_{T_1}}.
\end{aligned}
\end{equation}
\end{lem}

\begin{proof}
Recall that
$$
N(\phi,h)=F(z(t,x)+\phi)-F(z(t,x))-F'(z(t,x))[\phi]
$$
in (\ref{eq7.6}) with $F(u)$ and $F'(u)[v]$ are given by \eqref{deF} and \eqref{deF'} respectively, then we have
\begin{align*}
&\abs{N(\phi_1,\vec{h})-N(\phi_2,\vec{h})}
\\[2mm]
&=\abs{\int_0^1F'\big(z(t,x)+y\phi_1+(1-y)\phi_2\big)
\big[\phi_1-\phi_2\big]\,{\mathrm d}y
\,-\,
F'(z(t,x))[\phi_1-\phi_2]}
\\[2mm]
&=\abs{\int_0^1\Big{\{}\widetilde{F}'\big(z(t,x)+y\phi_1+(1-y)\phi_2\big)\big[\phi_1-\phi_2\big]
-\widetilde{F}'(z(t,x))[\phi_1-\phi_2]\Big{\}}\,{\mathrm d}y}
\\[2mm]
&=\abs{\int_0^1\int_0^y\Big{\{}\widetilde{F}''\big(z(t,x)+s\phi_1+(1-s)\phi_2\big)
\big[\phi_1-\phi_2,\,\phi_1-\phi_2\big]\Big{\}}
\,{\mathrm d}s{\mathrm d}y}
\\[2mm]
& \leq C\big[\Phi(t,x)\big]^2\Big{\{}\left\|\phi_1-\phi_2\right\|_{C_\Phi}
\,+\,
\left\|(\phi_1-\phi_2)_{x}\right\|_{C_\Phi}
\,+\,
\left\|(\phi_1-\phi_2)_{xx}\right\|_{C_\Phi}\Big{\}}^2,
\end{align*}
due to the fact that $\Phi(t,x)$ is bounded for $(t,x)\in (T_1,+\infty)\times{\mathbb R}$ in \eqref{dphi}.
In the above,
$$
\widetilde{F}'(u)[v]=F'(u)[v]+v^{(4)}
$$
and $\widetilde{F}''(u)[v_1,v_2]$ denotes the Frechet derivative of $\widetilde{F}'$.
Moreover, we have that

\begin{align*}
\abs{N(\phi,h^1)-N(\phi,h^2)}&\leq\abs{\int_0^{1}\int_0^1F''\big(sz^1(t,x)+(1-s)z^2(t,x)+y\phi\big)[z^1-z^2,\phi]
\,{\mathrm d}s{\mathrm d}y}
\\[2mm]
&\qquad +\abs{\int_0^1F''\big(s z^1(t,x)+(1-s)z^2(t,x)\big)[z^1-z^2,\phi]\,{\mathrm d}s}
\\[2mm]
&\leq C\left\{\sum_{l=1}^k\Big[\abs{h^1_l-h^2_l}+\abs{t}\abs{(h^1_l)'-(h^2_l)'}\Big]\right\}\abs{\phi}
\\[2mm]
&\leq C\left\{\sum_{l=1}^k\Big[\abs{h^1_l-h^2_l}+\abs{t}\abs{(h^1_l)'-(h^2_l)'}\Big]\right\}\Phi(t,x)\|\phi\|_{C_\Phi},
\end{align*}
where $z^i(t,x)$ is defined in \eqref{dzz} with $\vec{h}=\vec{h}^i$ for
$i=1,2$,  and $C$ only depends on $\widehat{C}$.
Combining the above estimates, we can obtain that (\ref{eq4.4}) holds.

On the other hand, recall that
\begin{align*}
  E(t,x)=&-\partial_tz(t,x)-\partial_{xx}\Big{(}\partial_{xx}z(t,x)-W'(z(t,x))\Big{)}
\\
&+W''(z(t,x))\Big{(}\partial_{xx}z(t,x)-W'(z(t,x))\Big{)},
\end{align*}
where the approximate solution $z(t,x)$ is defined in \eqref{dzz}.
By the definition of $\vec{\gamma}^0$ in \eqref{degama0} and Lemmas \ref{esxi}-\ref{lem33},
we have that there exists a positive constant $C=C(k,\sigma)$ such that
\begin{align*}
&\abs{\partial_tz^1(t,x)-\partial_tz^2(t,x)}
\\[2mm]\leq
&\abs{\sum_{l=1}^k(-1)^{l+1}\partial_t\left[\omega\big(x-\gamma^1_l(t)\big)
-\omega\big(x-\gamma^2_l(t)\big)\right]}
\,+\,
\abs{\sum_{l=1}^{k-1}(-1)^{l+1}\partial_t\left[\xi^1_l(t,x)-\xi^2_l(t,x)\right]}
\\[2mm]
&\quad\,+\,\abs{\sum_{l=1}^k(-1)^{l+1}\partial_t\left[\widehat{\xi}^1_l(t,x)-\widehat{\xi}^2_l(t,x)\right]}
\\[2mm]
\leq&\sum_{l=1}^k\Big\{\abs{\big[\gamma^1_l(t)\big]_t}\abs{\omega'\big(x-\gamma^1_l(t)\big)-\omega'\big(x-\gamma^2_l(t)\big)}
\ +\ \abs{\big[\gamma^1_l(t)\big]_t-\big[\gamma^2_l(t)\big]_t}\omega'\big(x-\gamma^2_l(t)\big)\Big\}
\\[2mm]
&\quad+\sum_{l=1}^{k-1}\abs{\partial_t\left[\xi^1_l(t,x)-\xi^2_l(t,x)\right]}
\ +\ \sum_{l=1}^k\abs{\partial_t\left[\widehat{\xi}^1_l(t,x)-\widehat{\xi}^2_l(t,x)\right]}
\\[2mm]
&\leq C \Phi(t,x)\big[ T_1]^{-\frac{\sigma}{16\sqrt{2}}}\|h^1-h^2\|_{\Lambda_{T_1}}.
\end{align*}
By the fact that $\omega$ satisfies \eqref{eqq} and  similar arguments as in the proof of Lemma \ref{lem10},  we have that
\begin{align*}
&\abs{F(z^1(t,x))-F(z^2(t,x))}\\[2mm]
&=\Bigg{|}\Big[\partial_{xx}-W''(z^1(t,x))\Big]\Big{(}\partial_{xx}z^1(t,x)-W'(z^1(t,x))\Big{)}
\\[2mm]&\qquad-\Big[\partial_{xx}-W''(z^2(t,x))\Big]\Big{(}\partial_{xx}z^2(t,x)-W'(z^2(t,x))\Big{)}\Bigg{|}
\\[2mm]&\leq C\big[ T_1]^{-\frac{\sigma}{16\sqrt{2}}}\Phi(t,x)\|h^1-h^2\|_{\Lambda_{T_1}}\qquad  \text {if }\ \frac{\gamma^0_{j-1}+\gamma^0_{j}}{2}\leq x\leq \frac{\gamma^0_{j+1}+\gamma^0_{j}}{2}, \ j=1,\cdots,k,
\end{align*}
with $\gamma_0=-\infty$ and $\gamma_{k+1}=+\infty$.
Combining the above estimates, we can get that (\ref{eq4.5}) holds.
\end{proof}

\begin{lem}
Let the functions $\vec{h}^1, \vec{h}^2\in \Lambda_{T_1}$, $\ \phi_1, \phi_2\in X_{T_1}$, $\ c^l\big(\phi^l, \vec{h}^l, t\big)=\big(c^l_1(t), \cdots, c^l_k(t)\big)$ satisfy system

\begin{align*}
  &\sum_{j=1}^kc^l_j(t)\int_{{\mathbb R}}\omega'\big(x-\gamma^l_j(t)\big)\omega'\big(x-\gamma^l_i(t)\big)\,{\mathrm d}x
  \\[2mm]
  &=\int_{{\mathbb R}} F'(z^l(t,x))[\phi_l]\omega'\big(x-\gamma^l_i(t)\big)\,{\mathrm d}x
  \,-\,
  \partial_t\gamma^l_i(t)\int_{0}^\infty \phi^l(t,x)\omega''\big(x-\gamma^l_i(t)\big)\,{\mathrm d}x
 \\[2mm]
 &  \ \ \ +\int_{{\mathbb R}}\big(E^l(t,x)+N(\phi_l)\big)\omega'\big(x-\gamma^l_i(t)\big)\,{\mathrm d}x, \ \ \ \ \forall\, i=1,\cdots,k\ \text{and}\ t>T,
  \end{align*}

\noindent
where $\gamma^l_i=\gamma_i^0+h^l_i$, $z^l(t,x)$ and $E^l(t,x)$ are given by \eqref{dzz} and \eqref{eq7.5} respectively with $\vec{\gamma}=\vec{\gamma}^l$, $l=1,2$ and $i=1,\cdots,k$. Then we have
\begin{align}\label{eq4.8}
  \nonumber &\abs{c\big(\phi_1,\vec{h}^1,t\big)-c\big(\phi_2,\vec{h}^2,t\big)}
  \\[2mm]
  &\leq C\Bigg{[}t^{-{\frac{5}{4}-\frac{\sigma}{8\sqrt{2}}}}\|\phi_1-\phi_2\|_{C_{\Phi}((T_1,+\infty)\times{\mathbb R})}
  \ +\ t^{-\frac{3}{4}-\frac{3\sigma}{16\sqrt{2}}}\|\vec{h}^1-\vec{h}^2\|_{\Lambda_{T_1}}\Bigg{]},
\end{align}
where $C$ only depends on $\widehat{C}$ given in $(\ref{eq4.9})$.
\end{lem}
\begin{proof}
Proving this lemma just needs to do some similar calculations in Lemmas \ref{lem6} and \ref{lem8}, we omit it here.
\end{proof}

\textbf{Proof of Proposition \ref{prop3} } We consider the operator $\mathbf{T}$ defined by (\ref{eq4.2}) from the space
$X_{T_1}$ in (\ref{eq4.9}) to itself. We will prove that $\mathbf{T}$ is a contraction mapping.
Thus by fixed-point theorem,
 the operator $\mathbf{T}$ has a unique fixed point $\phi$,  i.e. $\mathbf{T}(\phi)=\phi$.

For any $\phi_1,\phi_2\in X_{T_1}$, according to Lemmas \ref{lem10} and \ref{lem8}, Proposition \ref{prop2}, we can find that
\begin{equation*}
\|  \mathbf{T}(0)\|_{X_{T_1}}\leq \widehat{C}\left[T_1\right]^{-\frac{\sigma}{16\sqrt{2}}},
\end{equation*}
and
\begin{equation*}
\|  \mathbf{T}(\phi_1)-\mathbf{T}(\phi_2)\|_{X_{T_1}}\leq C\left[T_1\right]^{-\frac{\sigma}{16\sqrt{2}}}
\| \phi_1-\phi_2\|_{X_{T_1}},
\end{equation*}
providing $\phi_1, \phi_2\in X_{T_1}$ defined in \eqref{eq4.9},
where $\sigma\in(0,\sqrt{2})$.
Hence, $\mathbf{T}$ is a contraction mapping in $X_{T_1}$ when taking $T_1$ large enough. Hence, according to the Banach fixed point theorem, there exists $\phi\in X_{T_1}$ such that
$\mathbf{T}(\phi)=\phi$.

Next we will prove the estimate (\ref{eq4.3}). Choosing $\vec{h}^1,\vec{h}^2\in \Lambda_{T_1}$ defined in \eqref{ls}, according to the above proof, we know that for each $i\in\{1,2\}$, there exist
$\phi_i=\phi(t,x,\vec{h}^i)$ is a solution to problem \eqref{eq7.7}-\eqref{eq7.8} with $\vec{\gamma}^i=\vec{\gamma}^0+\vec{h}^i$ .

 We note that $\phi_1-\phi_2$ does not
 satisfy the orthogonality condition (\ref{eq7.8}).
Let us consider $\bar{\phi}=\phi_1-\bar{\phi}_2$ with
\begin{equation*}
  \bar{\phi}_2:= \phi_2-\sum_{j=1}^k\tilde{b}_j(t)\omega'\big(x-\gamma^1_j(t)\big),
\end{equation*}
where $\tilde{b}_j(t)$, $j=1,\cdots,k$, are defined by the following equalities
\begin{equation}\label{deb}
\sum_{l=1}^k\tilde{b}_l(t)\int_{{\mathbb R}}\omega'\big(x-\gamma^1_l(t)\big)\omega'\big(r-\gamma^1_j(t)\big)\,{\mathrm d}x
=\int_{{\mathbb R}}\phi_2(t,x)\omega'\big(x-\gamma^1_j(t)\big)\,{\mathrm d}x.
\end{equation}
According the proof of Lemma \ref{lem6}, the left hand side of above equality is nearly diagonal, hence these functions $\tilde{b}_j(t)$ with $j=1,\cdots,k$,
are well-defined.

Thus $\bar{\phi}_2$ satisfies the following problem

\begin{align*}
\bar{\phi}_t=F'(z^1(t,x))[\bar{\phi}] +\big[E(t,x,h^1)-E(t,x,h^2)\big] +\big[N(\phi_1,h^1)-N(\phi_2,h^2)\big]
+ R\big(\vec{h}^1,\vec{h}^2,\phi_2\big)
\\[3mm]\qquad\qquad
+\sum\limits_{j=1}^kc^2_j(t)\Big[\omega'\big(x-\gamma^2_j(t)\big)-\omega'\big(x-\gamma^1_j(t)\big)\Big]
-\sum\limits_{j=1}^k\big[c^1_j(t)-c^2_j(t)\big]\omega'\big(x-\gamma^1_{j}(t)\big)
\\[3mm]
\qquad\qquad\text{in}\ (T,+\infty)\times{\mathbb R},
\end{align*}

\noindent and
\begin{align*}
 \int_{{\mathbb R}} \bar{\phi}(t,x)\omega'\big(x-\gamma^1_j(t)\big)\,{\mathrm d}x=0 \qquad  \text{for all}\ j=1,\cdots,k\ \text{and} \ t>T,
 \end{align*}
where
\begin{align*}
R(\cdot)=&\Big{[}F'(z^1(t,x))-F'(z^2(t,x))\Big{]}[\phi_2]
\\
&\,+\,\sum_{j=1}^k\Big{[}\tilde{b}'_j(t)\omega'\big(x-\gamma^1_j(t)\big)
-[\gamma^1_j(t)]'\tilde{b}_j(t)\omega''(x-\gamma^1_j(t)\Big{]}.
\end{align*}
Hence  using Lemmas {\ref{lem7}} and {\ref{lem8}}, similar proof of Lemma \ref{lem10} and the fact
\begin{equation*}
   \frac{\abs{\omega''\big(x-\gamma_i(t)\big)}}{\Phi(t,x)}+\frac{\omega'\big(x-\gamma_i(t)\big)}{\Phi(t,x)}\leq C\frac{t^{\frac{3}{4}+\frac{\sigma}{8\sqrt{2}}}}{\big[\log t\big]^\alpha},\quad \forall\, i=1,\cdots,k,
\end{equation*}
 we can prove that
\begin{equation}\label{eq4.10}\begin{aligned}
 \|  \bar{\phi}\|_{X_{T_1}}
 \,\leq\, &C\left[T_1\right]^{-\frac{\sigma}{16\sqrt{2}}}\Big{\{}\|h^1-h^2\|_{\Lambda_{T_1}}+\|\phi_1-\phi_2\|_{X_{T_1}}\Big{\}}
\\[2mm]
&\ \ +C\sum_{l=1}^k\sup_{t\geq T_1}\frac{\,t^{\frac{3}{4}+\frac{\sigma}{8\sqrt{2}}}\,}
{\big[\log t\big]^\alpha}\Big{(}\frac{1}{t}\abs{\tilde{b}_l(t)}+\abs{\tilde{b}'_l(t)}\Big{)}.
\end{aligned}\end{equation}
By the orthogonality condition (\ref{eq7.8}) and the definition of $\Phi(t,x)$ in \eqref{dphi}, we get that
\begin{equation}\label{eq4.11}\begin{aligned}
  &\abs{\int_{{\mathbb R}}\phi_2(t,x)\omega'\big(x-\gamma^1_j(t)\big)\,{\mathrm d}x}
\\[3mm]
&=\abs{\int_{{\mathbb R}} \phi_2(t,x)\Big[\omega'\big(x-\gamma^2_j(t)\big)-\omega'\big(x-\gamma^1_j(t)\big)\Big]
\,{\mathrm d}x}
\\[3mm]
&\leq C\left[T_1\right]^{-\frac{\sigma}{16\sqrt{2}}}\|h^1-h^2\|_{\Lambda_{T_1}}t^{-\frac{3}{4}-\frac{\sigma}{8\sqrt{2}}}.
\end{aligned}\end{equation}
We consider
\begin{equation}\label{eq4.12}
\begin{aligned}
  &\abs{\frac{{\mathrm d}}{{\mathrm d}t}\int_{{\mathbb R}} \phi_2(t,x)\omega'\big(x-\gamma^1_j(t)\big)\,{\mathrm d}x}
  \\[2mm]
  &=\abs{\frac{{\mathrm d}}{{\mathrm d}t}\int_{{\mathbb R}} \phi_2(t,x)\big{[}\omega'\big(x-\gamma^2_j(t)\big)-\omega'\big(x-\gamma^1_j(t)\big)\big{]}\,{\mathrm d}x}.
\end{aligned}
\end{equation}
By equation (\ref{eq7.7}), we obtain that
\begin{align}\label{eq4.13}
\nonumber&\int_{{\mathbb R}}(\phi_2)_t\left[\omega'\big(x-\gamma^2_j(t)\big)-\omega'\big(x-\gamma^1_j(t)\big)\right]\,{\mathrm d}x
  \\[2mm]
  &=\int_{{\mathbb R}} \Big{[}(\phi_2)_t-F'(z^2(t,x))[\phi_2]\Big{]}
  \Big\{\omega'\big(x-\gamma^2_j(t)\big)-\omega'\big(x-\gamma^1_j(t)\big)\Big\}\,{\mathrm d}x\nonumber
\\[2mm]
&\qquad +\int_{{\mathbb R}}F'(z^2(t,x))[\phi_2]\Big\{\omega'\big(x-\gamma^2_j(t)\big)-\omega'\big(x-\gamma^1_j(t)\big)\Big\}\,{\mathrm d}x\nonumber
\\[2mm]
&=\int_{{\mathbb R}} \Big{[}(\phi_2)_t-F'(z^2(t,x))[\phi_2]\Big{]}\Big\{\omega'\big(x-\gamma^2_j(t)\big)-\omega'\big(x-\gamma^1_j(t)\big)\Big\}\,{\mathrm d}x\nonumber
\\[2mm]
& \qquad  +\int_{{\mathbb R}}[\phi_2]F'(z^2(t,x))\Big\{\omega'\big(x-\gamma^2_j(t)\big)-\omega'\big(x-\gamma^1_j(t)\big)\Big\}\,{\mathrm d}x,
\end{align}
where $\widetilde{\omega}=\omega'\big(x-\gamma^2_j(t)\big)-\omega'\big(x-\gamma^1_j(t)\big)$.

Since $\phi_2\in X_{T_1}$, using Lemma \ref{lem6}, we have
\begin{equation}\label{eq4.14}\begin{aligned}
  &\abs{\int_{{\mathbb R}}(\phi_2)_t\Big{[}\omega'\big(x-\gamma^2_j(t)\big)-\omega'\big(x-\gamma^1_j(t)\big)\Big{]}\,{\mathrm d}x} \leq C\left[T_1\right]^{-\frac{\sigma}{16\sqrt{2}}}\|h^1-h^2\|_{\Lambda_{T_1}}t^{-\frac{3}{4}-\frac{\sigma}{8\sqrt{2}}}.
\end{aligned}\end{equation}
By the above estimates (\ref{eq4.11}), (\ref{eq4.12}), (\ref{eq4.14}) and the definition of $\tilde{b}_j(t)$ in \eqref{deb}, we have
\begin{equation*}
  \abs{\tilde{b}_l(t)}+\abs{\tilde{b}'_l(t)}\leq C\left[T_1\right]^{-\frac{\sigma}{16\sqrt{2}}}\|h^1-h^2\|_{\Lambda_{T_1}}t^{-\frac{3}{4}-\frac{\sigma}{8\sqrt{2}}}.
\end{equation*}
Hence by \eqref{eq4.10}, there holds that
\begin{equation*}
  \|  \bar{\phi}\|_{X_{T_1}}\leq C \left[T_1\right]^{-\frac{\sigma}{16\sqrt{2}}}\Big{[}
\| \phi_1-\phi_2\|_{X_{T_1}}+\|h^1-h^2\|_{\Lambda_{T_1}}\Big{]}.
\end{equation*}
Eventually, we have
\begin{equation*}\begin{aligned}
  \| \phi_1-\phi_2\|_{X_{T_1}}&\leq  \| \bar{\phi}\|_{X_{T_1}}
+C\sum_{l=1}^k\sup_{t\geq T_1}\left\{\frac{t^{\frac{3}{4}+\frac{\sigma}{8\sqrt{2}}}}{\big[\log t\big]^\alpha}\abs{\tilde{b}_l(t)}\right\}
\\[2mm]
&
\leq C \left[T_1\right]^{-\frac{\sigma}{16\sqrt{2}}}\Big{[}
\| \phi_1-\phi_2\|_{X_{T_1}}+\|h^1-h^2\|_{\Lambda_{T_1}}\Big{]}.
\end{aligned}\end{equation*}
 Choosing $T_1$ large enough, we can get the inequality (\ref{eq4.3}).
 \qed

\section{The reduction procedure} \label{sec:tc}

\subsection{Deriving the nonlinear system with reduced equations}\label{section5.1}

We will choose suitable $\vec{h}$ such that these functions $c_i(t)=0$ with $i=1,\cdots,k$, in equation (\ref{eq7.7}). Note that
the functions $c_i(t)$ are depend on $\vec{h}$. According to (\ref{eq7.9}),
these relations $c_i(t)=0$ with $i=1,\cdots, k$ are equivalent to
\begin{align}
  0=&\int_{{\mathbb R}} F'(z(t,x))[\phi]\omega'\big(x-\gamma_i(t)\big)\,{\mathrm d}x
  \,-\,
  \gamma_i'(t)\int_{{\mathbb R}} \phi(t,x)\omega''\big(x-\gamma_i(t)\big)\,{\mathrm d}x
  \nonumber\\[2mm]&
  +\int_{{\mathbb R}}\big[E(t,x)+N(\phi)\big]\omega'\big(x-\gamma_i(t)\big)\,{\mathrm d}x\qquad \text{for all}\ i=1,\cdots,k\ \text{and}\ t>T,
  \label{eq7.12}
\end{align}
where the error term  $E(t,x)$ and the nonlinear term $N(\phi)$ are defined by (\ref{eq7.5}) and (\ref{eq7.6}) respectively, and the functions $\phi(t,x),\partial_x\phi(t,x), \partial_{xx}\phi(t,x) \in C_{\Phi}([T,+\infty)\times{\mathbb R})$ i.e.
\begin{equation}\label{spsi}
  \sup_{x\in{\mathbb R},\ t>T}\Big\{\big[\abs{\phi(t,x)}
  \,+\,\abs{\partial_{x}\phi(t,x)}
  \,+\,\abs{\partial_{xx}\phi(t,x)}\big]\Phi(t,x)\Big\}<+\infty,
\end{equation}
where the function $\Phi(t,x)$ is defined by \eqref{dphi}.

\subsubsection{}\label{section5.1.1}
We firstly estimate the projection of the error term $E(t,x)$ in \eqref{eq7.12}.

\noindent\textbf{Step 1.}
 By the definitions of $E(t,x)$ and $z(t,x)$ in \eqref{eq7.5} and \eqref{dzz}, Lemmas \ref{esxi}-\ref{lem33} and \ref{lem10}, we have that
\begin{align}\label{dEE}
\nonumber&\int_{\mathbb R} E(t,x)\omega'\big(x-\gamma_i(t)\big)\,{\mathrm d}x
\\[3mm]
&\nonumber=\int_{J_i(t)}E(t,x)\omega'\big(x-\gamma_i(t)\big)\,{\mathrm d}x
+\int_{\left[ J_i(t)\right]^c}E(t,x)\omega'\big(x-\gamma_i(t)\big)\,{\mathrm d}x
\\[3mm]
&\nonumber=\int_{J_i(t)}E(t,x)\omega'\big(x-\gamma_i(t)\big)\,{\mathrm d}x+O\left(t^{-1-\frac{\sigma}{4\sqrt{2}}}\right)
\\[3mm]
&\nonumber=-\int_{J_i(t)}\Big[\partial_{xx}-W''(z(t,x))\Big]\Big[\partial_{xx}z(t,x)-W'(z(t,x))\Big]\omega'\big(x-\gamma_i(t)\big)\,{\mathrm d}x
\\[3mm]
&\qquad+(-1)^{i+1}\frac{2\sqrt{2}}{3}\gamma'_i(t)\,+\, O\left(t^{-1-\frac{\sigma}{4\sqrt{2}}}\right),
\end{align}
where $\sigma\in(0,\sqrt{2})$, the intervals  $J_i(t)$ with $i=1,\cdots,k$, are defined by
\begin{align*}
  J_i(t):=\left[\frac{\gamma_{i}(t)+\gamma_{i-1}(t)}{2},\frac{\gamma_{i}(t)+\gamma_{i+1}(t)}{2}\right],
\end{align*}
and the interval $\big[J_i(t)\big]^c$ is the complementary set of $J_i(t)$ on the whole space ${\mathbb R}$. Here we have used the fact that
\begin{equation*}
  \int_{\mathbb R}[\omega'(x)]^2\,{\mathrm d}x=\frac{2\sqrt{2}}{3}.
\end{equation*}

\noindent\textbf{Step 2.}
In this step, we concern the first term at the right hand side of \eqref{dEE}.
According to the integration by parts and the equalities \eqref{eqq}, \eqref{esee}, \eqref{estmE}, we get that
\begin{align}\label{eseee1}
\nonumber&\int_{J_i(t)}\Big[\partial_{xx}-W''(z(t,x))\Big]\Big[\partial_{xx}z(t,x)-W'(z(t,x))\Big]\omega'\big(x-\gamma_i(t)\big)\,{\mathrm d}x
\\[2mm]\nonumber&=\int_{J_i(t)}\Big[\partial_{xx}-W''(z(t,x))\Big]
\Bigg[\sum_{j=i-1}^{i+1}(-1)^{j+1}d_j(t)\omega'\big(x-\gamma_j(t)\big)+(-1)^{i+1}\mathcal{E}_i(z(t,x))\Bigg]
\\[2mm]\nonumber
&\qquad\qquad\times\omega'\big(x-\gamma_i(t)\big)\,{\mathrm d}x
\\[2mm]\nonumber&=\int_{J_i(t)}\Big[\partial_{xx}-W''(z(t,x))\Big]
\Bigg[\sum_{j=i-1}^{i+1}(-1)^{j+1}d_j(t)\omega\big(x-\gamma_j(t)\big)\Bigg]\omega'\big(x-\gamma_i(t)\big)\,{\mathrm d}x
\\[2mm]\nonumber&\quad+(-1)^{i+1}\int_{J_i(t)}\Big[W''\big(\omega\big(x-\gamma_i(t)\big)\big)-W''(z(t,x))\Big]\mathcal{E}_i(z(t,x))\omega'\big(x-\gamma_i(t)\big)\,{\mathrm d}x
\\[2mm]
&\quad+O\left(t^{-1-\frac{\sigma}{4\sqrt{2}}}\right)
\nonumber\\[2mm]
&:=\mathfrak{e}_1(t)+\mathfrak{e}_2(t)+O\left(t^{-1-\frac{\sigma}{4\sqrt{2}}}\right),
\end{align}
where we have used the fact that $\mathcal{E}_i(z(t,x))=O\left(t^{-\frac{3}{4}-\frac{\sigma}{4\sqrt{2}}}\right)$ given in \eqref{estmE}, the functions $\mathfrak{e}_1(t)$ and $\mathfrak{e}_2(t)$ are defined by the last equality.
Note that the functions $d_j(t)$ are defined in \eqref{ddt} with the estimates in Lemma \ref{lem33}.

For the function $\mathfrak{e}_2(t)$,  we have that
\begin{align}\label{eseee2}
  \nonumber&(-1)^{i+1}\mathfrak{e}_2(t)=\int_{J_i(t)}\Big[W''\big(\omega\big(x-\gamma_i(t)\big)\big)-W''(z(t,x))\Big]\mathcal{E}_i(z(t,x))\omega'\big(x-\gamma_i(t)\big)\,{\mathrm d}x
\\[2mm]\nonumber
&=3\int_{J_i(t)}\Big[\omega\big(x-\gamma_i(t)\big)+(-1)^{i+1}z(t,x)\Big]
\\[2mm]\nonumber
&\qquad\qquad\qquad\times\Big[\omega\big(x-\gamma_i(t)\big)-(-1)^{i+1}z(t,x)\Big]\mathcal{E}_i(z(t,x))\omega'\big(x-\gamma_i(t)\big)\,{\mathrm d}x
\\[2mm]&=O\left(t^{-1-\frac{\sigma}{4\sqrt{2}}}\right),
\end{align}
where we have used the facts that $W''(s)$ is an even function and
\begin{align*}
\abs{(-1)^{i+1}z(t,x)-\omega\big(x-\gamma_i(t)\big)}&\leq C\left[\abs{\omega\big(x-\gamma_{i-1}(t)\big)-1}+\abs{\omega\big(x-\gamma_{i+1}(t)\big)+1}+t^{-\frac{1}{2}}\right]
\\[2mm]&\leq Ct^{-\frac{1}{4}},
\end{align*}
for all $x\in J_i(t)$, where the constant $C>0$ does not depend on $t$ and $x$.

\noindent\textbf{Step 3.}
We consider the function $\mathfrak{e}_1(t)$ in \eqref{eseee1}.
By equation \eqref{eqq}, there holds that
\begin{align}\label{esee1}
  \nonumber&\mathfrak{e}_1(t)=\sum_{j=i-1}^{i+1}(-1)^{j+1}d_j(t)\int_{J_i(t)}\Big[\partial_{xx}-W''(z(t,x))\Big]\Big[\omega'\big(x-\gamma_j(t)\big)\Big]\omega'\big(x-\gamma_i(t)\big)\,{\mathrm d}x
\\[2mm]\nonumber&=\sum_{j=i-1}^{i+1}(-1)^{j+1}d_j(t)\int_{J_i(t)}\Big[W''\big(\omega\big(x-\gamma_j(t)\big)\big)-W''(z(t,x))\Big]\omega'\big(x-\gamma_j(t)\big)\omega'\big(x-\gamma_i(t)\big)\,{\mathrm d}x
\\[2mm]\nonumber
&=3\sum_{j=i-1}^{i+1}(-1)^{j+1}d_j(t)\int_{J_i(t)}\Big[\omega\big(x-\gamma_j(t)\big)+(-1)^{j+1}z(t,x)\Big]\Big[\omega\big(x-\gamma_j(t)\big)-(-1)^{j+1}z(t,x)\Big]
\\[2mm]\nonumber
&\qquad\qquad\times\omega'\big(t,x-\gamma_j(t)\big)\omega'\big(x-\gamma_i(t)\big)\,{\mathrm d}x
\\[2mm]\nonumber
&=3(-1)^{i+1}d_i(t)\int_{J_i(t)}\Big[\omega\big(x-\gamma_i(t)\big)+(-1)^{i+1}z(t,x)\Big]\Big[\omega\big(x-\gamma_i(t)\big)-(-1)^{i+1}z(t,x)\Big]
\\[2mm]\nonumber
&\qquad\qquad\qquad\qquad\times\big[\omega'\big(x-\gamma_i(t)\big)\big]^2\,{\mathrm d}x
\\[2mm]
&\quad-(-1)^{i+1}\Bigg\{16\sqrt{2}d_{i-1}(t)e^{-\sqrt{2}\eta_{i}(t)}+16\sqrt{2}d_{i+1}(t)e^{-\sqrt{2}\eta_{i+1}(t)}
+O\left(t^{-1-\frac{\sigma}{4\sqrt{2}}}\right)\Bigg\},
\end{align}
where  the constant $\sigma\in(0,\sqrt{2})$, and the functions $\eta_i$ are given in \eqref{eta}.
In the above, we have used the following estimates
\begin{align*}
&\Big[\omega\big(x-\gamma_j(t)\big)+(-1)^{j+1}z(t,x)\Big]\Big[\omega\big(x-\gamma_j(t)\big)-(-1)^{j+1}z(t,x)\Big]
\\[2mm]&=\big[1-\omega\big(x-\gamma_i(t)\big)\big]\big[1+\omega\big(x-\gamma_i(t)\big)\big]+O\left(t^{-\frac{1}{4}}\right)
\\[2mm]&=\sqrt{2}\omega'\big(x-\gamma_i(t)\big)+O\left(t^{-\frac{1}{4}}\right),
\end{align*}
for all $x\in J_i(t)$ and $i\neq j$. By the expression of $\omega(x)$ in \eqref{domega},  there hold that
\begin{equation*}
  \int_{J_i(t)}\omega'\big(x-\gamma_j(t)\big)\omega'\big(x-\gamma_i(t)\big)\,{\mathrm d}x=O\left(t^{-\frac{1}{2}}\ln t\right),
\end{equation*}
and
\begin{align*}
  \int_{J_i(t)}\omega'\big(x-\gamma_j(t)\big)\big[\omega'\big(x-\gamma_i(t)\big)\big]^2\,{\mathrm d}x
=\frac{16}{3}e^{-\sqrt{2}\abs{\gamma_j(t)-\gamma_i(t)}}+O\left(t^{-\frac{3}{4}}\right).
\end{align*}

Next we need to study the first term at the right hand side of \eqref{esee1}.
By the definition of $z(t,x)$ in \eqref{dzz},
we get that
\begin{align*}
z\big(t,x+\gamma_i(t)\big)=&\,(-1)^{i+1}\omega(x)
\,+\,
\sum_{j< i}(-1)^{j+1}\big[\omega\big(x+\gamma_i(t)-\gamma_j(t)\big)-1\big]
\\[2mm]
&+\sum_{j> i}(-1)^{j+1}\big[\omega\big(x+\gamma_i(t)-\gamma_j(t)\big)+1\big]
-\sum_{j=1}^{k-1}(-1)^{j+1}\xi_j\big(t,x+\gamma_i(t)\big)
\\[2mm]
&-\sum_{j=1}^k(-1)^{j+1}\widehat{\xi}_j\big(t,x+\gamma_i(t)\big)
\\[2mm]
=&\,(-1)^{i+1}\Big[\omega(x)-g_1(t,x)-g_2(t,x)-\widehat{\xi}_i\big(t,x+\gamma_i(t)\big)+\xi_{i-1}\big(t,x+\gamma_i(t)\big)
\\[2mm]
&\qquad\qquad-\xi_i\big(t,x+\gamma_i(t)\big) +g(t,x)\Big],
\end{align*}
where
$$
g_1(t,x):=\omega\big(x+\gamma_i(t)-\gamma_{i-1}(t)\big)-1,
 \qquad
 g_2(t,x):=\omega\big(x+\gamma_i(t)-\gamma_{i+1}(t)\big)+1,
 $$
and $g(t,x)$ is defined by
\begin{align*}
  g(t,x):=&\sum_{j=1}^{i-2}(-1)^{j+i}\big[\omega\big(x+\gamma_i(t)-\gamma_j(t)\big)-1\big]
  \,+\,
  \sum_{j=i+2}^k(-1)^{j+i}\big[\omega\big(x+\gamma_i(t)-\gamma_j(t)\big)+1\big]
\\[2mm]
&\,-\,
\sum_{\substack{j=1\\ j\neq i-1, i}}^{k-1}(-1)^{j+i}\xi_j\big(t,x+\gamma_i(t)\big)
\,-\,
\sum_{\substack{j=1\\ j\neq i}}^{k}(-1)^{j+i}\widehat{\xi}_j\big(t,x+\gamma_i(t)\big).
\end{align*}
Hence by Lemmas \ref{esxi}-\ref{lem33}, we have that
\begin{align*}
&\Big[\omega(x)+(-1)^{i+1}z\big(t,x+\gamma_i(t)\big)\Big]\Big[\omega(x)-(-1)^{i+1}z\big(t,x+\gamma_i(t)\big)\Big]
\\[2mm]
&=\Big[g_1(t,x)+g_2(t,x)+\widehat{\xi}_i\big(t,x+\gamma_i(t)\big)-\xi_{i-1}\big(t,x+\gamma_i(t)\big)+\xi_i\big(t,x+\gamma_i(t)\big)-g(t,x)\Big]
\\[2mm]
&\quad\times\Big[2\omega(x)-g_1(t,x)-g_2(t,x)-\widehat{\xi}_i\big(t,x+\gamma_i(t)\big)+\xi_{i-1}\big(t,x+\gamma_i(t)\big)
-\xi_i\big(t,x+\gamma_i(t)\big)+g(t,x)\Big],
\end{align*}
where
$$
\abs{g(t,x)}=O\left(t^{-\frac{1}{2}-\frac{\sigma}{4\sqrt{2}}}\right)
$$
for all
$$
x\in\widehat{J}_i(t):=\left[{\frac{\gamma_{i-1}(t)-\gamma_{i}(t)}{2}},{\frac{\gamma_{i+1}(t)-\gamma_{i}(t)}{2}}\right]
=\left[-\frac{\eta_{i}(t)}{2},\frac{\eta_{i+1}(t)}{2}\right].
$$
Thus, combining the above identities and Lemmas \ref{esxi}-\ref{lem33},
we get that the first term at the right hand side of \eqref{esee1} can be computed in the following way
\begin{align}\label{esee2}
\nonumber&d_i(t)\int_{J_i(t)}\Big[\omega\big(x-\gamma_i(t)\big)+(-1)^{i+1}z(t,x)\Big]\Big[\omega\big(x-\gamma_i(t)\big)-(-1)^{i+1}z(t,x)\Big]\big[\omega'\big(x-\gamma_i(t)\big)\big]^2\,{\mathrm d}x
\\[2mm]\nonumber&=2d_i(t)\int_{\widehat{J}_i(t)}
\Big[g_1(t,x)+g_2(t,x)+\widehat{\xi}_i\big(t,x+\gamma_i(t)\big)-\xi_{i-1}\big(t,x+\gamma_i(t)\big)+\xi_i\big(t,x+\gamma_i(t)\big)\Big]
\\[2mm]\nonumber
&\qquad\qquad\qquad\times\big[\omega'(x)\big]^2\omega(x)\,{\mathrm d}x
\\[2mm]\nonumber
&\quad
+O\Bigg(t^{-\frac{1}{2}}\int_{{\mathbb R}}\Bigg[\sum_{i,j=1}^2\abs{g_i(t,x)g_j(t,x)}
  +\sum_{i=1}^2t^{-\frac{1}{2}}\abs{g_i(t,x)}\Bigg]\big[\omega'(x)\big]^2\,{\mathrm d}x
  \,+\, t^{-1-\frac{\sigma}{4\sqrt{2}}}\Bigg)
  \\[2mm]\nonumber
  &=2d_i(t)\int_{\mathbb R}
\Big[\widetilde{\xi}_i(t,x)-\xi_{i-1}\big(t,x+\gamma_i(t)\big)+\xi_i\big(t,x+\gamma_i(t)\big)\Big]\big[\omega'(x)\big]^2\omega(x)\,{\mathrm d}x
  \\[2mm]
  &\quad
  +2d_i(t)\int_{-\frac{1}{2}\eta_{i}(t)}^{\frac{1}{2}\eta_{i+1}(t)}
\Big[g_1(t,x)+g_2(t,x)\Big]\big[\omega'(x)\big]^2\omega(x)\,{\mathrm d}x
\ +\ O\left(t^{-1-\frac{\sigma}{4\sqrt{2}}}\right),
\end{align}
where in the last equality, we have used the following facts
\begin{align*}
\int_{\mathbb R} \abs{g_1(t,x)}[\omega'(x)]^2\,{\mathrm d}x
&=\int_{\mathbb R}\Big[1-\omega\big(x+\gamma_i(t)-\gamma_{i-1}(t)\big)\Big]\big[\omega'(x)\big]^2
\,{\mathrm d}x
\\[2mm]
&=16\int_{\mathbb R}\frac{e^{2\sqrt{2}x}}{\big(1+e^{\sqrt{2}x}\big)^4\big(1+e^{\sqrt{2}\{x+\gamma_i(t)-\gamma_{i-1}(t)\}}\big)}\,{\mathrm d}x
\\[2mm]
&=8\sqrt{2}\int_0^\infty\frac{y}{\big(1+y\big)^4\big(1+ye^{\sqrt{2}\eta_i(t)}\big)}\,{\mathrm d}y
\\[2mm]
&=O\left(t^{-\frac{1}{2}}\right),
\end{align*}
and
\begin{align*}
\int_{\mathbb R}\big[g_1(t,x)\big]^2\big[\omega'(x)\big]^2\,{\mathrm d}x
&=\int_{\mathbb R}\Big[1-\omega\big(x+\gamma_i(t)-\gamma_{i-1}(t)\big)\Big]^2\big[\omega'(x)\big]^2\,{\mathrm d}x
\\[2mm]&=32\int_{\mathbb R}\frac{e^{2\sqrt{2}x}}{\big(1+e^{\sqrt{2}x}\big)^4\big(1+e^{\sqrt{2}\{x+\gamma_i(t)-\gamma_{i-1}(t)\}}\big)^2}\,{\mathrm d}x
\\[2mm]&=16\sqrt{2}\int_0^\infty\frac{y}{\big(1+y\big)^4\big(1+ye^{\sqrt{2}\eta_i(t)}\big)^2}\,{\mathrm d}y
\\[2mm]&=O\left(t^{-1}\right).
\end{align*}
Similarly, there also hold that
\begin{align*}
\int_{\mathbb R} g_2(t,x)\big[\omega'(x)\big]^2\,{\mathrm d}x
&=\int_{\mathbb R}\Big[1+\omega\big(x+\gamma_i(t)-\gamma_{i+1}(t)\big)\Big]\big[\omega'(x)\big]^2\,{\mathrm d}x
\\[2mm]&=\int_{\mathbb R} \Big[1-\omega\big(x+\gamma_{i+1}(t)-\gamma_i(t)\big)\Big]\big[\omega'(x)\big]^2\,{\mathrm d}x
\\[2mm]&=O\left(t^{-\frac{1}{2}}\right),
\end{align*}
and
\begin{align*}
\int_{\mathbb R}\big[g_2(t,x)\big]^2\big[\omega'(x)\big]^2\,{\mathrm d}x
&=\int_{\mathbb R}\Big[1+\omega\big(x+\gamma_i(t)-\gamma_{i+1}(t)\big)\Big]^2\big[\omega'(x)\big]^2\,{\mathrm d}x
\\[2mm]
&=\int_{\mathbb R}\Big[1-\omega\big(x+\gamma_{i+1}(t)-\gamma_i(t)\big)\Big]^2\big[\omega'(x)\big]^2\,{\mathrm d}x
\\[2mm]
&=O\left(t^{-1}\right).
\end{align*}
By the definition of $\omega(x)$ in \eqref{domega}, we have that
\begin{align*}
&\int_{\mathbb R} \abs{g_2(t,x)g_1(t,x)}\big[\omega'(x)\big]^2\,{\mathrm d}x
\\[2mm]&=\int_{\mathbb R}\big[1-\omega\big(x+\gamma_i(t)-\gamma_{i-1}(t)\big)\big]\big[1+\omega\big(x+\gamma_i(t)-\gamma_{i+1}(t)\big)\big]\big[\omega'(x)\big]^2\,{\mathrm d}x
\\[2mm]&=32\int_{\mathbb R}\frac{e^{3\sqrt{2}x}}{\big(1+e^{\sqrt{2}x}\big)^4\big(1+e^{\sqrt{2}\{x+\eta_i(t)\}}\big)\big(e^{\sqrt{2}\eta_{i+1}(t)}+e^{\sqrt{2}x}\big)}\,{\mathrm d}x
\\[2mm]&=O\left(t^{-1}\right).
\end{align*}

In the rest part of this step, we estimate the terms in \eqref{esee2}.
This process includes the following three cases.

\noindent\textbf{Case one.}
We obtain that the estimates of the terms for $g_1(t,x)$ and $g_2(t,x)$. By the definition of $\omega(x)$ in \eqref{domega}, we have that
\begin{align*}
&\int_{-\frac{1}{2}\eta_{i}(t)}^{\frac{1}{2}\eta_{i+1}(t)} g_1(t,x)\omega(x)\big[\omega'(x)\big]^2\,{\mathrm d}x
\\[2mm]
&=-\int_{-\frac{1}{2}\eta_{i}(t)}^{\frac{1}{2}\eta_{i+1}(t)}
\Big[1-\omega\big(x+\gamma_i(t)-\gamma_{i-1}(t)\big)\Big]\omega(x)\big[\omega'(x)\big]^2\,{\mathrm d}x
\\[2mm]&=16\int_{-\frac{1}{2}\eta_{i}(t)}^{\frac{1}{2}\eta_{i+1}(t)}\frac{e^{2\sqrt{2}x}\big(1-e^{\sqrt{2}x}\big)}{\big(1+e^{\sqrt{2}x}\big)^5\big(1+e^{\sqrt{2}\{x+\eta_i(t)\}}\big)}\,{\mathrm d}x
\\[2mm]&=8\sqrt{2}\int_{e^{-\frac{\sqrt{2}}{2}\eta_{i}(t)}}^{e^{\frac{\sqrt{2}}{2}\eta_{i+1}(t)}}\frac{y\big(1-y\big)}{\big(1+y\big)^5\big(1+ye^{\sqrt{2}\eta_i(t)}\big)}\,{\mathrm d}y
\\[2mm]&=8\sqrt{2}\int_{e^{-\frac{\sqrt{2}}{2}\eta_{i}(t)}}^{\infty}\frac{y\big(1-y\big)}{\big(1+y\big)^5\big(1+ye^{\sqrt{2}\eta_i(t)}\big)}\,{\mathrm d}y+O\left(t^{-\frac{5}{4}}\right)
\\[2mm]&=\frac{4\sqrt{2}}{3}e^{-\sqrt{2}\eta_i(t)}+O\left(t^{-\frac{3}{4}}\right),
\end{align*}
and
\begin{align*}
&\int_{-\frac{1}{2}\eta_{i}(t)}^{\frac{1}{2}\eta_{i+1}(t)} g_2(t,x)\omega(x)\big[\omega'(x)\big]^2\,{\mathrm d}x
\\[2mm]
&=\int_{-\frac{1}{2}\eta_{i}(t)}^{\frac{1}{2}\eta_{i+1}(t)}
\Big[1+\omega\big(x+\gamma_i(t)-\gamma_{i+1}(t)\big)\Big]\omega(x)\big[\omega'(x)\big]^2\,{\mathrm d}x
\\[2mm]
&=-\int^{\frac{1}{2}\eta_{i}(t)}_{-\frac{1}{2}\eta_{i+1}(t)}
\Big[1-\omega\big(x+\gamma_{i+1}(t)-\gamma_i(t)\big)\Big]\omega(x)\big[\omega'(x)\big]^2\,{\mathrm d}x
\\[2mm]&=\frac{4\sqrt{2}}{3}e^{-\sqrt{2}\eta_{i+1}(t)}+O\left(t^{-\frac{3}{4}}\right).
\end{align*}

\noindent\textbf{Case two.}
We study the terms involving $\xi_i(t,x)$ and $\xi_{i+1}(t,x)$ in \eqref{esee2}.
By the definition of the function $\xi_i(t,x)$ in \eqref{exi}, we have that
\begin{align*}
&\int_{\mathbb R} \xi_i\big(t,x+\gamma_i(t)\big)\omega(x)\big[\omega'(x)\big]^2\,{\mathrm d}x
\\[2mm]&=\frac{24\sqrt{2}}{e^{\sqrt{2}\eta_{i+1}(t)}-1}\Bigg\{\int_0^\infty \left[ e^{\sqrt{2}\eta_{i+1}(t)}\ln\left(1+ye^{-\sqrt{2}\eta_{i+1}(t)}\right)-\ln\left(1+y\right)\right]\frac{y-1}{\big(y+1\big)^5}\,{\mathrm d}y
\\[2mm]&\qquad+\int_0^\infty \left[\ln\left(1+y\right)
-e^{-\sqrt{2}\eta_{i+1}(t)}\ln\left(1+ye^{\sqrt{2}\eta_{i+1}(t)}\right)\right]\frac{1-y}{\big(y+1\big)^5}\,{\mathrm d}y\Bigg\}
\\[2mm]&=\frac{2}{3}\sqrt{2}e^{-\sqrt{2}\eta_{i+1}(t)}+O\left(t^{-1}\ln t\right),
\end{align*}
and
\begin{align*}
&\int_{\mathbb R} \xi_{i-1}\big(t,x+\gamma_i(t)\big)\omega(x)\big[\omega'(x)\big]^2\,{\mathrm d}x
\\[2mm]&=\frac{24\sqrt{2}}{e^{\sqrt{2}\eta_{i}(t)}-1}\Bigg\{\int_0^\infty \left[ \ln\left(1+y\right)- e^{-\sqrt{2}\eta_{i}(t)}\ln\left(1+ye^{\sqrt{2}\eta_{i}(t)}\right)\right]\frac{y-1}{\big(y+1\big)^5}\,{\mathrm d}y
\\[2mm]&\qquad+\int_0^\infty \left[
e^{\sqrt{2}\eta_{i}(t)}\ln\left(1+ye^{-\sqrt{2}\eta_{i}(t)}\right)-\ln\left(1+y\right)\right]\frac{1-y}{\big(y+1\big)^5}\,{\mathrm d}y\Bigg\}
\\[2mm]&=-\frac{2}{3}\sqrt{2}e^{-\sqrt{2}\eta_{i}(t)}+O\left(t^{-1}\ln t\right).
\end{align*}

\noindent\textbf{Case three.}
We study the term involving $\widetilde{\xi}_i(t,x)$ in \eqref{esee2}. Using the integration by parts and the definition of $\widetilde{\xi}_i(t,x)$ in \eqref{dxi1}, we have that
\begin{align*}
\int_{\mathbb R} \widetilde{\xi}_{i}(t,x)\omega(x)\big[\omega'(x)\big]^2\,{\mathrm d}x
&=\int_{\mathbb R}\left\{\int_0^x\big(\omega'(s)\big)^{-2}\,{\mathrm d}s\int_{-\infty}^s\widetilde{\textrm{I}}_{i}(t,\tau)\omega'(\tau)\mathrm{d}\tau\right\}\omega(x)\big[\omega'(x)\big]^3\,{\mathrm d}x,
\\[2mm]&=\frac{\sqrt{2}}{6}\int_{\mathbb R} \omega'(x)\int_{-\infty}^x\widetilde{\textrm{I}}_{i}(t,\tau)\omega'(\tau)\mathrm{d}\tau \,{\mathrm d}x
\\[2mm]&=-\frac{\sqrt{2}}{6}\int_{\mathbb R} \widetilde{\textrm{I}}_{i}(t,x)\omega'(x)\omega(x)\,{\mathrm d}x,
\end{align*}
where we have used the following facts that
\begin{equation}\label{fact2}
  \sqrt{2}\omega'(x)=1-\omega^2(x)\qquad \text{and}\qquad\int_{\mathbb R}\widetilde{\textrm{I}}_{i}(t,\tau)\omega'(\tau)d\tau=0.
\end{equation}
Moreover, according to the definition of $\widetilde{\textrm{I}}_{i}(t,x)$ in \eqref{dIi} and $\omega(x)$ is an odd function, we have that
\begin{align*}
&\int_{\mathbb R} \widetilde{\xi}_{i}(t,x)\omega(x)\big[\omega'(x)\big]^2\,{\mathrm d}x
=-\frac{\sqrt{2}}{6}\int_{\mathbb R} \widetilde{\textrm{I}}_{i}(t,x)\omega'(x)\omega(x)\,{\mathrm d}x
\\[2mm]&=-\frac{\sqrt{2}}{2}\int_{\mathbb R}\left\{\big(1-\omega(x)\big)^2\big[\omega(x-\eta_{i+1}(t))+1\big]-\sqrt{2}\omega'(x)\xi_i\big(t,x+\gamma_i(t)\big)\right\}\omega'(x)\omega(x)\,{\mathrm d}x
\\[2mm]&\quad-\frac{\sqrt{2}}{2}\int_{\mathbb R}\left\{\big(1+\omega(x)\big)^2\big[\omega(x+\eta_{i}(t))-1\big]+\sqrt{2}\omega'(x)\xi_{i-1}\big(t,x+\gamma_i(t)\big)\right\}\omega'(x)\omega(x)\,{\mathrm d}x,
\\[2mm]&=\frac{2}{3}\sqrt{2}e^{-\sqrt{2}\eta_{i}(t)}+\frac{2}{3}\sqrt{2}e^{-\sqrt{2}\eta_{i+1}(t)}+O\left(t^{-1}\ln t\right).
\end{align*}
where we have used the first equality in \eqref{fact2} again, and the following two equalities
\begin{align*}
\\&\int_{\mathbb R}\big(1-\omega(x)\big)^2\big[\omega(x-\eta_{i+1}(t))+1\big]\omega'(x)\omega(x)\,{\mathrm d}x
\\[2mm]&=16\sqrt{2}\int_{\mathbb R}\frac{e^{2\sqrt{2}x}\big(e^{\sqrt{2}x}-1\big)}{\big(1+e^{\sqrt{2}x}\big)^5\big(e^{\sqrt{2}\eta_i(t)}+e^{\sqrt{2}x}\big)}\,{\mathrm d}x
\\[2mm]&=16\int_0^\infty\frac{y\big(y-1\big)}{\big(1+y\big)^5\big(e^{\sqrt{2}\eta_i(t)}+y\big)}\,{\mathrm d}y
\\[2mm]&=O\left(t^{-1}\right),
\end{align*}
and
 \begin{align*}
 &\int_{\mathbb R}\big(1+\omega(x)\big)^2\big[\omega(x+\eta_{i}(t))-1\big]\omega'(x)\omega(x)\,{\mathrm d}x
 \\[2mm]&=\int_{\mathbb R}\big(1-\omega(x)\big)^2\big[1+\omega(x-\eta_{i}(t))\big]\omega'(x)\omega(x)\,{\mathrm d}x
 \\[2mm]&=O\left(t^{-1}\right).
\end{align*}

Combing the above arguments of all three cases and the  equalities  \eqref{esee1} and \eqref{esee2} , we get that
\begin{align}
  \mathfrak{e}_1(t)=&(-1)^{i+1}16\sqrt{2}\left\{ d_i(t)\left[e^{-\sqrt{2}\eta_{i+1}(t)}+e^{-\sqrt{2}\eta_{i}(t)}\right]-d_{i+1}(t)e^{-\sqrt{2}\eta_{i+1}(t)}-d_{i-1}(t)e^{-\sqrt{2}\eta_{i}(t)}\right\}
\nonumber  \\[2mm]
  &+O\left(t^{-1-\frac{\sigma}{4\sqrt{2}}}\right).
  \label{5.8-11}
\end{align}
This finishes all estimates in Step 3.

Finally, by combining the estimates in Step 1 through Step 3, see \eqref{dEE}, \eqref{eseee1} and \eqref{eseee2} and \eqref{5.8-11}, we obtain the estimate of the projection of $E(t,x)$ as the following
\begin{align}\label{dEEE}
&\int_{\mathbb R} E(t,x)\omega'\big(x-\gamma_i(t)\big)\,{\mathrm d}x
=(-1)^{i+1}\frac{2\sqrt{2}}{3}\gamma'_i(t)
\,-\,\mathfrak{e}_1(t)
\,+\, O\left(t^{-1-\frac{\sigma}{4\sqrt{2}}}\right)
  \nonumber\\[2mm]
&=\,
-(-1)^{i+1}16\sqrt{2}\bigg\{
d_i(t)\left[e^{-\sqrt{2}\eta_{i+1}(t)}+e^{-\sqrt{2}\eta_{i}(t)}\right]
-d_{i+1}(t)e^{-\sqrt{2}\eta_{i+1}(t)}-d_{i-1}(t)e^{-\sqrt{2}\eta_{i}(t)}
\bigg\}
  \nonumber\\[2mm]
&\quad \,+\, (-1)^{i+1}\frac{2\sqrt{2}}{3}\gamma_i'(t) \,+\, O\left(t^{-1-\frac{\sigma}{4\sqrt{2}}}\right).
\end{align}

\subsubsection{}\label{section5.1.2}
Next we estimate other terms in \eqref{eq7.12}. By the definitions of $N(\phi)$ and $\Phi(t,s)$ in \eqref{eq7.6} and \eqref{dphi} respectively,  and the condition \eqref{spsi}, we have that
\begin{align*}
  \int_{{\mathbb R}}\abs{N(\phi)}\omega'\big(x-\gamma_i(t)\big)\,{\mathrm d}x&\leq C\int_{{\mathbb R}}\big{(}\phi(t,x)\big{)}^2\omega'\big(x-\gamma_i(t)\big)\,{\mathrm d}x
  \\[2mm]
  &\leq C\int_{{\mathbb R}}\big{(}\Phi(t,x)\big{)}^2\omega'\big(x-\gamma_i(t)\big)\,{\mathrm d}x\\[2mm]
&\leq C t^{-\frac{3}{2}},
\end{align*}
where the constant $C>0$ does not depend on $t$.

 According to the definition of the operator $F'(u)[v]$ in \eqref{deF'}, the condition \eqref{spsi}, the equalities \eqref{esee} and \eqref{estmE}, the Taylor expansion and integration by parts, we have that

\begin{align*}
&\int_{{\mathbb R}} F'(z(t,x))[\phi]\omega'\big(x-\gamma_i(t)\big)\,{\mathrm d}x
\\[2mm]
&=\int_{{\mathbb R}}\Big{(}\phi_{xx}-W''(z(t,x))\phi\Big{)} \Big{(}W''(z(t,x))-W''\big(\omega\big(x-\gamma_i(t)\big)\big)\Big{)}\omega'\big(x-\gamma_i(t)\big)\,{\mathrm d}x
\\[2mm]
&\qquad-\int_{{\mathbb R}} W'''(z(t,x))\Big{[}\partial_{xx}z(t,x)-W'(z(t,x))\Big{]}\phi(t,x)\omega'\big(x-\gamma_i(t)\big)\,{\mathrm d}x
\\[2mm]
&\leq Ct^{-\frac{3}{4}-\frac{\sigma}{8\sqrt{2}}}\bigg\{\int_{{\mathbb R}}\abs{W''(z(t,x))-W''\big(\omega\big(x-\gamma_i(t)\big)\big)}\omega'\big(x-\gamma_i(t)\big)\,{\mathrm d}x
\\[2mm]
&\qquad\qquad  +\int_{{\mathbb R}}\abs{W'''(z(t,x))\Big{[}\partial_{xx}z(t,x)-W'(z(t,x))\Big{]}}\omega'\big(x-\gamma_i(t)\big)\,{\mathrm d}x\bigg\}
\\[2mm]
&\leq  Ct^{-\frac{3}{4}-\frac{\sigma}{8\sqrt{2}}}\bigg\{\int_{{\mathbb R}}\Big[\abs{\omega\big(x-\gamma_{i-1}(t)\big)-1}
+\abs{\omega\big(x-\gamma_{i+1}(t)\big)+1}\Big]\omega'\big(x-\gamma_i(t)\big)\,{\mathrm d}x
 \ +\ t^{-\frac{1}{2}}\bigg\}
 \\[2mm]
&\leq Ct^{-1-\frac{\sigma}{4\sqrt{2}}},
\end{align*}
where $\sigma\in(0,\sqrt{2})$, and the constant $C>0$ is independent of $t$.
Here we have used the following facts that
\begin{align*}
&\int_{{\mathbb R}}\big[1-\omega\big(x-\gamma_{i-1}(t)\big)\big]\omega'\big(x-\gamma_i(t)\big)\,{\mathrm d}x
\\[2mm]
&=\int_{{\mathbb R}}\big[1-\omega\big(x+\gamma_i(t)-\gamma_{i-1}(t)\big)\big]\omega'(x)\,{\mathrm d}x
\\[2mm]&=4\sqrt{2}\int_{\mathbb R}\frac{e^{\sqrt{2}x}}{\big(1+e^{\sqrt{2}x}\big)^2\big(1+e^{\sqrt{2}\{x+\gamma_i(t)-\gamma_{i-1}(t)\}}\big)}\,{\mathrm d}x
\\[2mm]&=4\int_0^\infty\frac{1}{\big(1+y\big)^2\big(1+ye^{\sqrt{2}\eta_i(t)}\big)}\,{\mathrm d}y
\\[2mm]&=O\left(t^{-\frac{1}{2}}\ln t\right),
\end{align*}
and
\begin{align*}
&\int_{{\mathbb R}}\big[1+\omega\big(x-\gamma_{i+1}(t)\big)\big]\omega'\big(x-\gamma_i(t)\big)\,{\mathrm d}x
\\[2mm]&=\int_{{\mathbb R}}\big[1+\omega\big(x+\gamma_i(t)-\gamma_{i+1}(t)\big)\big]\omega'(x)\,{\mathrm d}x
\\[2mm]&=\int_{{\mathbb R}}\big[1-\omega\big(x+\gamma_{i+1}(t)-\gamma_i(t)\big)\big]\omega'(x)\,{\mathrm d}x
\\[2mm]&=O\left(t^{-\frac{1}{2}}\ln t\right).
\end{align*}

\subsubsection{}\label{section5.1.3}
We finally draw a conclusion of Section \ref{section5.1}.
Combining the results in Section \ref{section5.1.1} (see \eqref{dEEE}) and also those in Section \ref{section5.1.2} together with the estimates of the functions $d_i(t)$ in Lemma \ref{lem33},
we can obtain that (\ref{eq7.12}) is equivalent to the following ODE system:
\begin{align}\label{system1}
\nonumber&\frac{2\sqrt{2}}{3}\gamma'_i(t)
-384\bigg\{
\,-\,e^{-\sqrt{2}\{\gamma_{i}(t)-\gamma_{i-2}(t)\}}
\,+\,2e^{-2\sqrt{2}\{\gamma_i(t)-\gamma_{i-1}(t)\}}
\\
&\qquad\qquad\qquad\qquad\,-\,2e^{-2\sqrt{2}\{\gamma_{i+1}(t)-\gamma_i(t)\}} \,+\,e^{-\sqrt{2}\{\gamma_{i+2}(t)-\gamma_i(t)\}}
\bigg\}=G_i(\vec{\gamma},\vec{\gamma}'),
  \end{align}
with $i=1,\cdots,k$, where we assume that $\gamma_0=\gamma_{-1}=-\infty$ and $\gamma_{k+1}=\gamma_{k+2}=+\infty$. We recall that $\gamma_i(t)=\gamma^0_i+h_i(t)$, and $\gamma^0_i$ is given by \eqref{degama0} for $i=1,\cdots,k$.
The function $\vec{h}$ belongs to the space $\Lambda$,
$$\Lambda=\Big{\{}\vec{h}=[h_1,\cdots,h_k]^\top
\,:\, h_i\in C^1(-\infty, -T]\ \text{and} \ \|\vec{h}\|_{\Lambda}\leq 1\Big{\}}$$
with $T>T_1$, where $T_1$ is given by Proposition \ref{prop3}. and the norm
\begin{equation*}
  \ \|\vec{h}\|_{\Lambda}:=\sum_{i=1}^k\sup_{t\geq T}\Big\{\abs{h_i(t)}+\abs{t}\abs{ h'_i}\Big\}.
\end{equation*}

We set
\begin{equation*}
\hat{G}_i(\vec{h},\vec{h}') :=G_i(\vec{\gamma},\vec{\gamma}'),\ \forall\, i=1,\cdots, k,
\qquad  \text{and}\qquad
\hat{\mathbf{G}}:=\big[\hat{G}_1,\cdots,\hat{G}_k\big]^\top.
\end{equation*}
According to the above arguments, Proposition \ref{prop3} and Lemma \ref{lem10}, we have that
\begin{prop}\label{prop9}
Let $\sigma\in\left(0,\sqrt{2}\ \right]$ and $\vec{h},\ \vec{h}^1, \ \vec{h}^2\in\Lambda$. Then we have that
\begin{equation*}
  \abs{\hat{\mathbf{G}}\left(\vec{h},\vec{h}'\right)}\leq Ct^{-1-\frac{\sigma}{4\sqrt{2}}},
\end{equation*}
and
\begin{equation*}
  \abs{\hat{\mathbf{G}}\left(\vec{{h}}^1,(\vec{h}^1)'\right)-\hat{\mathbf{G}}\left(\vec{h}^2,(\vec{h}^2)'\right)}\leq C t^{-1-\frac{\sigma}{4\sqrt{2}}}\|\vec{h}^1-\vec{h}^2\|_{\Lambda},
\end{equation*}
where $C$ is a uniform positive constant.
\qed
\end{prop}
 In Sections \ref{section5.2}-\ref{section5.3}, we will prove that the Toda system (\ref{system1}) is solvable by using some ideas in the references \cite{dG1,dG2}.

\subsection{The choice of $\vec{\gamma}^0$}\label{section5.2}
Let $k\geq2$.
In order to solve the Toda system \eqref{system1}, we first study the solvability of the following Toda system:
\begin{align}\label{system2}
\nonumber \frac{2\sqrt{2}}{3}\gamma'_i(t)=R_i(t):=
&384\Big\{
\,-\,e^{-\sqrt{2}\{\gamma_{i}(t)-\gamma_{i-2}(t)\}}
\,+\,2e^{-2\sqrt{2}\{\gamma_i(t)-\gamma_{i-1}(t)\}}
\\
&\qquad \qquad \,-\,2e^{-2\sqrt{2}\{\gamma_{i+1}(t)-\gamma_i(t)\}} \,+\,e^{-\sqrt{2}\{\gamma_{i+2}(t)-\gamma_i(t)\}}\Big\},
\end{align}
with $i=1,\cdots,k$, where we assume that
$$
\gamma_{-1}(t)=\gamma_0(t)=-\infty
\quad\mbox{ and }\quad
\gamma_{k+1}(t)=\gamma_{k+2}(t)=+\infty.
$$
The main result is given by the following lemma.

\begin{lem}\label{lem11}
 Let $k\geq2$ and the constants
\begin{align*}
-a_i=a_{k-i+1}:=\frac{1}{2\sqrt{2}}\sum_{l=i+1}^{k-i+1}\ln \left[\frac{(l-1)(k-l+1)}{2}\right],\qquad \text{for all}\ \ 2i\leq k,
\end{align*}
and $ a_{\frac{k+1}{2}}=0$ when $k$ is an odd number.
Then the above Toda system \eqref{system2} has a solution $\vec{\gamma}^0=\big[\gamma_1^0(t),\cdots,\gamma_k^0(t)\big]^\top$ with
\begin{equation}\label{eq8.4}
  \gamma^0_i(t):=\frac{1}{2\sqrt{2}}\left(i-\frac{k+1}{2}\right)\ln\left[1152t\right]+a_i,
\end{equation}
with $i=1,\cdots,k$.
\end{lem}

\begin{proof}
Let $\vec{\gamma}^0=\big[\gamma^0_1(t),\cdots,\gamma^0_k(t)\big]^\top$ be a solution of \eqref{system2} and $\eta^0_l(t):=\gamma^0_{l}(t)-\gamma^0_{l-1}(t)$ for $l=2,\cdots,k$. Then according to system \eqref{system2}, the functions $\eta^0_i(t)$ with $2\leq i\leq k$ satisfy that
\begin{align}\label{system3}
\nonumber\frac{2\sqrt{2}}{3}\eta'_{i}(t)=&R_{i+1}(t)-R_i(t)
\\[2mm]
=&\nonumber384\Big\{-2e^{-2\sqrt{2}\eta_{i-1}(t)}+4e^{-2\sqrt{2}\eta_{i}(t)}-2e^{-2\sqrt{2}\eta_{i+1}(t)}-e^{-\sqrt{2}\{\eta_{i}(t)+\eta_{i+1}(t)\}}
\\[2mm]
&\quad+e^{-\sqrt{2}\{\eta_{i-1}(t)+\eta_{i-2}(t)\}}-e^{-\sqrt{2}\{\eta_{i}(t)+\eta_{i-1}(t)\}}+e^{-\sqrt{2}\{\eta_{i+1}(t)+\eta_{i+2}(t)\}}\Big\},
\end{align}
where we assume that $\eta_0(t)=\eta_1(t)=\eta_{k+1}(t)=\eta_{k+2}(t)=+\infty$.

We want to find a solution of \eqref{system3} of the form
$$
\eta_i(t)=\frac{1}{2\sqrt{2}}\ln\big[1152t\big]-\frac{1}{\sqrt{2}}\ln x_i,
\quad i=2,\cdots,k.
$$
Thus the constants $x_2,\cdots,x_k$ satisfy that the following system
\begin{align}\label{system5}
\left\{
\begin{array}{l}
4(x_2)^2-2(x_3)^2-x_2x_3+x_3x_4=1,
\\[2mm]-2(x_2)^2+4(x_3)^2-2(x_4)^2-x_3x_4-x_2x_3+x_4x_5=1,
\\[2mm]-2(x_3)^2+4(x_4)^2-2(x_5)^2-x_4x_5+x_2x_3-x_3x_4+x_5x_6=1,
\\[2mm]\cdots\cdots
\\[2mm]-2(x_{i-1})^2+4(x_i)^2-2(x_{i+1})^2-x_ix_{i+1}+x_{i-1}x_{i-2}-x_{i-1}x_i+x_{i+1}x_{i+2}=1,
\\[2mm]\cdots\cdots
\\[2mm]-2(x_{k-2})^2+4(x_{k-1})^2-2(x_k)^2-x_{k-1}x_{k}+x_{k-3}x_{k-2}-x_{k-1}x_{k-2}=1,
\\[2mm]-2(x_{k-1})^2+4(x_k)^2+x_{k-2}x_{k-1}-x_{k-1}x_k=1.
\end{array}
\right.
\end{align}
It is easy to check that the above system has a unique positive solution
\begin{align}\label{dxhat}
2x_i^0=2x_{k+2-i}^0=(i-1)k-\left(i-1\right)^2=\left(i-1\right)\left(k-i+1\right),
\end{align}
with $i=2,\cdots,k$. Hence, we get that system \eqref{system3} has a solution
\begin{align*}
  \eta^0_i(t):=\frac{1}{2\sqrt{2}}\ln\big[1152t\big]-\frac{1}{\sqrt{2}}\ln \left[\frac{(i-1)(k-i+1)}{2}\right],
\end{align*}
with $i=2,\cdots,k$.

Furthermore, using the facts that $\eta^0_l(t)=\gamma^0_{l}(t)-\gamma^0_{l-1}(t)$ for all $l=2,\cdots,k$,  and

$$\sum_{j=1}^k\gamma_j^0(t)=0,$$
we obtain that system \eqref{system2} has a solution of the form
\begin{align*}
  \gamma^0_i(t)=\frac{1}{2\sqrt{2}}\left(i-\frac{k+1}{2}\right)\ln\left[1152t\right]+a_i,
\end{align*}
where the constants $a_1,\cdots,a_k$ are defined by
\begin{align*}
 -a_i=a_{k-i+1}=\frac{1}{2\sqrt{2}}\sum_{l=i+1}^{k-i+1}\ln \left[\frac{(l-1)(k-l+1)}{2}\right],\qquad \text{for}\ \ 2i\leq k,
\end{align*}
and when $k$ is odd, we choose
  $ a_{\frac{k+1}{2}}=0.$ The proof is completed.
\end{proof}

\subsection {The solvability of the reduced system}\label{section5.3}
The resolution theory of system \eqref{system1} will be divided into several steps.

\noindent{\textbf{Step 1.}}
Keeping the notation of the previous subsection, we look for a solution of the form $\vec{\gamma}=\vec{\gamma}^0+\vec{h}$ to system \eqref{system1}. Then the function $\vec{h}$ satisfies that
\begin{align}\label{system4}
\nonumber \vec{h}'+D_{\vec{\gamma}}\mathbf{R}(\vec{\gamma}^0)\vec{h}
=&\mathbf{G}\big(\vec{\gamma}^0+\vec{h}, (\vec{\gamma}^0)'+\vec{h}'\big)
+D_{\vec{\gamma}}\mathbf{R}(\vec{\gamma}^0)\vec{h}-\mathbf{R}(\vec{\gamma}^0)
\\[2mm]
=&\mathbf{Q}(\vec{h},\vec{h}')\qquad \text{in}\ [T_0,+\infty),
\end{align}
where the matrices
$$
\mathbf{G}(\vec{\gamma},\vec{\gamma}'):=\big[G_1(\vec{\gamma},\vec{\gamma}'),\cdots,G_k(\vec{\gamma},\vec{\gamma}')\big]^\top \qquad\text{and} \qquad
\mathbf{R}(\vec{\gamma}):=\big[R_1(\vec{\gamma}),\cdots,R_k(\vec{\gamma})\big]^\top.
$$
In the above, the functions $G_i(\vec{\gamma},\vec{\gamma}')$
and $R_i(\vec{\gamma})$ are given by \eqref{system1} and \eqref{system2} respectively.
 Recall that the functions $\gamma_1(t),\cdots,\gamma_k(t)$ satisfy that
 $$
 \gamma_i(t)=-\gamma_{\frac{k+1}{2}+i}(t),\qquad i\leq\frac{k+1}{2},
 $$
 so the functions $h_1,\cdots,h_k(t)$ do.
 This implies that the solution $\phi(t,x)$ to \eqref{eq7.7}-\eqref{eq7.8} is even with respect to $x$ and
 thus we obtain that
 \begin{align}\label{eq8.1}
 Q_i(\vec{h},\vec{h}')=-Q_{\frac{k+1}{2}+i}(\vec{h},\vec{h}')\qquad \text{for all}\ i\leq\frac{k+1}{2}.
 \end{align}

\noindent{\textbf{Step 2.}}
We intend to solve system \eqref{system4}.
Let

\begin{equation*}
 \mathcal{B}:=\left(
      \begin{array}{cccccc}
        -1 & 1 & 0 & 0&\cdots & 0 \\
        0 & -1 & 1  &0&\cdots & 0 \\
       \vdots & \ddots & \ddots & \ddots&\ddots&\vdots \\
       0&\cdots&0&-1&1&0\\
        0 & \cdots &0& 0 & -1 & 1 \\
        1 & \cdots &1&1 &1 & 1 \\
      \end{array}
    \right)_{k\times k},\ \
    \overline{\mathcal{B}}:=\left(
      \begin{array}{cccccc}
        -1 & 1 & 0 & \cdots & 0&0 \\
        0 & -1 & 1 & \cdots & 0&0 \\
       \vdots & \ddots & \ddots & \ddots & \vdots&\vdots \\
        0 & \cdots & 0 & -1 & 1&0 \\
        0 & \cdots & 0 &0 & -1&1 \\
      \end{array}
    \right)_{(k-1)\times k},
 \end{equation*}

\noindent and
$$
\vec{p}:=\mathbf{\mathcal{B}}\vec{h}.
$$
Then $\vec{p}=\big[p_1(t),\cdots,p_k(t)\big]^\top$ satisfies that
\begin{align}\label{eq8.2}
\vec{p}\,'+D_{\vec{\vartheta}}\mathbf{S}\big(\vec{\vartheta}^0\big)\vec{p}
=\mathbf{\mathcal{B}}\mathbf{Q}\left(\mathbf{\mathcal{B}}^{-1}\vec{p},\mathbf{\mathcal{B}}^{-1}\vec{p}\,'\right)\qquad\text{in}\ [T_0,+\infty),
\end{align}
 where the matrices

 \begin{align*}
\mathbf{S}(\vec{\vartheta}):=\mathbf{\mathcal{B}}\mathbf{R}(\vec{\gamma}),\qquad \vec{\vartheta}:=\mathbf{\mathcal{B}}\vec{\gamma}\quad \text{and}\quad\vec{\vartheta}^0:=\mathbf{\mathcal{B}}\vec{\gamma}^0.
\end{align*}

\noindent
Thus by the relations in \eqref{eq8.1}, we have that system  \eqref{eq8.2} is equivalent to
 \begin{align}\label{eq8.3}
 \left\{
\begin{array}{c}
\ \overline{p}\,' \, +\, D_{\overline{\vartheta}}\overline{\mathbf{S}}\big(\overline{\vartheta}^0\big)\overline{p}
=\mathbf{M}(\vec{p},\vec{p}\,')\qquad\text{in}\ [T_0,+\infty),
\\[2mm]
\qquad p'_k(t)=0\qquad\text{in}\ [T_0,+\infty),
\end{array}
\right.
\end{align}
where the matrices
$$
\overline{p}:=\big[p_1(t),\cdots,p_{k-1}(t)\big]^\top,
\qquad
\overline{\vartheta}:=\overline{\mathbf{\mathcal{B}}}\vec{\gamma},
\qquad
\overline{\vartheta}^0:=\overline{\mathbf{\mathcal{B}}}\vec{\gamma}^0,
$$
and
$$
\overline{\mathbf{S}}(\overline{\vartheta}):=\overline{\mathbf{\mathcal{B}}}\mathbf{R}(\vec{\gamma}),
\qquad
\mathbf{M}(\vec{p},\vec{p}\,'):=\overline{\mathbf{\mathcal{B}}}\mathbf{Q}\big(\mathbf{\mathcal{B}}^{-1}\vec{p},\mathbf{\mathcal{B}}^{-1}\vec{p}\,'\big).
$$
For the sake of simplicity, we choose the function $p_k=0$ in system \eqref{eq8.3}.

\noindent{\textbf{Step 3.}}
We now consider system \eqref{eq8.3}. By the expression of $R_{i+1}(t)-R_i(t)$ in \eqref{system3} and the definition of $\gamma_i^0(t)$ in \eqref{eq8.4}, we have that
\begin{align*}
D_{\overline{\vartheta}}\overline{\mathbf{S}}\big(\overline{\vartheta}^0\big)=-\frac{1}{2t}\mathbf{\mathcal{H}}=-\frac{1}{2t}\left[H_{lj}\right]_{(k-1)\times(k-1)}=-\frac{1}{2t}\left[\frac{\partial H_{l+1}(\widehat{x})}{\partial x_{j+1}}x_{j+1}\Bigg{|}_{\widehat{x}=\widehat{x}^0}\right]_{(k-1)\times(k-1)},
\end{align*}
where
$$
\widehat{x}:=\big[x_2,\cdots,x_k\big]^\top,
\qquad
\widehat{x}^0:=\big[x_2^0,\cdots,x_k^0\big]^\top,
$$
with the number $x_i^0$ defined by \eqref{dxhat}, and
$$
H_i(\widehat{x}):=-2(x_{i-1})^2+4(x_i)^2-2(x_{i+1})^2-x_ix_{i+1}+x_{i-1}x_{i-2}-x_{i-1}x_i+x_{i+1}x_{i+2},
$$
with $i=2,\cdots,k$, and with the conventions that
$$
x_0=x_1=x_{k+1}=x_{k+2}=0.
$$
By some straightforward computations, we get that the eigenvalues of $\mathbf{\mathcal{H}}$
 are explicitly given by
 \begin{align*}
 \lambda_1=1\times 2,\ \ \lambda_2=3\times 4,\ \ \lambda_3=6\times 7,\ \ \lambda_4= 10\times 11,\ \cdots,\ \lambda_{k-1}=\frac{(k-1)k}{2}\times\left[\frac{(k-1)k}{2}+1\right].
 \end{align*}
Hence there exists an invertible matrix $\mathbf{\mathcal{D}}$ such that
\begin{align*}
\mathbf{\mathcal{D}}\mathbf{\mathcal{H}}\mathbf{\mathcal{D}}^{-1}=\mathbf{\mathcal{C}}:=\text{diag}\big\{\lambda_1,\lambda_2,\cdots,\lambda_{k-1}\big\}.
\end{align*}

 Let $\overline{q}:=\mathbf{\mathcal{D}}\overline{p}$. The first equation in \eqref{eq8.3} becomes equivalent to
\begin{align}\label{eq8.5}
\overline{q}\,'-\frac{1}{2t}\mathcal{C}\overline{q}=\mathbf{\mathcal{D}}\mathbf{M}\left(\mathbf{\mathcal{D}}^{-1}\overline{q},\ \mathbf{\mathcal{D}}^{-1}\overline{q}\,'\right)
\qquad\text{in}\ [T_0,+\infty),
\end{align}
where we have used the fact that $p_k(t)=0$.

 \noindent{\textbf{Step 4.}}
We will solve system \eqref{eq8.5} by applying the fixed-point theorem.
It is easy to check that $\overline{q}(t)=\big[q_1(t), \cdots, q_{k-1}(t))\big]^\top$ is a bounded solution of problem \eqref{eq8.5}  if and only if $\overline{q}$ satisfies that
\begin{equation}\label{eq8.6}
 q_i(t)=-t^{\frac{\lambda_i}{2}}\int^{+\infty}_ts^{-\frac{\lambda_i}{2}}{\mathcal J}_i\big(\overline{q}(s),\overline{q}\,'(s)\big) \,{\mathrm d}s\qquad \text{for all}\ i=1,\cdots,k-1,
\end{equation}
for all $t\geq T_0$, where we denote the $i$-th  component of the vector $\mathbf{\mathcal{D}}\mathbf{M}\left(\mathbf{\mathcal{D}}^{-1}\overline{q},\mathbf{\mathcal{D}}^{-1}\overline{q}\,'\right)$ by ${\mathcal J}_i\big(\overline{q},\overline{q}\,'\big)$.

Let $\mathbf{A}(\overline{q}):=\big[A_1(\overline{q}),\cdots,A_{k-1}(\overline{q})\big]^\top$ be a solution of \eqref{eq8.6}. By Proposition \ref{prop9} and \eqref{eq8.6}, we have that
\begin{equation}\label{eq8.7}
  \abs{\mathbf{A}(0)}\leq C_1\big[T_0\big]^{-\frac{\sigma}{4\sqrt{2}}}\quad \text{and}\quad \abs{t\mathbf{A}'(0)}\leq C_2\big[T_0\big]^{-\frac{\sigma}{4\sqrt{2}}},
\end{equation}
for all $t\geq T_0$, where the parameter $\sigma\in\left(0,\sqrt{2}\right)$, the constants $C_1, C_2>0$ are independent of $t$ and $T_0$,  and
\begin{equation*}
  \abs{\mathbf{A}(\overline{q})}:=\sum_{j=1}^{k-1}\abs{A_j(\overline{q})}.
\end{equation*}

Let $T_0>1$. We define a space
\begin{equation*}
  X:=\left\{\overline{q}(t)=\left[q_1(t),\cdots,q_{k-1}(t)\right]^\top\,:\  q_i(t)\in C^1[T_0,+\infty)\ \text{and}\ \sup_{t\geq0} \left[\abs{q_i(t)}+\abs{tq'_i(t)}\right]<+\infty \right\},
\end{equation*}
and its a closed subset
\begin{equation*}
  S_{c_0}:=\left\{\overline{q}(t)\in X\ :\ \|\overline{q}\|_\Lambda=\sum_{j=1}^{k-1}\sup_{t\geq T_0}\left[\abs{q_j(t)}+\abs{tq'_j(t)}\right]<2c_0 \right\},
\end{equation*}
where the constant $c_0:=C_1+C_2$ with the constants $C_1,C_2$ in \eqref{eq8.7}.  A similar argument in \eqref{eq8.7} yields that
\begin{align*}
\abs{\mathbf{A}\big(\overline{q}_1(t)\big)-\mathbf{A}\big(\overline{q}_2(t)\big)}\leq C\big[T_0\big]^{-\frac{\sigma}{4\sqrt{2}}}\|\overline{q}_1-\overline{q}_2\|_\Lambda,
\end{align*}
and
\begin{align*}
\abs{t}\abs{\mathbf{A}'\big(\overline{q}_1(t)\big)-\mathbf{A}'\big(\overline{q}_2(t)\big)}\leq C\big[T_0\big]^{-\frac{\sigma}{4\sqrt{2}}}\|\overline{q}_1-\overline{q}_2\|_\Lambda,
\end{align*}
for all $t\geq T_0$ and $\overline{q}_1,\overline{q}_2\in S_{c_0}$, where $C>0$ does not depend on $\overline{q}_1,\overline{q}_2$, $t$ and $T_0$.
Thus we get that
\begin{align*}
\|\mathbf{A}\big(\overline{q}_1(t)\big)-\mathbf{A}\big(\overline{q}_2(t)\big)\|_\Lambda\leq C\big[T_0\big]^{-\frac{\sigma}{4\sqrt{2}}}\|\overline{q}_1-\overline{q}_2\|_\Lambda,
\end{align*}
where $C>0$ only depends on $\sigma$. Hence Banach fixed-point theorem implies that problem \eqref{eq8.6} has a solution $\overline{p}\in S_{c_0}$ if we choose $T_0$ big enough.
Furthermore, according to \eqref{eq8.6}, there holds that
\begin{equation*}
  \abs{\overline{p}(t)}\leq Ct^{-\frac{\sigma}{4\sqrt{2}}}\qquad \text{as}\quad t\rightarrow+\infty,
\end{equation*}
where $\sigma\in\left(0,\sqrt{2}\right)$.

\section{Appendices}\label{appendix}
\subsection{The proof of Lemma \ref{esxi}}\label{section6.1}

By the definition of $\xi_i(t,x)$ in \eqref{dxi} and the integration by parts, we have that
\begin{align}\label{exi}
  \nonumber\xi_i(t,x)&=
  -12\sqrt{2}e^{\sqrt{2}x}\int_{\frac{\gamma_i(t)+\gamma_{i+1}(t)}{2}}^xe^{-2\sqrt{2}s}\int_0^{e^{\sqrt{2}s}}\frac{ye^{\sqrt{2}\gamma_i(t)}}{\big(y+e^{\sqrt{2}\gamma_i(t)}\big)
  \big(y+e^{\sqrt{2}\gamma_{i+1}(t)}\big)}\,{\mathrm d}y \,{\mathrm d}s
  \\[2mm]&\nonumber\qquad+12e^{\sqrt{2}[x-\gamma_i(t)-\gamma_{i+1}(t)]}\int_{0}^{e^{\frac{1}{\sqrt{2}}[\gamma_i(t)+\gamma_{i+1}(t)]}}\frac{ye^{\sqrt{2}\gamma_i(t)}}{\big(y+e^{\sqrt{2}\gamma_i(t)}\big)
  \big(y+e^{\sqrt{2}\gamma_{i+1}(t)}\big)}\,{\mathrm d}y
  \\[2mm]&\nonumber=6e^{-\sqrt{2}x}\int_0^{e^{\sqrt{2}x}}\frac{ye^{\sqrt{2}\gamma_i(t)}}{\big(y+e^{\sqrt{2}\gamma_i(t)}\big)
  \big(y+e^{\sqrt{2}\gamma_{i+1}(t)}\big)}\,{\mathrm d}y
  \\[2mm]&\nonumber\qquad+6e^{\sqrt{2}[x-\gamma_i(t)-\gamma_{i+1}(t)]}\int_{0}^{e^{\frac{1}{\sqrt{2}}[\gamma_i(t)+\gamma_{i+1}(t)]}}\frac{ye^{\sqrt{2}\gamma_i(t)}}{\big(y+e^{\sqrt{2}\gamma_i(t)}\big)
  \big(y+e^{\sqrt{2}\gamma_{i+1}(t)}\big)}\,{\mathrm d}y
  \\[2mm]&\nonumber\qquad-6\sqrt{2}e^{\sqrt{2}x}\int_{\frac{\gamma_i(t)+\gamma_{i+1}(t)}{2}}^x\frac{e^{\sqrt{2}\gamma_i(t)}}{\big(e^{\sqrt{2}\tau}+e^{\sqrt{2}\gamma_i(t)}\big)
  \big(e^{\sqrt{2}\tau}+e^{\sqrt{2}\gamma_{i+1}(t)}\big)}d\tau
  \\[2mm]&\nonumber=6e^{-\sqrt{2}x}\int_{0}^{e^{\sqrt{2}x}}\frac{ye^{\sqrt{2}\gamma_i(t)}}{\big(y+e^{\sqrt{2}\gamma_i(t)}\big)
  \big(y+e^{\sqrt{2}\gamma_{i+1}(t)}\big)}\,{\mathrm d}y
  \\[2mm]&\nonumber\qquad+6e^{\sqrt{2}\big[x-\gamma_i(t)-\gamma_{i+1}(t)\big]}\int_{0}^{e^{\sqrt{2}[\gamma_i(t)+\gamma_{i+1}(t)-x]}}\frac{ye^{\sqrt{2}\gamma_i(t)}}{\big(y+e^{\sqrt{2}\gamma_i(t)}\big)
  \big(y+e^{\sqrt{2}\gamma_{i+1}(t)}\big)}\,{\mathrm d}y
  \\[2mm]&\nonumber=\frac{6}{e^{\sqrt{2}[\gamma_{i+1}(t)-\gamma_i(t)]}-1}\bigg\{e^{-\sqrt{2}\{x-\gamma_{i+1}(t)\}}\ln\left(1+e^{\sqrt{2}\{x-\gamma_{i+1}(t)\}}\right)
  \\[2mm]&\nonumber\qquad-e^{-\sqrt{2}\{x-\gamma_i(t)\}}\ln\left(1+e^{\sqrt{2}\{x-\gamma_i(t)\}}\right)+e^{\sqrt{2}\{x-\gamma_i(t)\}}\ln\left(1+e^{-\sqrt{2}\{x-\gamma_i(t)\}}\right)
  \\[2mm]&\qquad-e^{\sqrt{2}\{x-\gamma_{i+1}(t)\}}\ln\left(1+e^{-\sqrt{2}\{x-\gamma_{i+1}(t)\}}\right)\bigg\}.
\end{align}
Therefore according to the expressions in the above, we can easily get the desired results in Lemma \ref{esxi}.
\qed

\subsection{The proof of  Lemma \ref{lem33}}\label{section6.2}
We firstly study the asymptotic expansion of the functions $d_i(t)$. Using the definition of $d_i(t)$ in \eqref{ddt} and the integration by parts, there holds that
\begin{align*}
&d_i(t)=\frac{\int_{\mathbb R} \Big[\widehat{\textrm{I}}_{1,i}\big(t,x+\gamma_i(t)\big)+\widehat{\textrm{I}}_{2,i}\big(t,x+\gamma_i(t)\big)\Big]\omega'(x)\mathrm{d}x}{\int_{\mathbb R}\big(\omega'(x)\big)^2\mathrm{d}x}
\\[2mm]
&=\frac{9\sqrt{2}}{4}\bigg\{\int_{\mathbb R}\left[\big(1-\omega(x)\big)^2\big(1+\omega\big(x+\gamma_i(t)-\gamma_{i+1}(t)\big)\big)-\big(1-\omega^2(x)\big)\xi_i\big(t,x+\gamma_i(t)\big)\right]\omega'(x)\,{\mathrm d}x
\\[2mm]
&\quad+\int_{\mathbb R}\left[\big(1+\omega(x)\big)^2\big(\omega\big(x+\gamma_i(t)-\gamma_{i-1}(t)\big)-1\big)+\big(1-\omega^2(x)\big)\xi_{i-1}\big(t,x+\gamma_i(t)\big)\right]\omega'(x)\,{\mathrm d}x\bigg\}
\\[2mm]
&=\frac{9\sqrt{2}}{4}\bigg\{\frac{1}{3}\int_{\mathbb R}\omega'\big(x+\gamma_i(t)-\gamma_{i+1}(t)\big)\left[1-\omega(x)\right]^3\,{\mathrm d}x-\int_{\mathbb R}\left[1-\omega^2(x)\right]\xi_i\big(t,x+\gamma_i(t)\big)\omega'(x)\,{\mathrm d}x
\\[2mm]
&\quad-\frac{1}{3}\int_{\mathbb R}\omega'\big(x+\gamma_i(t)-\gamma_{i-1}(t)\big)\left[1+\omega(x)\right]^3\,{\mathrm d}x+\int_{\mathbb R}\left[1-\omega^2(x)\right]\xi_{i-1}\big(t,x+\gamma_i(t)\big)\omega'(x)\,{\mathrm d}x\bigg\}.
\end{align*}

Next we will estimate these terms in the last equality. Let $\eta_{i+1}(t):=\gamma_{i+1}-\gamma_i(t)$ with $i=0,\cdots,k$.
According to the definition of $\omega(x)$ in \eqref{domega} and integration by parts, we have that
\begin{align*}
  &\int_{\mathbb R}\omega'\big(x+\gamma_i(t)-\gamma_{i+1}(t)\big)\left[1-\omega(x)\right]^3\,{\mathrm d}x
\\[2mm]&=16e^{\sqrt{2}\eta_{i+1}(t)}\int_0^\infty\frac{1}{\big(1+y\big)^3\Big(e^{\sqrt{2}\eta_{i+1}(t)}+y\Big)^2}\,{\mathrm d}y
\\[2mm]&=8e^{-\sqrt{2}\eta_{i+1}(t)}-16e^{\sqrt{2}\eta_{i+1}(t)}\int_0^\infty\frac{1}{\big(1+y\big)^2\Big(e^{\sqrt{2}\eta_{i+1}(t)}+y\Big)^3}\,{\mathrm d}y
\\[2mm]&=8e^{-\sqrt{2}\eta_{i+1}(t)}+O\left(e^{-2\sqrt{2}\eta_{i+1}(t)}\right),
\end{align*}
and
\begin{align*}
  &\int_{\mathbb R}\omega'\big(x+\gamma_i(t)-\gamma_{i-1}(t)\big)\left[1+\omega(x)\right]^3\,{\mathrm d}x
\\[2mm]&=16e^{\sqrt{2}\eta_i(t)}\int_0^\infty\frac{1}{\big(1+y\big)^3\big(e^{\sqrt{2}\eta_i(t)}+y\big)^2}\,{\mathrm d}y
\\[2mm]&=8e^{-\sqrt{2}\eta_i(t)}-16e^{\sqrt{2}\eta_i(t)}\int_0^\infty\frac{1}{\big(1+y\big)^2\big(e^{\sqrt{2}\eta_i(t)}+y\big)^3}\,{\mathrm d}y
\\[2mm]&=8e^{-\sqrt{2}\eta_i(t)}+O\left(e^{-2\sqrt{2}\eta_i(t)}\right).
\end{align*}

Furthermore, by the definition of $\xi_i(t,x)$ in \eqref{dxi}, we get that
\begin{align*}
   &\int_{\mathbb R}\xi_i\big(t,x+\gamma_i(t)\big)\left[1-\omega^2(x)\right]\omega'(x)\,{\mathrm d}x
   \\[2mm]&=\frac{48}{e^{\sqrt{2}\eta_{i+1}(t)}-1}
   \int_0^\infty\frac{y}{\big(1+y\big)^4}\left[\frac{e^{\sqrt{2}\eta_{i+1}(t)}}{y}\ln\left(1+\frac{y}{e^{\sqrt{2}\eta_{i+1}(t)}}\right)-\frac{1}{y}\ln\left(1+y\right)\right]\,{\mathrm d}y
\\[2mm]&\quad+\frac{48}{e^{\sqrt{2}\eta_{i+1}(t)}-1}
   \int_0^\infty\frac{y}{\big(1+y\big)^4}\left[\frac{1}{y}\ln\left(1+y\right)-\frac{1}{e^{\sqrt{2}\eta_{i+1}(t)}y}\ln\left(1+e^{\sqrt{2}\eta_{i+1}(t)}y\right)\right]\,{\mathrm d}y
\\[2mm]&=\frac{48}{e^{\sqrt{2}\eta_{i+1}(t)}-1}
   \int_0^\infty\frac{e^{\sqrt{2}\eta_{i+1}(t)}}{\big(1+y\big)^4}\ln\left(1+\frac{y}{e^{\sqrt{2}\eta_{i+1}(t)}}\right)\,{\mathrm d}y+O\left(\eta_{i+1}(t)e^{-2\sqrt{2}\eta_{i+1}(t)}\right)
   \\[2mm]&=8e^{-\sqrt{2}\eta_{i+1}(t)}+O\left(\eta_{i+1}(t)e^{-2\sqrt{2}\eta_{i+1}(t)}\right),
\end{align*}
and
\begin{align*}
   &\int_{\mathbb R}\xi_{i-1}\big(t,x+\gamma_i(t)\big)\left[1-\omega^2(x)\right]\omega'(x)\,{\mathrm d}x
   \\[2mm]&=\frac{48}{e^{\sqrt{2}\eta_{i}(t)}-1}
   \int_0^\infty\frac{y}{\big(1+y\big)^4}\left[\frac{1}{y}\ln\left(1+y\right)-\frac{e^{-\sqrt{2}\eta_{i}(t)}}{y}\ln\left(1+ye^{\sqrt{2}\eta_{i}(t)}\right)\right]\,{\mathrm d}y
\\[2mm]&\quad+\frac{48}{e^{\sqrt{2}\eta_{i}(t)}-1}
   \int_0^\infty\frac{y}{\big(1+y\big)^4}\left[\frac{e^{\sqrt{2}\eta_{i}(t)}}{y}\ln\left(1+ye^{-\sqrt{2}\eta_{i}(t)}\right)-\frac{1}{y}\ln\left(1+y\right)\right]\,{\mathrm d}y
\\[2mm]&=\frac{48}{e^{\sqrt{2}\eta_{i}(t)}-1}
   \int_0^\infty\frac{y}{\big(1+y\big)^4}\left[\frac{e^{\sqrt{2}\eta_{i}(t)}}{y}\ln\left(1+\frac{y}{e^{\sqrt{2}\eta_{i-1}(t)}}\right)
   -\frac{1}{ye^{\sqrt{2}\eta_{i}(t)}}\ln\left(1+ye^{\sqrt{2}\eta_{i}(t)}\right)\right]\,{\mathrm d}y
   \\[2mm]&=8e^{-\sqrt{2}\eta_{i}(t)}+O\left(\eta_{i}(t)e^{-2\sqrt{2}\eta_{i}(t)}\right).
\end{align*}

Thus combining the above arguments, we have that

\begin{equation*}
  d_i(t)=12\sqrt{2}\left\{e^{-\sqrt{2}\eta_i(t)}-e^{-\sqrt{2}\eta_{i+1}(t)}\right\}+O\left(\eta_{i}(t)e^{-2\sqrt{2}\eta_{i}(t)}\right)
  +O\left(\eta_{i+1}(t)e^{-2\sqrt{2}\eta_{i+1}(t)}\right).
\end{equation*}
Thus we obtain the expansions of $d_i(t)$ in Lemma \ref{lem33}.

In the next, we study the function $\widetilde{\xi}_i(t,x)$ defined in \eqref{dxi1} with $i=1,\cdots,k$. There exist the following two cases:

\textbf{Case $1$}: $x\geq0$. By \eqref{dxi1}, we have that
\begin{align}\label{esxi+}
 \nonumber &\widetilde{\xi}_i(t,x)=-\omega'(x)\int_0^x\big(\omega'(s)\big)^{-2}\,{\mathrm d}s\int^{+\infty}_s\widetilde{\textrm{I}}_i(t,\tau)\omega'(\tau)\mathrm{d}\tau
   \\[2mm] \nonumber
   &=-3\omega'(x)\int_0^x\big(\omega'(s)\big)^{-2}\int^{\infty}_s\Big[\big(1-\omega(\tau)\big)^2\big(1+\omega(\tau-\eta_{i+1}(t))\big)
   -\big(1-\omega^2(\tau)\big)\xi_i\big(t,\tau+\gamma_i(t)\big)
   \\[2mm] \nonumber
   &\qquad\qquad-\big(1+\omega(\tau)\big)^2\big(1-\omega(\tau+\eta_{i}(t))\big)
   +\big(1-\omega^2(\tau)\big)\xi_{i-1}\big(t,\tau+\gamma_i(t)\big)\Big]\omega'(\tau)\mathrm{d}\tau\mathrm{d}s
\\[2mm]
&\qquad+d_i(t)\omega'(x)\int_0^x\big(\omega'(s)\big)^{-2}\,{\mathrm d}s\int^{+\infty}_s\big[\omega'(\tau)\big]^2\mathrm{d}\tau
\end{align}
where we have used the fact that

\begin{equation*}
  \int_{\mathbb R} \widehat{\textrm{I}}_1(t,x)\omega'(x)\,{\mathrm d}x=0.
\end{equation*}

We next study the terms in \eqref{esxi+}. By the definition of $\omega(x)$ in \eqref{domega}, there hold that
\begin{align*}
  &\int^{+\infty}_s\big(1-\omega(\tau)\big)^2\big(1+\omega(\tau-\eta_{i+1}(t))\big)\omega'(\tau)\mathrm{d}\tau
  \\[2mm]&=16\sqrt{2}\int^{+\infty}_s\frac{e^{2\sqrt{2}\tau}}{\big(1+e^{\sqrt{2}\tau}\big)^4\big(e^{\sqrt{2}\eta_{i+1}(t)}+e^{\sqrt{2}\tau}\big)}\mathrm{d}\tau
  \\[2mm]&=16\int^{+\infty}_{e^{\sqrt{2}s}}\frac{y}{\big(1+y\big)^4\big(y+e^{\sqrt{2}\eta_{i+1}(t)}\big)}\,{\mathrm d}y
\\[2mm]&=O\left(e^{-2\sqrt{2}s}e^{-\sqrt{2}\eta_{i+1}(t)}\right),
\end{align*}
and
\begin{equation*}
\begin{aligned}
&\int^{+\infty}_s\big(1+\omega(\tau)\big)^2\big(1-\omega(\tau+\eta_{i}(t))\big)\omega'(\tau)d\tau
\\[2mm]&=16\sqrt{2}\int^{+\infty}_s\frac{e^{3\sqrt{2}\tau}}{\big(1+e^{\sqrt{2}\tau}\big)^4\big(1+e^{\sqrt{2}\{\tau+\eta_i(t)\}}\big)}d\tau
\\[2mm]&=16\int^\infty_{e^{\sqrt{2}s}}\frac{y^2}{\big(1+y\big)^4\big(1+ye^{\sqrt{2}\eta_{i}(t)}\big)}\,{\mathrm d}y
\\[2mm]&=O\Big(e^{-\sqrt{2}\eta_i(t)}e^{-2\sqrt{2}s}\Big).
\end{aligned}
\end{equation*}
Moreover, using the estimates of $\xi_i(t,x)$ in Lemma \eqref{esxi}, we get that
\begin{align*}
\int^{+\infty}_s\big(1-\omega^2(\tau)\big)\xi_i\big(t,\tau+\gamma_i(t)\big)\omega'(\tau)\mathrm{d}\tau
&=8\sqrt{2}\int_s^{+\infty}\frac{e^{2\sqrt{2}\tau}}{\left(1+e^{\sqrt{2}\tau}\right)^4}\xi_i\big(t,\tau+\gamma_i(t)\big)\mathrm{d}\tau
\\&=O\left(e^{-\sqrt{2}\eta_i(t)}e^{-2\sqrt{2}s}\right),
\end{align*}
and
\begin{align*}
\int^{+\infty}_s\big(1-\omega^2(\tau)\big)\xi_{i-1}\big(t,\tau+\gamma_i(t)\big)\omega'(\tau)\mathrm{d}\tau
&=8\sqrt{2}\int_s^{+\infty}\frac{e^{2\sqrt{2}\tau}}{\left(1+e^{\sqrt{2}\tau}\right)^4}\xi_{i-1}\big(t,\tau+\gamma_i(t)\big)\mathrm{d}\tau
\\&=O\left(e^{-\sqrt{2}\eta_{i-1}(t)}e^{-2\sqrt{2}s}\right).
\end{align*}

Hence, combining the above arguments, the definition of $\omega(x)$ in\eqref{domega}, Lemma \ref{esxi}  and \ref{esxi+}, we get that
\begin{align}\label{exi+}
  \nonumber\abs{\widetilde{\xi}_i(t,x)}&\leq C\left[e^{-\sqrt{2}\eta_{i-1}(t)}+e^{-\sqrt{2}\eta_{i}(t)}+\abs{d_i(t)}\right]\omega'(x)\int_0^x\big(\omega'(s)\big)^{-2}e^{-2\sqrt{2}s}\,{\mathrm d}s
\\[2mm]&\leq C\left[e^{-\sqrt{2}\eta_{i-1}(t)}+e^{-\sqrt{2}\eta_{i}(t)}+\abs{d_i(t)}\right]e^{-\sqrt{2}x}x,
\end{align}
for all $x\geq0$, where $C>0$ is independent of $t$ and $x$.

\textbf{Case $2$}: $x\leq0$. By the definition of $\xi_i(t,x)$ in \eqref{dxi1} and $\omega(x)$ is an odd function, we have that
\begin{align*}
  \widetilde{\xi}_i(t,-x)&=\omega'(x)\int_0^{-x}\big(\omega'(s)\big)^{-2}\,{\mathrm d}s\int_{-\infty}^s\widetilde{\textrm{I}}_i(t,\tau)\omega'(\tau)\mathrm{d}\tau
\\[2mm]&=-\omega'(x)\int_0^{x}\big(\omega'(s)\big)^{-2}\,{\mathrm d}s\int_{-\infty}^{-s}\widetilde{\textrm{I}}_i(t,\tau)\omega'(\tau)\mathrm{d}\tau
\\[2mm]&=-\omega'(x)\int_0^{x}\big(\omega'(s)\big)^{-2}\,{\mathrm d}s\int^{+\infty}_{s}\widetilde{\textrm{I}}_i(t,-\tau)\omega'(\tau)\mathrm{d}\tau,
\end{align*}
where the function
\begin{align*}
  \widetilde{\textrm{I}}_i(t,-\tau)&=\big(1+\omega(\tau)\big)^2\big(1-\omega(\tau+\eta_{i+1}(t))\big)
   -\big(1-\omega^2(\tau)\big)\xi_i\big(t,-\tau+\gamma_i(t)\big)
   \\[2mm]
   &\quad-\big(1-\omega(\tau)\big)^2\big(1+\omega(\tau-\eta_{i}(t))\big)
   +\big(1-\omega^2(\tau)\big)\xi_{i-1}\big(t,-\tau+\gamma_i(t)\big)-d_i(t)\omega(\tau).
\end{align*}
Thus by similar arguments of the case of $x\geq0$, we can get that
\begin{equation}\label{exi-}
\abs{\widetilde{\xi}_i(t,x)}\leq C\left[e^{-\sqrt{2}\eta_{i-1}(t)}+e^{-\sqrt{2}\eta_{i}(t)}+\abs{d_i(t)}\right]e^{\sqrt{2}x}\abs{x},
\end{equation}
for all $x\leq0$, where $C>0$ is independent of $t$ and $x$.

Combining \eqref{exi+} and \eqref{exi-}, we can get that the estimate of $\widetilde{\xi}_i(t,x)$ in Lemma \ref{lem33}.
Moreover, according to the following facts that
\begin{equation*}
  \partial_x\widetilde{\xi}_i(t,x)=-\omega''(x)\int_0^x\big(\omega'(s)\big)^{-2}\,{\mathrm d}s\int^{+\infty}_s\widetilde{\textrm{I}}_i(t,\tau)\omega'(\tau)\mathrm{d}\tau
-\big(\omega'(x)\big)^{-1}\int^{+\infty}_x\widetilde{\textrm{I}}_i(t,\tau)\omega'(\tau)\mathrm{d}\tau
\end{equation*}
and
\begin{equation*}
  \partial_t\widetilde{\xi}_i(t,x)=-\omega'(x)\int_0^x\big(\omega'(s)\big)^{-2}\,{\mathrm d}s\int^{+\infty}_s\partial_t\widetilde{\textrm{I}}_i(t,\tau)\omega'(\tau)\mathrm{d}\tau,
\end{equation*}
where $\widetilde{\textrm{I}}_i(t,\tau)$ is defined by \eqref{dIi}, similar arguments in above \textbf{Case} $1$ imply that estimates of $\partial_x\widetilde{\xi}_i(t,x)$ and $\partial_x\widetilde{\xi}_i(t,x)$ in Lemma \ref{lem33}. The proof is completed.
\qed

\bigskip
\bigskip
\noindent {\bf Data Availability Statements:} Not Applicable.

\smallskip
\noindent {\bf Conflict of interest Statements:} The authors declared that they have no conflicts of interest to this work.



\end{document}